\documentclass[]{amsart}
\usepackage{mathrsfs}
\usepackage{color}
\usepackage{bm}
\usepackage{mathrsfs}
\usepackage{amsmath}
\usepackage{amsfonts}
\usepackage{mathtools}
\usepackage{amssymb}
\usepackage{color}
\usepackage{graphicx}
\usepackage{hyperref}
\usepackage{cleveref}
\usepackage{enumitem}
\setlist[enumerate,1]{label=\textup{(\arabic*)}}

 \newtheorem{Theorem}{Theorem}[section]
 \newtheorem{Corollary}[Theorem]{Corollary}
 \newtheorem{Lemma}[Theorem]{Lemma}
 \newtheorem{Proposition}[Theorem]{Proposition}

 \newtheorem{Definition}[Theorem]{Definition}
\newtheorem{Question}[Theorem]{Question}

 \newtheorem{Remark}[Theorem]{Remark}

 \newtheorem{Example}[Theorem]{Example}
 \numberwithin{equation}{section}

\DeclareMathOperator{\Tn}{Tn}

\DeclareMathOperator{\ord}{ord}
\DeclareMathOperator{\sing}{sing}
\DeclareMathOperator{\reg}{reg}

\def\i{\mathrm{i}}

\begin{document}
\title[The existence of valuative interpolation at a singular point]
{The existence of valuative interpolation at a singular point}

\author{Shijie Bao}
\address{Shijie Bao: Academy of Mathematics
	and Systems Science, Chinese Academy of Sciences, Beijing 100190, China.}
\email{bsjie@amss.ac.cn}

\author{Qi'an Guan}
\address{Qi'an Guan: School of Mathematical Sciences,
	Peking University, Beijing, 100871, China.}
\email{guanqian@math.pku.edu.cn}

\author{Zhitong Mi}
\address{Zhitong Mi: School of Mathematics and Statistics, Beijing Jiaotong University, Beijing,
	100044, China.
}
\email{zhitongmi@amss.ac.cn}

\author{Zheng Yuan}
\address{Zheng Yuan: State Key Laboratory of Mathematical Sciences, Academy of Mathematics and Systems Science, Chinese Academy of Sciences, Beijing 100190, China.}
\email{yuanzheng@amss.ac.cn}

\thanks{}



\date{\today}

\dedicatory{}

\commby{}



\begin{abstract}
	The present paper studies the  existence of valuative interpolation on the local ring of an irreducible analytic subvariety at singular points. We firstly develop the concepts and methods of Zhou weights and Tian functions near singular points of  irreducible analytic subvarieties. By applying these tools, we  establish the necessary and sufficient conditions for the existence of valuative interpolations on the rings of germs of holomorphic functions and weakly holomorphic functions at a singular point.

	As applications, we characterize  the existence of valuative interpolations on the quotient ring of the ring of convergent power series in real  variables.  We also present separated necessary and sufficient conditions  for the existence of valuative interpolations on the quotient ring of polynomial rings with complex  coefficients and real coefficients. Furthermore, we show that the conditions become both necessary and sufficient under certain conditions on the zero set of the given polynomials.

\end{abstract}

\maketitle

\tableofcontents

\section{Introduction}

Valuation theory is deeply connected  to the  theory of singularities in several complex variables and complex algebraic geometry (see e.g. \cite{FJ04,FJ05,FJ05b,BFJ08,JON-Mus2012,BGMY-valuation}). 
Let $X$ be an analytic subvariety and $x$ a point of $X$.
Denote the ring of germs of holomorphic functions at $x$ in $X$ by $\mathcal{O}_{X,x}$, and $\mathcal{O}_{X,x}^*\coloneqq \mathcal{O}_{X,x}\setminus\{0\}$.
Recall that
a \emph{valuation} on $\mathcal{O}_{X,x}$ is a nonconstant map $\nu\colon \mathcal{O}_{X,x}^* \rightarrow\mathbb{R}_{\ge0}$ satisfying the following:
\begin{enumerate}
    \item $\nu(f g)=\nu(f)+\nu(g)$;
    \item $\nu(f+g)\ge\min\{\nu(f),\nu(g)\}$;
    \item $\nu(c)=0$, where $c\not=0$ is a constant function.
\end{enumerate}
 Similarly, we can define the valuations on the polynomial ring $\mathbb{C}[z_1,\ldots,z_n]/I$, where $I$ is a proper ideal of $\mathbb{C}[z_1,\ldots,z_n]$.

A natural problem is how to characterize the existence of valuative interpolation:

\begin{Question}
	\label{q:1}
Given any positive integer $m$, a finite set of elements $\{f_j\}_{1\le j\le m}$ in $\mathcal{O}_{X,x}$ (or  $\mathbb{C}[z_1,\ldots,z_n]/I$), and any finite nonnegative real numbers $\{a_j\}_{1\le j\le m}$, can one find necessary and sufficient conditions for the existence of the valuation $\nu$ on $\mathcal{O}_{X,x}$ (or on $\mathbb{C}[z_1,\ldots,z_n]/I$) such that $\nu(f_j)=a_j$ for all $j\in\{1,\ldots, m\}$?
\end{Question}

When  $x$ is a regular point, we \cite{BGMY-interpolation25A} established the  necessary and sufficient conditions for the existence of valuative interpolations on $O_{X,x}$. The proof in \cite{BGMY-interpolation25A} relies heavily on the properties of Zhou weights and Tian functions, where the  concepts and properties of Zhou weights and Tian functions in the smooth case were  established by the authors in \cite{BGMY-valuation}. As applications, characterizations of the existence of valuative interpolation in the polynomial ring case and the real case were also established in  \cite{BGMY-interpolation25A}. 

In the present paper, we consider the valuative interpolation problem when $x$ is a singular point of an irreducible analytic subvariety $X$. To do this, we generalize the concepts and properties of Zhou weights, Zhou valuations and Tian functions from the smooth case (presented in \cite{BGMY-valuation}) to the singular case. The properties of jumping numbers and relative types in the singular case are also studied. Based on these preparations, we  present  necessary and sufficient conditions for the existence of valuative interpolations on  rings of germs of holomorphic functions and weakly holomorphic functions at a singular point $x$ of an irreducible analytic subvariety $X$. As an application, we also give the  necessary and sufficient conditions separately for the existence of valuative interpolations on the quotient ring of polynomial rings with complex  coefficients and we show that the conditions become both necessary and sufficient when the intersection of the zero sets of the given polynomials is the origin $o$ contained in $X$. As a further application, we derive necessary and sufficient conditions for the valuative interpolation problem in the real case.

\subsection{Zhou weights and Zhou valuations near singular points}
\label{sec:1.1}

Let $X$ be an analytic subset of $\Omega\in\mathbb{C}^{n}$ with pure dimension $d$. Let 
\[dV_{X}\coloneqq \frac{1}{2^{d}d!}\bigwedge^d (\sum_{1\le j\le n}\sqrt{-1}dz_j\wedge d\bar z_j)\]
be the volume form on $X$ induced by the standard volume form on $\mathbb{C}^{n}$.

We call a Lebesgue function $\varphi$ on an analytic set $X$ a \emph{plurisubharmonic} function, if $\varphi$ is plurisubharmonic on $X_{\reg}$ and is bounded above near any $z\in X_{\sing}$. In this article, when we consider the values of plurisubharmonic functions on analytic sets, we ignore their values on singular points.

Let $(X,z_0)$ be an irreducible germ of an analytic set. 
Let $f_{0}=(f_{0,1},\cdots,f_{0,m})$ be a vector,
where $f_{0,1},\cdots,f_{0,m}$ are holomorphic functions near $z_0$.
Denote by $|f_{0}|^{2}=|f_{0,1}|^{2}+\cdots+|f_{0,m}|^{2}$.
Let $\varphi_{0}$ be a plurisubharmonic function near $z_0$,
such that $|f_{0}|^{2}e^{-2\varphi_{0}}$ is integrable near $z_0$.

\begin{Definition}
	\label{def:max_relat}
	We call a plurisubharmonic function $\Phi^{f_0,\varphi_0}_{z_0,\max}$ ($\Phi_{z_0,\max}$ for short) near $z_0$ on $X$ a \textbf{local Zhou weight related to $|f_{0}|^{2}e^{-2\varphi_{0}}$ near $z_0$},
	if the following three statements hold
	
	$(1)$ $|f_{0}|^{2}e^{-2\varphi_{0}}|z|^{2N_{0}}e^{-2\Phi_{z_0,\max}}$ is integrable near $z_0$
	for large enough $N_{0}\gg0$, where $|z|^2\coloneqq \sum_{1\le j\le n}|z_j|^2$;
	
	$(2)$ $|f_{0}|^{2}e^{-2\varphi_{0}}e^{-2\Phi_{z_0,\max}}$
	is not integrable near $z_0$;
	
	$(3)$ for any plurisubharmonic function $\varphi'\geq\Phi_{z_0,\max}+O(1)$ near $z_0$
	such that $|f_{0}|^{2}e^{-2\varphi_{0}}e^{-2\varphi'}$
	is not integrable near $z_0$,
	$\varphi'=\Phi_{z_0,\max}+O(1)$ holds.
\end{Definition}

Let $\varphi$ be a plurisubharmonic function near $z_0$. The existence of local Zhou weights follows from the strong openness property of multiplier ideal sheaves \cite{GZopen-c}.

\begin{Remark}
	\label{rem:max_existence}
	Assume that $|f_{0}|^{2}e^{-2\varphi_{0}}|z|^{2N_{0}}e^{-2\varphi}$ is integrable near $o$
	for large enough $N_{0}\gg0$,
	and $|f_{0}|^{2}e^{-2\varphi_{0}-2\varphi}$
	is not integrable near $z_0$.
	
	Then there exists a local Zhou weight $\Phi_{z_0,\max}$ related to $|f_{0}|^{2}e^{-2\varphi_{0}}$ near $z_0$
	such that $\Phi_{z_0,\max}\geq\varphi$.
	
	Moreover, for any local Zhou weight $\Phi_{z_0,\max}$, $\Phi_{z_0,\max}\geq N\log|z|+O(1)$ near $z_0$ holds for some $N\gg0$.
\end{Remark}
We recall the definition of \emph{weakly holomorphic functions} on a complex space.
\begin{Definition}[{see \cite[Definition 7.1]{demailly-book}}]
   Let $X$ be a complex space. A weakly holomorphic function $f$ on $X$ is a holomorphic function on $X_{\mathrm{\reg}}$ such that any point of $X_{\mathrm{\sing}}$ has a neighborhood $V$ for which $f$ is bounded on $X_{\mathrm{\reg}} \cap V$. We denote by  $\mathcal{O}^{w}_{X,x}$ the ring of germs of weakly holomorphic functions on neighborhoods of $x$ and $\mathcal{O}^{w}_X$ the associated sheaf.
\end{Definition}

When $X$ is normal at some point $x$, we know that $\mathcal{O}^{w}_{X,x}=\mathcal{O}_{X,x}$. In general, we only have $\mathcal{O}_{X,x}\subset\mathcal{O}^{w}_{X,x}$.

When $f_{0,i}$ ($i=1,\ldots,m$) are weakly holomorphic functions near $z_0$, we can also define the local Zhou weight related to $|f_0|^2e^{-2\varphi_0}$.
\begin{Remark}
 \label{rem:zhou weight for weak holomorphic function}
Let $f_{0,i}$ ($i=1,\ldots,m$) be weakly holomorphic functions near $z_0$ and $\varphi_{0}$ be a plurisubharmonic function near $z_0$,
such that $|f_{0}|^{2}e^{-2\varphi_{0}}$ is integrable near $z_0$. It follows from Theorem \ref{th:universal denominators} that there exists a
holomorphic function $\delta$ near $o$ such that $g_{0,i}\coloneqq \delta f_{0,i}$ ($i=1,\ldots,m$) is holomorphic near $z_0$. Then we define that a plurisubharmonic function $\Phi$ near $o$ is a local Zhou weight related  to $|f_0|^2e^{-2\varphi_0}$ if and only if $\Phi$ is a local Zhou weight related to $|g_0|^2e^{-2\log|\delta|-2\varphi_0}$.

Thus, by Remark \ref{rem:max_existence}, the local Zhou weights exists for weakly holomorphic functions.
\end{Remark}

When $z_0$ is a smooth point of $X$, to study the singularity of the plurisubharmonic functions near $z_0$,
Rashkovskii \cite{Rash06} introduced the concept of \textbf{relative type}
$$\sigma(\psi,\varphi)\coloneqq\sup\{c\ge0 \colon \psi\le c\varphi+O(1)\ \text{near\ } o\},$$
where $\psi$ is a  plurisubharmonic function near  $z_0$ and $\varphi$ is a maximal weights with an isolated singularity at $z_0$. The notation of relative type generalizes the classical Lelong number \cite{Lelong} and Kiselman number \cite{Kiselman}.

When $z_0$ is a singular point of $X$, we can generalize the definition of relative type to the singular case.
Let  $\psi,\varphi$ be any plurisubharmonic functions defined near
$z_0$. Define the \textbf{relative type} by
   \begin{equation*}
        \begin{split}
            \sigma(\psi,\varphi)\coloneqq \sup\{c\ge0 \colon &\psi\le c\varphi+O(1)\ \text{holds} \ \text{on}\  U\backslash X_{\sing}\\
            & \text{for some open neighborhood}\  U\  \text{of}\ z_0 \ \text{in} \ X \}.
        \end{split}
    \end{equation*}

When the $\varphi=\Phi_{z_0,\max}$ is a local Zhou weight,
we call the relative type $\sigma(\cdot,\Phi_{z_0,\max})$ the \textbf{Zhou number}.

Note that for any $b<\sigma(\psi,\Phi_{z_0,\max})$,
$|f_{0}|^{2}e^{-2\varphi_{0}}e^{-2\max\{\Phi_{z_0,\max},\frac{1}{b}\psi\}}$ is not integrable near $z_0$.
Then it follows from the \emph{strong openness property} of multiplier ideal sheaves (Theorem \ref{thm:SOC})
that $|f_{0}|^{2}e^{-2\varphi_{0}}e^{-2\max\big\{\Phi_{z_0,\max},\frac{1}{\sigma(\psi,\Phi_{z_0,\max})}\psi\big\}}$
is not integrable near $z_0$.
Note that
$$\max\left\{\Phi_{z_0,\max},\frac{1}{\sigma(\psi,\Phi_{z_0,\max})}\psi\right\}\geq\Phi_{z_0,\max}.$$
Then
$$\max\left\{\Phi_{z_0,\max},\frac{1}{\sigma(\psi,\Phi_{z_0,\max})}\psi\right\}=\Phi_{z_0,\max}+O(1),$$
which implies that
$$\Phi_{z_0,\max}\geq\frac{1}{\sigma(\psi,\Phi_{z_0,\max})}\psi+O(1),$$
i.e.,
\begin{equation*}
	\psi\leq \sigma(\psi,\Phi_{z_0,\max})\Phi_{z_0,\max}+O(1).
\end{equation*}

For the Lelong numbers and the Kiselman numbers, there are expressions in integral
form (see \cite{demailly-book}). Now,  we show that Zhou numbers also have  expressions in integral form.
\begin{Theorem}
	\label{thm:expression of relative types}
	Let $\Phi_{z_0,\max}$ be a local Zhou weight related to $|f_{0}|^{2}e^{-2\varphi_{0}}$ near $z_0$.
	Then for any plurisubharmonic function $\psi$ near $z_0$ satisfying $\psi\le c\log|z|+O(1)$ near $z_0$ for some $c>0$, we have
	\begin{equation}\nonumber
		\begin{split}
			\sigma(\psi,\Phi_{z_0,\max})
			&=
			\lim_{t\to+\infty}
			\frac{\int_{\{\Phi_{z_0,\max}<-t\}}|f_{0}|^{2}e^{-2\varphi_{0}}(-\psi)}
			{t\int_{\{\Phi_{z_0,\max}<-t\}}|f_{0}|^{2}e^{-2\varphi_{0}}}.
		\end{split}
	\end{equation}
\end{Theorem}

Denote by
$$\nu(f,\Phi_{z_0,\max})\coloneqq \sigma(\log|f|,\Phi_{z_0,\max})$$
for any $(f,z_0)\in\mathcal{O}_{X,z_0}$.  By Theorem \ref{thm:expression of relative types} and the definition of $\nu(\cdot,\Phi_{z_0,\max})$, we see that $\nu(\cdot,\Phi_{z_0,\max})$ is a \emph{valuation} of $\mathcal{O}_{X,z_0}$ for any local Zhou weight $\Phi_{z_0,\max}$, and we call it \textbf{Zhou valuation}.
\begin{Corollary}
	\label{coro:valuation}For any local Zhou weight $\Phi_{z_0,\max}$ near $o$, $\nu(\cdot,\Phi_{z_0,\max})\colon\mathcal{O}_{X,z_0}\to\mathbb{R}_{\ge0}$ satisfies the following:
	\begin{enumerate}
	    \item $\nu(fg,\Phi_{z_0,\max})=\nu(f,\Phi_{z_0,\max})+\nu(g,\Phi_{z_0,\max})$;
	    \item $\nu(f+g,\Phi_{z_0,\max})\ge\min\big\{\nu(f,\Phi_{z_0,\max}),\nu(g,\Phi_{z_0,\max})\big\}$;
	    \item $\nu(f,\Phi_{z_0,\max})=0$ if and only if $f(z_0)\not=0$.
	\end{enumerate}
\end{Corollary}

Let $G$ be a holomorphic function near $z_0$. We recall the definition of jumping number (see \cite{JON-Mus2012,JON-Mus2014})     
$$c^G_{z_0}(\Phi_{z_0,\max})\coloneqq \sup\big\{c:|G|^{2}e^{-2c\Phi_{z_0,\max}}\text{ is integrable near }z_0\big\}.$$ When $G=1$,  the jumping number was also called complex singularity exponent and denote $c_{z_0}(\Phi_{z_0,\max})\coloneqq c^1_{z_0}(\Phi_{z_0,\max})$ (see \cite{tian87,demailly2010}).

\begin{Theorem}
	\label{thm:valu-jump}Let $\Phi_{z_0,\max}$  be a local Zhou weight near $z_0$.
	For any holomorphic function $G$ near $z_0$,	
	we have the following relation between  $c^G_{z_0}(\Phi_{z_0,\max})$ and  $\sigma(G,\Phi_{z_0,\max})$,
	\begin{equation*}
		\begin{split}
			\nu(G,\Phi_{z_0,\max}) + c_{z_0}(\Phi_{z_0,\max})
			&\le c^G_{z_0}(\Phi_{z_0,\max}) \\
			&\le  \nu(G,\Phi_{z_0,\max})-\sigma(\log|f_0|,\Phi_{z_0,\max})+1+\sigma(\varphi_0,\Phi_{z_0,\max}).
		\end{split}
	\end{equation*}
	Especially, if $|f_0|^2e^{-2\varphi_0}\equiv1$, we have
	$$\nu(G,\Phi_{z_0,\max})+1=c^G_{z_0}(\Phi_{z_0,\max}).$$
\end{Theorem}

\subsection{Interpolation problem}
\label{sec:1.2}
Let $X$ be an analytic variety with pure dimension $d$ contained in $\mathbb{C}^{n}$.  Assume that the origin $o\in X_{\text{sing}}$ and $(X,o)$ is an irreducible germ of an analytic set. The following theorem presents a criterion for the valuative interpolation problem on $\mathcal{O}_{X,o}$.

\begin{Theorem}\label{thm:interpolation}
Let $f_j\in \mathcal{O}^*_{X,o}$ for $0\le j\le m$. Let $a_0=0$ and $\{a_j\}_{1\le j\le m}$ be positive numbers. The following two statements are equivalent:
	\begin{enumerate}
	    \item There exists a valuation $\nu$ on $\mathcal{O}_{X,o}$  such that $\nu(f_j)=a_j$ for all $0\le j\le m$;
	    \item $\sigma(\log|F|,\varphi)=\sum_{1\le j\le m}a_j$, where \[F\coloneqq \prod_{0\le j\le m}f_j, \quad \varphi\coloneqq \log\Big(\sum_{1\le j\le m}|f_j|^{\frac{1}{a_j}}\Big).\]
	\end{enumerate}
\end{Theorem}

Note that $\nu(f_1f_2)=0$ if and only if $ \nu(f_1)=0\  \text{and}\  \nu(f_2)=0$. Thus, in the above theorem, it suffices to consider the case  only one function $f_0$ satisfies $\nu(f_0)=0.$ As $\nu(1)=0$ holds for any valuation $\nu$, in the above theorem, we can actually ignore $f_0$ when we take $f_0\equiv1$.

 For  valuations on the germs of weakly holomorphic functions $\mathcal{O}^{w}_{X,o}$, we also have the following interpolation result.
\begin{Theorem}\label{thm:interpolation weakly}
Let $\{f_j\}_{0\le j\le m}$ be weakly holomorphic functions near $o$. Let $a_0=0$ and $\{a_j\}_{1\le j\le m}$ be positive numbers. The following two statements are equivalent:
	\begin{enumerate}
	    \item There exists a valuation $\nu$ on $\mathcal{O}^{w}_{X,o}$  such that $\nu(f_j)=a_j$ for all $0\le j\le m$;
	    \item $\sigma(\log|F|,\varphi)=\sum_{1\le j\le m}a_j$, where \[F\coloneqq \prod_{0\le j\le m}f_j, \quad \varphi\coloneqq \log\Big(\sum_{1\le j\le m}|f_j|^{\frac{1}{a_j}}\Big).\]
	\end{enumerate}
	\end{Theorem}
	
	Let us consider the valuative interpolation problem on the polynomial ring $\mathbb{C}[z_1,\ldots,z_n]/I$, where $I$ is a prime ideal in $\mathbb{C}[z_1,\ldots,z_n]$. Let $X\coloneqq V(I)$ be  the affine variety defined by $I$ and $o\in X$. Denote the germ of the set $X$ at $o$ by $(X,o)$ and we assume that $(X,o)$ is irreducible as a germ of analytic set.
	
	\begin{Remark}
		It follows from the main theorem in \cite{Zariski48} that, when the affine variety $X$ is normal at $o$, $(X,o)$ is irreducible as a germ of analytic set. So there exists many affine varieties $X$ such that $(X,o)$ is irreducible as a germ of analytic set.
	\end{Remark}

	Let $\{f_j\}_{0\le j\le m}$ be polynomials in $\mathbb{C}[z_1,\ldots,z_n]/I$, $a_0=0$ and  $\{a_j\}_{1\le j\le m}$ be $m$ positive numbers. Denote $F\coloneqq \prod_{0\le j\le m}f_j$ and $\varphi\coloneqq \log(\sum_{1\le j\le m}|f_j|^{\frac{1}{a_j}})$.

\begin{Corollary}\label{C:interpolation-comp-poly}

 If $\sigma(\log|F|,\varphi)=\sum_{1\le j\le m}a_j$, then there exists a valuation $\nu$ on $\mathbb{C}[z_1,\ldots,z_n]/I$ such that $\nu(f_j)=a_j$ for every $1\le j\le m$.

Conversely, if there exists a valuation $\nu$ on $\mathbb{C}[z_1,\ldots,z_n]/I$ satisfying that $\nu(f_j)=a_j$ for every $1\le j\le m$ and $\nu(z_l)>0$ for every $1\le l\le n$, then we have $\sigma(\log|F|,\varphi)=\sum_{1\le j\le m}a_j$.
\end{Corollary}

When $\cap_{1\le j\le m}\{f_j=0\}=\{o\}$, \Cref{C:interpolation-comp-poly} gives a necessary and sufficient condition for the existence of valuative interpolations on $\mathbb{C}[z_1,\ldots,z_n]$.

\begin{Corollary}\label{C:interpolation-comp-poly2}
 If $\cap_{1\le j\le m}\{f_j=0\}=\{o\}$, then there exists a valuation $\nu$ on $\mathbb{C}[z_1,\ldots,z_n]/I$ such that $\nu(f_j)=a_j$ for every $j$ if and only if $\sigma(\log|F|,\varphi)=\sum_{1\le j\le m}a_j$.
\end{Corollary}

The following example (see \cite{BGMY-interpolation25A}) shows that the condition ``$\cap_{1\le j\le m}\{f_j=0\}=\{o\}$" can not be removed.
\begin{Example}Let $I=(z^3_1-z^2_2)$. Denote $X\coloneqq \mathbb{C}[z_1,z_2]/I$. Then $o$ and $(1,1)$ are points in $X$ where $o$ is a singular point and $(1,1)$ is a smooth point. Note that $z^3_1-z^2_2$ is irreducible in $\mathbb{C}[z_1,z_2]$, $\mathcal{O}_{\mathbb{C}^2,o}$ and $\mathcal{O}_{\mathbb{C}^2,(1,1)}$. We know that $I$ is a prime ideal in $\mathbb{C}[z_1,z_2]$ and $\big(X,o\big)$, $\big(X,(1,1)\big)$ is irreducible as a germ of analytic set. 

    Let $h_1=[z_1]$ $h_2=[z_2]$, $h_3=[z_1z_2]$, $g_1=[z_1-1]$ and $g_2=[z_2-1]$ on $X$. Take $\{f_1,f_2,f_3,f_4,f_5,f_6\}=\{h_1g_1,h_1g_2,h_2g_1,h_2g_2,h_3g_1,h_3g_2\}$ and $a_1=\ldots=a_6=1$.  It is clear that 
    $$\cap_{1\le l\le 6}\{f_l=0\}=\{o,(1,1)\}.$$

   Let  $t\in \mathbb{C}\backslash{\{0\}}$ and $z_1=t^2$ and $z_2=t^3$. Then $t$ is a coordinate on $X_{\text{reg}}$.
  Note that the point $(1,1)\in X$ corresponds to $t=1$. As $\log F=\log(\prod_{1\le l\le 6}f_l)=3\log(|z_1-1||z_2-1|)+O(1)$ and $\varphi=\log(\sum_{1\le l\le 6}|f_l|^{\frac{1}{a_l}})=\log(|z_1-1|+|z_2-1|)+O(1)$ near $(1,1)$, then we know that $\log|F|=6\log|t-1|+O(1)$ and $\varphi=\log|t-1|+O(1)$ near $t=1$.
   Then we have
    $$\sigma_{(1,1)}(\log|F|,\varphi)=6=\sum_{1\le l\le 6}a_l$$
    holds on $X$ near $(1,1)$.
    
    By \Cref{C:interpolation-comp-poly}, there exists a valuation $\nu$ on $\mathbb{C}[z_1,z_2]/I$ such that $\nu(f_l)=a_l$ for every $1\le l\le 6$.
    Since $F=4\log|z_1z_2|+O(1)$ near $o$ and $\varphi=\log(|z_1|+|z_2|+|z_1z_2|)+O(1)$ near $o$ in $\mathbb{C}^2$, we know that 
    $$\sup\{c\colon\log|F|\le c\varphi+O(1)\text{ near $o$ on $\mathbb{C}^2$}\}>6=\sum_{1\le l\le 6}a_l.$$
    Note that the relative type may increase after restriction. We have $$\sigma_{o}(\log|F|,\varphi)>6=\sum_{1\le l\le 6}a_l.$$
  Thus, the condition ``$\cap_{1\le j\le m}\{f_j=0\}=\{o\}$" in  \Cref{C:interpolation-comp-poly2} can not be removed.
\end{Example}

Denote the set of all germs of real analytic functions near the origin $o'\in\mathbb{R}^n$ by $C^{\text{an}}_{o'}$. There exists an injective ring homomorphism $P\colon C^{\text{an}}_{o'}\rightarrow\mathcal{O}_o$ that satisfies
$$P\Big(\sum_{{\alpha\in\mathbb{Z}_{\ge0}^n}}a_{\alpha}x^{\alpha}\Big)=\sum_{\alpha\in\mathbb{Z}_{\ge0}^n}a_{\alpha}z^{\alpha},$$
where $(x_1,\ldots,x_n)$ and $(z_1,\ldots,z_n)$ are the standard coordinates on $\mathbb{R}^n$ and $\mathbb{C}^n$ respectively, and $\sum_{\alpha\in\mathbb{Z}_{\ge0}^n}a_{\alpha}x^{\alpha}$ is the power series expansion of an arbitrarily given real analytic function near $o'$. It is clear that, for any $(h,o)\in \mathcal{O}_o$, there exists a unique pair of real analytic functions $(h_1, h_2)$ near $o'$ such that
\[h=P(h_1)+\i P(h_2) \quad \ \text{near} \ o.\]

Assume that $I$ is an ideal in $C^{\text{an}}_{o'}$. Denote by $\big(P(I)\big)$ the ideal generated by $P(I)$ in $\mathcal{O}_o$. 
Assume that $\big(P(I)\big)$ is prime ideal in $\mathcal{O}_o$. Note that $\big(P(I)\big)$ is a prime ideal in $\mathcal{O}_o$ implies $I$ is a prime ideal in $C^{\text{an}}_{o'}$ (see Remark \ref{rem:prime ideal from complex to real}).
Then we have an induced ring homomorphism  $\widetilde{P}\colon C^{\text{an}}_{o'}/I\rightarrow \mathcal{O}_o/\big(P(I)\big)$ defined by $\widetilde{P}([f])=[P(f)]$. It is easy to check that $\widetilde{P}$ is well defined.
Let $X$ be the zero variety defined by  $\big(P(I)\big)$.

\begin{Corollary}
	\label{C:interpolation-real}
		Let $f_j\in C^{\text{an}}_{o'}/I$ ($0\le j \le m$) defined near $o'$. Given $a_0=0$ and $m$  positive numbers $\{a_j\}_{1\le j\le m}$, the following two statements are equivalent:
	\begin{enumerate}
	    \item There exists a valuation $\nu$ on $C^{\text{an}}_{o'}$ such that $\nu(f_j)=a_j$ for all $0\le j\le m$;
	    \item $\sigma(\log|F|,\varphi)=\sum_{1\le j\le m}a_j$, where
	    \[F\coloneqq \prod_{0\le j\le m}\widetilde{P}(f_j), \quad \varphi\coloneqq \log\Big(\sum_{1\le j\le m}|\widetilde{P}(f_j)|^{\frac{1}{a_j}}\Big).\]
	\end{enumerate}
\end{Corollary}

\begin{Remark}
Theorem \ref{thm:interpolation} and Corollary \ref{C:interpolation-real} show that for any valuation $\nu$ on $C^{\text{an}}_{o'}/I$ and any finite collection of real analytic functions $\{f_j\}_{0\le j\le m}$ near $o'$ with $\nu(f_j)>0$ for all $j$, there exists a valuation $\widetilde\nu$ on $\mathcal{O}_o/(\widetilde{P}(I))$ such that $\tilde\nu(\widetilde{P}(f_j))=\nu(f_j)$ for all $j$.
\end{Remark}

 Let $I$ be an ideal in $\mathbb{R}[x_1,\ldots,x_n]$.
Denote  the restriction of $P$ on $\mathbb{R}[x_1,\ldots,x_n]$ also by $P$ and $\widetilde{P}$ on $\mathbb{R}[x_1,\ldots,x_n]/I$ also by $\widetilde{P}$. Assume that $\big(P(I)\big)$ is a prime ideal in
$\mathbb{C}[x_1,\ldots,x_n]$. Note that $\big(P(I)\big)$ is a prime ideal in $\mathbb{C}[z_1,\ldots,z_n]$ implies $I$ is a prime ideal in $\mathbb{R}[z_1,\ldots,z_n]$ (see Remark \ref{rem:prime ideal from complex to real poly}). Denote $X$ be the  affine variety defined by $\big(P(I)\big)$. Assume that $o\in X$ and $(X,o)$ is irreducible as a germ of analytic set. Let $\{f_j\}_{0\le j\le m}\subset \mathbb{R}[x_1,\ldots,x_n]/I$, $a_0=0$ and  $\{a_j\}_{1\le j\le m}$ be $m$ positive numbers. Set $F\coloneqq \prod_{0\le j\le m}\widetilde{P}(f_j)$ and $\varphi\coloneqq \log(\sum_{1\le j\le m}|\widetilde{P}(f_j)|^{\frac{1}{a_j}})$.

\begin{Corollary}\label{C:interpolation-real-poly}
 If $\sigma(\log|F|,\varphi)=\sum_{1\le j\le m}a_j$, then there exists a valuation $\nu$ on $\mathbb{R}[x_1,\ldots,x_n]/I$ such that $\nu(f_j)=a_j$ for all $0\le j\le m$.
	
	Conversely, if there exists a valuation $\nu$ on $\mathbb{R}[x_1,\ldots,x_n]/I$ satisfying that $\nu(f_j)=a_j$ for all $0\le j\le m$ and $\nu(x_l)>0$ for every $1\le l\le n$, then we have $\sigma(\log|F|,\varphi)=\sum_{1\le j\le m}a_j$.
\end{Corollary}

\section{Preparation}

\subsection{Basic knowledge}

In this subsection, we recall some basic results which will be used when we discuss Zhou weights near singular points.

Let $X$ be an analytic subset of $\Omega\in\mathbb{C}^{n}$ with pure dimension $d$. Let 
\[dV_{X}\coloneqq \frac{1}{2^{d}d!}\bigwedge^d(\sum_{1\le j\le n}\sqrt{-1}dz_j\wedge d\bar z_j)\]
be the volume form on $X$ induced by the standard volume form on $\mathbb{C}^{n}$.

Let $(u_{\alpha})$ be a family of plurisubharmonic functions on an open subset of $X$, and we assume that the family $(u_{\alpha})$  is locally uniformly bounded from above. We recall the following Choquet's Lemma.
\begin{Lemma}[see \cite{demailly-book}]
	\label{lem:Choquet} Every family $(u_{\alpha})$  has a countable subfamily $(v_{j})=(u_{\alpha(j)})$,
	such that its upper-envelope $v=\sup_{j}v_{j}$ satisfies $v\leq u\leq u^{*} = v^{*}$,
	where $u=\sup_{\alpha}u_{\alpha}$, $u^{*}(z)\coloneqq \lim_{\varepsilon\to0}\sup_{U_{z,\varepsilon}}u$
	and $v^{*}(z)\coloneqq \lim_{\varepsilon\to0}\sup_{U_{z,\varepsilon}}v$ are the regularizations of $u$ and $v$, where $U_{z,\varepsilon}\coloneqq X\cap\mathbb{B}^{n}(z,\varepsilon)$.
	
	Moreover, the upper regularization $u^{*}$ is plurisubharmonic
	and equals almost everywhere to $u$.
\end{Lemma}

\begin{Remark}
Note that the original version of Lemma \ref{lem:Choquet} is stated for  smooth  $X$. When $X$ has singularities, since we ignore the values of plurisubharmonic functions on singularities and $(u_{\alpha})$  is locally uniformly bounded from above, we  apply upper semicontinuous regularization on $X_{\text{reg}}$ and the same argument as the smooth case shows the Lemma \ref{lem:Choquet} still holds in the singular case.

\end{Remark}

Let $z\in X$ be a smooth point. Let $\varphi$ and $\varphi_0$ be plurisubharmonic functions on $X$ defined near $z$ and $f$ be a holomorphic function defined on $X$ near $z$.
The following result considers the growth of the volume of sub-level sets of plurisubharmonic functions.

\begin{Lemma}[see \cite{GZopen-effect}]
	\label{lem:JM-0}
	Assume that $|f_{0}|^{2}e^{-2(\varphi+\varphi_{0})}$ is not integrable near $z$,
	and $|f_{0}|^{2}e^{-2\varphi_{0}}$ is integrable near $z$.
	Then for any small enough neighborhood $U$ of $z$,
	there exists $C>0$ such that
	$$e^{2t}\int_{\{\varphi<-t\}\cap U}|f_{0}|^{2}e^{-2\varphi_{0}}>C$$
	for any $t\ge0$.
\end{Lemma}

We recall the following desingularization theorem due to Hironaka.

\begin{Theorem}[\cite{Hironaka}, see also \cite{BM-1991}]\label{thm:desing}
	Let $Y$ be a complex manifold, and $M$ be an analytic sub-variety in $Y$. Then there is a local finite sequence of blow-ups $\mu_j\colon Y_{j+1}\rightarrow Y_j$ $(Y_1\coloneqq Y,j=1,2,\ldots)$ with smooth centers $S_j$ such that:
	\begin{enumerate}
	    \item Each component of $S_j$ lies either in $(M_j)_{\sing}$ or in $M_j\cap E_j$, where $M_1\coloneqq M$, $M_{j+1}$ denotes the strict transform of $M_j$ by $\mu_j$, $(M_j)_{\sing}$ denotes the singular set of $M_j$, and $E_{j+1}$ denotes the exceptional divisor $\mu^{-1}_j(S_j\cup E_j)$;
	    \item Let $M'$ and $E'$ denote the final strict transform of $M$ and the exceptional divisor respectively. Then:
	    \begin{itemize}
	        \item[(a)] The underlying point-set $|M'|$ is smooth;
	        \item[(b)] $|M'|$and $E'$ simultaneously have only normal crossings;
	    \end{itemize}
	\end{enumerate}
\end{Theorem}

The {\it (b)} in the above theorem means that, locally, there is a coordinate system in which $E'$ is a union of coordinate hyperplanes and $|M'|$ is a coordinate subspace.

We recall that every locally bounded plurisubharmonic function can be extended across closed pluripolar sets.

\begin{Lemma}[{see \cite[Theorem 5.24]{demailly-book}}]
\label{l:psh-extend}
	Let $\varphi$ be a plurisubharmonic function on $\Omega\setminus E$, where $\Omega$ is a domain in $\mathbb{C}^n$ and $E$ is a closed pluripolar set in $\Omega$. If $\varphi$ is locally bounded above near $E$, then $\varphi$ extends  uniquely to a plurisubharmonic function on $\Omega$.
\end{Lemma}

Recall that  $X$ is an analytic subset of $\Omega\in\mathbb{C}^{n}$ with pure dimension $d$.
Let $z_0\in X$ be a singular point. Let $\varphi$ and $\varphi_0$ be plurisubharmonic functions on $X$ defined near $z_0$ and $f$ be a holomorphic function defined on $X$ near $z_0$. Recall that
\[dV_{X}\coloneqq \frac{1}{2^{d}d!}\bigwedge^d\big(\sum_{1\le j\le n}\sqrt{-1}dz_j\wedge d\bar z_j\big)\]
be the volume form on $X$ induced by the standard volume form on $\mathbb{C}^{n}$.

Using Lemma \ref{lem:JM-0}, Lemma \ref{l:psh-extend} and Theorem \ref{thm:desing}, we can generalize Lemma \ref{lem:JM-0} to the singular case.
\begin{Lemma}
	\label{lem:JM}
	Assume that $|f_{0}|^{2}e^{-2(\varphi+\varphi_{0})}$ is not integrable near $z_0$, and $|f_{0}|^{2}e^{-2\varphi_{0}}$ is integrable near $z_0$. Then for any small enough neighborhood $U\subset X$ of $z_0$, there exists $C>0$ such that
	$$e^{2t}\int_{\{\varphi<-t\}\cap U}|f_{0}|^{2}e^{-2\varphi_{0}}>C$$
	for any $t\ge0$.
\end{Lemma}

\begin{proof}

	We use Theorem \ref{thm:desing} on $\Omega$ to resolve
	the singularities of $X$, and denote the corresponding proper modification by $\mu\colon\widetilde\Omega\rightarrow\Omega$. Denote the strict transform of $X$ by $\widetilde X$ and denote $Z_0\coloneqq \widetilde X\cap\mu^{-1}(\{z_0\})$. 
	
	As $|f_{0}|^{2}e^{-2(\varphi+\varphi_{0})}$ is not integrable near $z_0$, there exists $\tilde{z}_0\in Z_0$ such that $$\int_{U}\mu^{*}(|f_{0}|^{2}e^{-2(\varphi+\varphi_{0})}dV_{X})=+\infty$$ for any neighborhood $U\subset\widetilde X$ of $\widetilde{z}_0$. Note that $\varphi\circ \mu$ and $\varphi_0\circ \mu$ are plurisubharmonic functions on $\widetilde X$ (where we can use Lemma \ref{l:psh-extend} to extend $\varphi\circ\mu$ and $\varphi_0\circ\mu$ from $X\setminus\mu^{-1}(X_{\sing})$ to $X$), and there exist holomorphic functions $\widetilde f_1,\ldots,\widetilde f_s$ near $\widetilde z_0$ such that $\mu^{*}(|f_{0}|^{2}dV_{X})=\sum_{1\le j\le s}|\widetilde f_j|^2dV_{\widetilde{X}}$, where $dV_{\widetilde{X}}$ is a volume form on $\widetilde X$. By Lemma \ref{lem:JM-0}, for any small neighborhood $U'\subset\widetilde X$ of $\widetilde{z}_0$, there exists $C>0$ such that
	$$e^{2t}\int_{\{\mu^{-1}(\varphi)<-t\}\cap U'}\mu^{*}(|f_{0}|^{2}e^{-2\varphi_{0}}dV_{X})>C$$
	for any $t\ge0$, which proves Lemma \ref{lem:JM}.
\end{proof}

The following result is a corollary of Guan--Zhou's lower semicontinuity property of
plurisubharmonic functions with a multiplier ideal in \cite{GZopen-effect}.

\begin{Theorem}[\cite{GZopen-effect}]
	\label{thm:SOC}
	Let $f$ be a holomorphic function near $z_0$, and $\{\varphi_j\}$ be a sequence of plurisubharmonic functions on a neighborhood $V\subset X$ of $z_0$. Assume that $\{\varphi_j\}$ is increasingly convergent to a plurisubharmonic function $\varphi$ almost everywhere with respect to Lebesgue measure. If $|f_0|^2e^{-2\varphi_j}$ is not integrable near $z_0$ for any $j$, then  $|f_0|^2e^{-2\varphi}$ is not integrable near $z_0$.
\end{Theorem}

Let $\varphi$ and $\varphi_{0}$ be plurisubharmonic functions near $z_0\in X$
such that $|f_{0}|^{2}e^{-2\varphi_{0}}$ is integrable near $z_0$.
Assume that $\varphi\geq N\log|z|+O(1)$ near $z_0$ for large enough $N\gg0$.

\begin{Lemma}
	\label{lem:jump_asyp_C}
	Assume that $(f_{0},z_0)\not\in\mathcal{I}(\varphi+\varphi_{0})_{z_0}$,
	and $(f_{0},z_0)\in\mathcal{I}\big((1-\varepsilon)\varphi+\varphi_{0}\big)_{z_0}$ for any $\varepsilon\in(0,1)$.
	Then
	for any neighborhood $U\subset X$ of $z_0$,
	$$\lim_{t\to+\infty}\frac{-\log\int_{\{\varphi<-t\}\cap U}|f_{0}|^{2}e^{-2\varphi_{0}}}{2t}=1.$$
\end{Lemma}

\begin{proof}
		As $(f_{0},z_0)\not\in\mathcal{I}(\varphi+\varphi_{0})_{z_0}$, we have
	$$\liminf_{t\rightarrow+\infty}e^{2t}\int_{\{\varphi<-t\}\cap U}|f_0|^2e^{-2\varphi_0}>0$$
	(see Lemma \ref{lem:JM}), which implies that $$\limsup_{t\to+\infty}\frac{-\log \int_{\{\varphi<-t\}\cap U}|f_{0}|^{2}e^{-2\varphi_{0}}}{2t}\leq 1.$$

	As $\varphi\geq N\log|z|+O(1)$ for large enough $N\gg0$ and $(f_0,o)\in\mathcal{I}((1-\epsilon)\varphi+\varphi_0)_o$ for any $\epsilon>0$, we have
	\begin{equation}
		\nonumber\begin{split}
			&\limsup_{t\to+\infty}e^{2(1-\epsilon)t}\int_{\{\varphi<-t\}\cap U}|f_{0}|^{2}e^{-2\varphi_{0}}\\
			\le& \ \limsup_{t\to+\infty}\int_{\{\varphi<-t\}\cap U}|f_{0}|^{2}e^{-2\varphi_{0}-2(1-\epsilon)\varphi}\\
			=& \ 0,
		\end{split}
	\end{equation}
	which implies that 
	$$\liminf_{t\to+\infty}\frac{-\log \int_{\{\varphi<-t\}\cap U}|f_{0}|^{2}e^{-2\varphi_{0}}}{2t}\ge 1-\epsilon.$$
	Then we have 
		$$\lim_{t\to+\infty}\frac{-\log\int_{\{\varphi<-t\}\cap U}|f_{0}|^{2}e^{-2\varphi_{0}}}{2t}=1.$$
\end{proof}

\subsection{Concavity}

Let $(X,z_0)$ be an irreducible germ of an analytic set. Let $u$ and $v$ be Lebesgue measurable functions on $X$ with upper-bounds near $z_0$. Let $g$ be a nonnegative Lebesgue measurable function on $X$.

The following Lemma \ref{lem:G_key} and \ref{lem:jialidun0131} come from real analysis. 
When $z_0$ is a singular point of $X$, the proof of Lemma \ref{lem:G_key} and \ref{lem:jialidun0131} are the same as the original proof in the smooth case and hence we omit the proofs.
\begin{Lemma}[see \cite{demailly2010}]
	\label{lem:G_key}
	Assume that $g^{2}e^{2(l_{1}v-(1+l_{2})u)}$ is integrable near $z_0$, where $l_{1},l_{2}>0$.
	Then
	$g^{2}e^{-2u}-g^{2}e^{-2\max\big\{u,\frac{l_{1}}{l_{2}}v\big\}}$ is integrable
	on a small enough neighborhood of $z_0$.
\end{Lemma}

Denote by
\[A_{u,v}(t)\coloneqq \sup\big\{c\colon g^{2}e^{2(tv-cu)} \ \text{is integrable near} \ z_0\big\}.\]
Note that  $A_{u,v}(t)$ is increasing with respect to $t$ (maybe $+\infty$ or $-\infty$).
Assume that $A_{u,v}(t)\in(0,+\infty]$ on $(t_{0}-\delta,t_{0}+\delta)$.
It follows from H\"{o}lder's inequality that $A_{u,v}(t)$ is concave on $(t_{0}-\delta,t_{0}+\delta)$. We recall the following basic property of $A_{u,v}(t)$.

\begin{Lemma}[see {\cite[Lemma 2.14]{BGMY-valuation}}]
	\label{lem:jialidun0131}
	Assume that $A_{u,v}(t)$ is strictly increasing on $(t_{0}-\delta,t_{0}+\delta)$.
	Then for any $b>0$,
	\begin{equation}
		\nonumber
		A_{\max\{u,\frac{1}{b}v\},v}(t_{0})=A_{u,v}(t_{0})
	\end{equation}
	holds if and only if $b\in\left(0,\lim_{\Delta t\to0-0}\frac{A_{u,v}(t_{0}+\Delta t)-A_{u,v}(t_{0})}{\Delta t}\right]$.
\end{Lemma}

\subsection{Tian functions and Zhou numbers}
\label{sec:tian function}

Let $\varphi$, $\psi$, $\varphi_{0}$ be plurisubharmonic functions near $z_0$, and let $f_{0}=(f_{0,1},\ldots,f_{0,m})$ be a vector, where $f_{0,1},\ldots,f_{0,m}$ are holomorphic functions near $z_0$.

Denote by
\[c_{z_0}(\varphi,t\psi)\coloneqq \sup\big\{c\colon |f_{0}|^{2}e^{-2\varphi_{0}}e^{2t\psi}e^{-2c\varphi} \ \text{is integrable near} \ z_0\big\},\]
which is a generalization of the jumping number (see \cite{JON-Mus2012,JON-Mus2014}).
Define the \emph{Tian function}
$$\mathrm{Tn}(t)\coloneqq c_{z_0}(\varphi,t\psi)$$
for any $t\in\mathbb{R}$. The H\"{o}lder inequality shows that $\mathrm{Tn}(t)$ is concave with respect to $t\in(-\infty,+\infty)$.

In the following lemma, assume that the following three statements hold
\begin{enumerate}
    \item $|f_{0}|^{2}e^{-2\varphi_{0}}$ is integrable near $z_0$;
    \item There exists integer $N_{0}\gg0$ such that $|f_{0}|^{2}e^{-2\varphi_{0}}|z|^{2N_{0}}e^{-2\mathrm{Tn}(0)\varphi}$ is integrable near $z_0$;
    \item There exists $s_{0}>0$, such that $\psi\leq s_{0}\log|z|+O(1)$ near $z_0\in X$.
\end{enumerate}

In this subsection, we discuss the derivatives of Tian functions $\mathrm{Tn}(t)$.

We give the	strictly increasing property of Tian functions in the following lemma.
\begin{Lemma}
	\label{lem:strict_decreasing}
	$\mathrm{Tn}(t)$ is strictly increasing near $0$.
\end{Lemma}

\begin{proof}
	Theorem \ref{thm:SOC} implies that
	there exists $\varepsilon_{0}>0$ such that
	$$|f_{0}|^{2}e^{-2\varphi_{0}}|z|^{2N_{0}}e^{-2(1+\varepsilon_{0})\mathrm{Tn}(0)\varphi}$$
	is integrable near $z_0$.
	Lemma \ref{lem:G_key}  shows that
	$$|f_{0}|^{2}e^{-2\varphi_{0}}e^{-2\big(1+\frac{\varepsilon_{0}}{2}\big)\mathrm{Tn}(0)\varphi}-
	|f_{0}|^{2}e^{-2\varphi_{0}}e^{-2\max\big\{\big(1+\frac{\varepsilon_{0}}{2}\big)\mathrm{Tn}(0)\varphi,\frac{2+\varepsilon_{0}}{\varepsilon_{0}}\log|z|^{N_{0}}\big\}}$$
	is integrable near $z_0$.
		Note that for any $t>0$,
	$$
	|f_{0}|^{2}e^{-2\varphi_{0}}\left(e^{2t\psi}e^{-2\big(1+\frac{\varepsilon_{0}}{2}\big)\mathrm{Tn}(0)\varphi}-
	e^{2t\psi}e^{-2\max\big\{\big(1+\frac{\varepsilon_{0}}{2}\big)\mathrm{Tn}(0)\varphi,\frac{2+\varepsilon_{0}}{\varepsilon_{0}}\log|z|^{N_{0}}\big\}}\right)
	$$
	is integrable near $z_0$
	and
	\begin{equation*}
		\begin{split}
			&|f_{0}|^{2}e^{-2\varphi_{0}}e^{2t\psi}e^{-2\max\big\{\big(1+\frac{\varepsilon_{0}}{2}\big)\mathrm{Tn}(0)\varphi,\frac{2+\varepsilon_{0}}{\varepsilon_{0}}\log|z|^{N_{0}}\big\}}
			\\\leq&
			C|f_{0}|^{2}e^{-2\varphi_{0}}|z|^{2s_0t}e^{-2\max\big\{\big(1+\frac{\varepsilon_{0}}{2}\big)\mathrm{Tn}(0)\varphi,\frac{2+\varepsilon_{0}}{\varepsilon_{0}}\log|z|^{N_{0}}\big\}}
			\\\leq&
			C|f_{0}|^{2}e^{-2\varphi_{0}}|z|^{2s_0t}e^{-2\frac{(2+\varepsilon_0)N_{0}}{\varepsilon_{0}}\log|z|}
		\end{split}
	\end{equation*}
	near $z_0$. Then it is clear that for any
	$t>\frac{2+\varepsilon_0}{\varepsilon_{0}}\frac{N_0}{s_0}$,
	\[|f_{0}|^{2}e^{-2\varphi_{0}}e^{2t\psi}e^{-2(1+\frac{\varepsilon_{0}}{2})\mathrm{Tn}(0)\varphi}\]
	is integrable near $z_0$, which shows
	$\mathrm{Tn}(t)>\big(1+\frac{\varepsilon_{0}}{2}\big)\mathrm{Tn}(0)$ for any $t>\frac{2+\varepsilon_0}{\varepsilon_{0}}\frac{N_0}{s_0}$.
	Then the concavity of $\mathrm{Tn}(t)$ implies that $\mathrm{Tn}(t)$ is strictly increasing near $0$.
\end{proof}

The following property of Tian functions $\mathrm{Tn}(t)$ will be used in the proof of Theorem \ref{thm:expression of relative types}.

\begin{Proposition}
	\label{p:sidediv}Assume that there exists $N\gg0$ such that $\varphi\geq N\log|z|$ near $o$. The following inequality holds
	\begin{equation*}
		\begin{split}
			&\frac{1}{\mathrm{Tn}(0)}\lim_{t\to0+0}\frac{\mathrm{Tn}(0)-\mathrm{Tn}(t)}{-t}
			\\\leq \ &
			\liminf_{t_{1}\to+\infty}\frac{1}{2t_{1}}\frac{\int_{\{\mathrm{Tn}(0)\varphi<-t_{1}\}\cap U}|f_{0}|^{2}e^{-2\varphi_{0}}(-2\psi)}{\int_{\{\mathrm{Tn}(0)\varphi<-t_{1}\}\cap U}|f_{0}|^{2}e^{-2\varphi_{0}}}\\
			\leq \ &
			\limsup_{t_{1}\to+\infty}\frac{1}{2t_{1}}\frac{\int_{\{\mathrm{Tn}(0)\varphi<-t_{1}\}\cap U}|f_{0}|^{2}e^{-2\varphi_{0}}(-2\psi)}{\int_{\{\mathrm{Tn}(0)\varphi<-t_{1}\}\cap U}|f_{0}|^{2}e^{-2\varphi_{0}}}\\
			\leq \  &\frac{1}{\mathrm{Tn}(0)}\lim_{t\to0+0}\frac{\mathrm{Tn}(0)-\mathrm{Tn}(-t)}{t}.
		\end{split}
	\end{equation*}
\end{Proposition}

\begin{proof}
	We prove Proposition \ref{p:sidediv} in two steps.
	
	\textbf{Step 1.}
	Theorem \ref{thm:SOC} shows that for a small enough constant $t>0$ ($t$ is dependent on $|f_{0}|^{2}e^{-2\varphi_{0}}$ and $\psi$), $|f_{0}|^{2}e^{-2\varphi_{0}}e^{-2t\psi}$ is  integrable near $z_0$. Since $e^{-2t\psi}\ge (-2t\psi)$, we have $$|f_{0}|^{2}e^{-2\varphi_{0}}(-2\psi)=\frac{1}{t}|f_{0}|^{2}e^{-2\varphi_{0}}(-2t\psi)\le \frac{1}{t}|f_{0}|^{2}e^{-2\varphi_{0}}e^{-2t\psi}$$ 
	is  integrable near $z_0$.
	Note that there exists $N\gg0$ such that $\varphi\geq N\log|z|$ near $z_0$. 
	There exists a neighborhood $U$ of $z_0$ such that
	for any small enough $t>0$ and $\varepsilon>0$ ($\epsilon$ depends on $t$),
	$$\limsup_{t_{1}\to+\infty}\int_{\{\mathrm{Tn}(0)\varphi<-t_{1}\}\cap U}|f_{0}|^{2}e^{-2\varphi_{0}}e^{-2\big(t\psi+(1-\varepsilon)\mathrm{Tn}(-t)\varphi\big)}=0,$$
	which implies that
	$$\limsup_{t_{1}\to+\infty}e^{2t_{1}}\int_{\{\mathrm{Tn}(0)\varphi<-t_{1}\}\cap U}|f_{0}|^{2}e^{-2\varphi_{0}}e^{-2\big(t\psi+(1-\varepsilon)\mathrm{Tn}(-t)\varphi+t_{1}\big)}=0.$$
	Then for large enough $t_{1}>0$,
	$$\int_{\{\mathrm{Tn}(0)\varphi<-t_{1}\}\cap U}|f_{0}|^{2}e^{-2\varphi_{0}}e^{-2\big(t\psi+(1-\varepsilon)\mathrm{Tn}(-t)\varphi+t_{1}\big)}<e^{-2t_{1}},$$
	i.e.,
	$$\log\left(\int_{\{\mathrm{Tn}(0)\varphi<-t_{1}\}\cap U}|f_{0}|^{2}e^{-2\varphi_{0}}e^{-2\big(t\psi+(1-\varepsilon)\mathrm{Tn}(-t)\varphi+t_{1}\big)}\right)<-2t_{1}.$$
	Combining with Lemma \ref{lem:jump_asyp_C},
	we obtain
	\begin{equation}
		\label{equ:tiandao}
		\begin{split}
			&\limsup_{t_{1}\to+\infty}\frac{1}{2t_{1}}\log\frac{\int_{\{\mathrm{Tn}(0)\varphi<-t_{1}\}\cap U}|f_{0}|^{2}e^{-2\varphi_{0}}e^{-2\big(t\psi+(1-\varepsilon)\mathrm{Tn}(-t)\varphi+t_{1}\big)}}{\int_{\{\mathrm{Tn}(0)\varphi<-t_{1}\}\cap U}|f_{0}|^{2}e^{-2\varphi_{0}}}
			\\ = \ &
			\limsup_{t_{1}\to+\infty}\frac{1}{2t_{1}}\log\int_{\{\mathrm{Tn}(0)\varphi<-t_{1}\}\cap U}|f_{0}|^{2}e^{-2\varphi_{0}}e^{-2\big(t\psi+(1-\varepsilon)\mathrm{Tn}(-t)\varphi+t_{1}\big)}
			\\&-\lim_{t_{1}\to+\infty}\frac{1}{2t_{1}}\log\int_{\{\mathrm{Tn}(0)\varphi<-t_{1}\}\cap U}|f_{0}|^{2}e^{-2\varphi_{0}}
			\\\leq \ &1-1=0.
		\end{split}
	\end{equation}

	It follows from Jensen's inequality and the concavity of the logarithm function that
	\begin{equation}
		\nonumber
		\begin{split}
			&\log\frac{\int_{\{\mathrm{Tn}(0)\varphi<-t_{1}\}\cap U}|f_{0}|^{2}e^{-2\varphi_{0}}e^{-2\big(t\psi+(1-\varepsilon)\mathrm{Tn}(-t)\varphi+t_{1}\big)}}{\int_{\{\mathrm{Tn}(0)\varphi<-t_{1}\}\cap U}|f_{0}|^{2}e^{-2\varphi_{0}}}
			\\\geq \ &
			\frac{\int_{\{\mathrm{Tn}(0)\varphi<-t_{1}\}\cap U}|f_{0}|^{2}e^{-2\varphi_{0}}\log\left(e^{-2(t\psi+(1-\varepsilon)\mathrm{Tn}(-t)\varphi)+t_{1}}\right)}{\int_{\{\mathrm{Tn}(0)\varphi<-t_{1}\}\cap U}|f_{0}|^{2}e^{-2\varphi_{0}}}
			\\= \ &
			\frac{\int_{\{\mathrm{Tn}(0)\varphi<-t_{1}\}\cap U}|f_{0}|^{2}e^{-2\varphi_{0}}\big(-2(t\psi+(1-\varepsilon)\mathrm{Tn}(-t)\varphi+t_{1})\big)}{\int_{\{\mathrm{Tn}(0)\varphi<-t_{1}\}\cap U}|f_{0}|^{2}e^{-2\varphi_{0}}}
			\\\geq \ &
			\frac{\int_{\{\mathrm{Tn}(0)\varphi<-t_{1}\}\cap U}|f_{0}|^{2}e^{-2\varphi_{0}}\big(-2(t\psi+(1-\varepsilon)\mathrm{Tn}(-t)(-t_{1})\frac{1}{\mathrm{Tn}(0)}+t_{1})\big)}{\int_{\{\mathrm{Tn}(0)\varphi<-t_{1}\}\cap U}|f_{0}|^{2}e^{-2\varphi_{0}}}.
		\end{split}
	\end{equation}
	Combining with inequality \eqref{equ:tiandao},
	we obtain that
	$$\limsup_{t_{1}\to+\infty}\frac{1}{2t_{1}}\frac{\int_{\{\mathrm{Tn}(0)\varphi<-t_{1}\}\cap U}|f_{0}|^{2}e^{-2\varphi_{0}}\big(-2(t\psi+(1-\varepsilon)\mathrm{Tn}(-t)(-t_{1})\frac{1}{\mathrm{Tn}(0)}+t_{1})\big)}{\int_{\{\mathrm{Tn}(0)\varphi<-t_{1}\}\cap U}|f_{0}|^{2}e^{-2\varphi_{0}}}\leq 0.$$
	Letting $\varepsilon\to0+0$ and 
	$t\to0+0$,
	we obtain
	\begin{equation}
		\label{equ:degang20190406c}
		\limsup_{t_{1}\to+\infty}\frac{1}{2t_{1}}\frac{\int_{\{\mathrm{Tn}(0)\varphi<-t_{1}\}\cap U}|f_{0}|^{2}e^{-2\varphi_{0}}(-2\psi)}{\int_{\{\mathrm{Tn}(0)\varphi<-t_{1}\}\cap U}|f_{0}|^{2}e^{-2\varphi_{0}}}\leq \frac{1}{\mathrm{Tn}(0)}\lim_{t\to0+0}\frac{\mathrm{Tn}(0)-\mathrm{Tn}(-t)}{t}.
	\end{equation}
	
	\
	
	\textbf{Step 2.} 
	By a similar discussion in Step 1 (only replacing $t$ by $-t$), 
	we obtain that
	$$\limsup_{t_{1}\to+\infty}\frac{1}{2t_{1}}\frac{\int_{\{\mathrm{Tn}(0)\varphi<-t_{1}\}\cap U}|f_{0}|^{2}e^{-2\varphi_{0}}\big(-2(-t\psi+(1-\varepsilon)\mathrm{Tn}(t)(-t_{1})\frac{1}{\mathrm{Tn}(0)}+t_{1})\big)}{\int_{\{\mathrm{Tn}(0)\varphi<-t_{1}\}\cap U}|f_{0}|^{2}e^{-2\varphi_{0}}}\leq 0.$$
	Letting $\varepsilon\to0+0$ and 
	$t\to0+0$,
	we obtain
	\begin{equation}
		\label{equ:20190406b}
		\begin{split}
			\frac{1}{\mathrm{Tn}(0)}\lim_{t\to0+0}\frac{\mathrm{Tn}(0)-\mathrm{Tn}(t)}{-t}
			\leq
			\liminf_{t_{1}\to+\infty}\frac{1}{2t_{1}}\frac{\int_{\{\mathrm{Tn}(0)\varphi<-t_{1}\}\cap U}|f_{0}|^{2}e^{-2\varphi_{0}}(-2\psi)}{\int_{\{\mathrm{Tn}(0)\varphi<-t_{1}\}\cap U}|f_{0}|^{2}e^{-2\varphi_{0}}}.
		\end{split}
	\end{equation}
	
	Combining inequality (\ref{equ:20190406b}) and inequality (\ref{equ:degang20190406c}),
	we get Proposition \ref{p:sidediv}.
\end{proof}

Let $\Phi_{z_0,\max}$ be a local Zhou weight related to $|f_0|^2e^{-2\varphi_0}$ near $z_0$. Taking $\varphi=\Phi_{z_0,\max}$, we have $\mathrm{Tn}(0)=1$.

\begin{Proposition}
	\label{p:relative=derivative}
	For any plurisubharmonic function $\psi$ satisfying that  $\psi\leq s_{0}\log|z|+O(1)$ near $z_0\in X$ for some $s_0>0$, the Tian function is differentiable at $t=0$, and
	$$\mathrm{Tn}(t)=\mathrm{Tn}(0)+\sigma(\psi,\Phi_{z_0,\max})t$$
for any $t\ge0.$
Especially, 	$$\sigma(\psi,\Phi_{z_0,\max})=\lim_{t\rightarrow0}\frac{\mathrm{Tn}(t)-\mathrm{Tn}(0)}{t}.$$
\end{Proposition}

\begin{proof}As $\psi\leq s_{0}\log|z|+O(1)$ near $z_0$ and $\Phi_{z_0,\max}\ge N\log|z|+O(1)$ near $z_0$ for some $N$, we have $\sigma(\psi,\Phi_{z_0,\max})>0$.
Lemma \ref{lem:jialidun0131} and the strong openness property (Theorem \ref{thm:SOC}) show that $|f_0|^2e^{-2\varphi_{0}}e^{-2\max\{\Phi_{z_0,\max},\frac{1}{b}\psi\}}$ is not integrable near $z_0$ if and only if $b\le \lim_{t\rightarrow0+0}\frac{\mathrm{Tn}(0)-\mathrm{Tn}(-t)}{t}$. 

If $b<\sigma(\psi,\Phi_{z_0,\max})$, it follows from $\psi\le b\Phi_{z_0,\max}+O(1)$ that $\max\{\frac{1}{b}\psi,\Phi_{z_0,\max}\}=\Phi_{z_0,\max}+O(1).$ Then $|f_0|^2e^{-2\varphi_{0}}e^{-2\max\{\Phi_{z_0,\max},\frac{1}{b}\psi\}}$ is not integrable near $z_0$. If $b>\sigma(\psi,\Phi_{z_0,\max})$, by definition of $\sigma(\psi,\Phi_{z_0,\max})$, we have $\max\{\frac{1}{b}\psi,\Phi_{z_0,\max}\}\not=\Phi_{z_0,\max}+O(1)$ and $\max\{\frac{1}{b}\psi,\Phi_{z_0,\max}\}\ge\Phi_{z_0,\max}+O(1)$ near $z_0$. Following form the definition of local Zhou weights, we know that  $|f_0|^2e^{-2\varphi_{0}}e^{-2\max\{\Phi_{z_0,\max},\frac{1}{b}\psi\}}$ is  integrable near $z_0$.
Then we have
\begin{equation}
\label{eq:0908a}
\sigma(\psi,\Phi_{z_0,\max})=\lim_{t\rightarrow0+0}\frac{\mathrm{Tn}(0)-\mathrm{Tn}(-t)}{t}.
\end{equation}

It follows from that $\psi\le b\Phi_{z_0,\max}+O(1)$ near $z_0$ for any $b<\sigma(\psi,\Phi_{z_0,\max})$ that 
$\mathrm{Tn}(t)\geq \mathrm{Tn}(0)+\sigma(\psi,\Phi_{z_0,\max})t$ for any $t\ge0$.
By and equality \eqref{eq:0908a} and the concavity of $\mathrm{Tn}(t)$,
we obtain that $\mathrm{Tn}(t)=\mathrm{Tn}(0)+\sigma(\psi,\Phi_{z_0,\max})t$
for any $t\ge0$ and $\lim_{t\rightarrow0}\frac{\mathrm{Tn}(t)-\mathrm{Tn}(0)}{t}=\sigma(\psi,\Phi_{z_0,\max}).$

Thus, Proposition \ref{p:relative=derivative} holds.
\end{proof}

\subsection{Relation between relative types and valuations}
In this subsection, we consider the relation between relative types and valuations for  holomorphic and plurisubharmonic functions.

We recall Skoda's division theorem.

\begin{Theorem}[see \cite{demailly2010}]
	\label{thm:Skoda}
	Let $\Omega$ be a weakly pseudoconvex K\"ahler manifold of dimension $n$  and $\varphi$ be a plurisubharmonic function on $\Omega$. Set $m=\min\{n,r-1\}$. Then for every holomorphic function $f$ on $\Omega$ such that
	$$I\coloneqq \int_{\Omega}|f|^2|g|^{-2(m+1+\epsilon)}e^{-\varphi}<+\infty,$$
	there exist holomorphic functions $(h_1,\ldots,h_r)$ on $\Omega$ such that $f=\sum_{1\le j\le r}g_j h_j$ and
	$$\int_{\Omega}|h|^2|g|^{-2(m+\epsilon)}e^{-\varphi}\le\Big(1+\frac{m}{\epsilon}\Big)I.$$
\end{Theorem}

Recall that $X$ is an analytic variety with pure dimension $d$ contained in $\mathbb{C}^{n}$. Let $z=(z_1,\ldots,z_n)$ be  the coordinate on $\mathbb{C}^{n}$.  Assume that the origin $o\in X_{\text{sing}}$ and $(X,o)$ is an irreducible germ of an analytic set.

Denote $\mathcal{O}^{w}_{X,x}$ be the ring of germs of weakly holomorphic functions defined on $X$ near $x$ and $\mathfrak{M}_{X,x}$ be the ring of germs of meromorphic functions defined on $X$ near $x$.

\begin{Theorem}[see \cite{demailly-book}]
\label{th:universal denominators}
    For every point $x \in X$, there is a neighborhood $V$ of $x$ and $h \in \mathcal{O}_X(V)$ such that $h^{-1}(0)$ is nowhere dense in $V$ and $h_y \mathcal{O}^{w}_{X,y} \subseteq \mathcal{O}_{X,y}$ for all $y \in V$; such a function $h$ is called a \emph{universal denominator} on $V$. In particular $\mathcal{O}^{w}_{X,x}$ is contained in the ring $\mathfrak{M}_{X,x}$.
\end{Theorem}
Using Theorem \ref{th:universal denominators}, we have following lemma.
\begin{Lemma}
\label{lem:lower bound of weak holomorphic}
Let $\nu$ be a valuation on $\mathcal{O}_{X,x}$ and $\nu$ naturally extended to be a valuation on the quotient field $\mathfrak{M}_{X,x}$. For any $s\in\mathcal{O}^{w}_{X,x}\subset \mathfrak{M}_{X,x}$, the value $\nu(s)$ has a universal lower bound $-\nu(h)$, where $h$ is the universal denominator of $\mathcal{O}^{w}_{X,x}$ near $x$.
\end{Lemma}
\begin{proof}
Let $s\in\mathcal{O}^{w}_{X,x}$ be the germ of any weakly holomorphic function defined near $x$. It follows from Theorem \ref{th:universal denominators} that we can find an $h_{x}\in \mathcal{O}_{X,x}$ such that $h_{x}s_{x}\in \mathcal{O}_{X,x}$. Note that $\nu|_{\mathcal{O}_{X,x}}\ge 0$. Then $\nu(s)=\nu(hs)-\nu(h)\ge -\nu(h)$.

Lemma \ref{lem:lower bound of weak holomorphic} has been proved.
\end{proof}

\begin{Lemma}\label{l:relativetype}
	For any valuation $\nu$ on $\mathcal{O}_{X,o}$ where $o\in X_{\sing}$ and holomorphic functions $f_0,f_1,\ldots,f_r$ belonging to $\mathcal{O}_{X,o}$ with $\nu(f_j)>0$ for $1\le j\le r$, we have
	$\sigma(\log|f_0|,\varphi)\le \nu(f_0)$, where $\varphi\coloneqq  \log\big(\sum_{1\le j\le r}|f_j|^{\frac{1}{\nu(f_j)}}\big)$.
\end{Lemma}

\begin{proof}
	We prove this lemma by contradiction: if not, there exists $\delta_1>0$ such that
	$\log|f_0|\le (\nu(f_0)+\delta_1)\varphi+O(1)$ near $o$. Then we can find some $\delta_2>0$ and a set of positive integers $\{m_0,\ldots,m_r\}$ such that
	\[\log|f_0|\le \frac{\nu(f_0)+\delta_2}{m_0}\tilde\varphi+O(1) \quad  \ \text{near} \ o,\]
	where
	$\tilde\varphi\coloneqq\log\big(\sum_{1\le j\le r}|f_j^{m_j}|\big)$ and $\nu(f_j^{m_j})\ge m_0$ for $1\le j\le r$. For any positive integer $l$, we have
	\[\log|f_0^l|\le \frac{\nu(f_0^l)+l\delta_2}{m_0}\tilde\varphi+O(1)\]
	holds on $\big(U\cap X\big)\backslash X_{\sing}$, where $U\coloneqq \{x\in \mathbb{C}^n: \log|z|<\delta\}$.
	
		We use Theorem \ref{thm:desing} on $\Omega$ to resolve
	the singularities of $X$, and denote the corresponding proper modification by $\mu\colon\tilde\Omega\rightarrow\Omega$. Denote the strict transform of $X$ by $\tilde X$ and denote $Z_0\coloneqq \tilde\Omega\cap\mu^{-1}(\{o\})$. Then $\mu^{-1}(U)\cap \widetilde{X}$ contains $Z_0\cap \widetilde{X}$ and 
	\begin{equation}
	\label{eq:relativetype1}
	    \log|f_0^l\circ \mu|\le \frac{\nu(f_0^l)+l\delta_2}{m_0}\tilde\varphi\circ \mu+O(1)
	\end{equation}
	holds on $\mu^{-1}(U)\cap \widetilde{X}$. Note that $Z_o\cap \widetilde{X}$ is compact in $\widetilde{X}$, taking $\delta$ small enough, we may assume that $\mu^{-1}(U)\cap \widetilde{X}$ is relatively compact in $\widetilde{X}$. Inequality \eqref{eq:relativetype1} and the relative compactness of $\mu^{-1}(U)\cap \widetilde{X}$ imply that there exists a small $\epsilon>0$ (independent of $l$) such that
		\begin{equation}
	\label{eq:relativetype2}
	   \int_{\mu^{-1}(U)\cap \widetilde{X}} |f_0^l\circ\mu|^2e^{-2\frac{\nu(f_0^l)+l\delta_2+\epsilon}{m_0}\tilde\varphi\circ\mu}<+\infty.
	\end{equation}

	Let $|F|\coloneqq \sum_{1\le j\le r}|f_j^{m_j}\circ\mu|$, $\epsilon_0\coloneqq \frac{\epsilon}{m_0}$ and $m\coloneqq \min\{n,r-1\}$. By inequality \eqref{eq:relativetype2}, we have
	\begin{equation*}
		\begin{split}
			& \int_{\mu^{-1}(U)\cap \widetilde{X}}|f_0^l\circ\mu|^2|F|^{-2(m+1)-2\epsilon_0}|F|^{-2\big(\frac{\nu(f_0^l)+l\delta_2}{m_0}-(n+1)\big)} \\
			\le&C \int_{\mu^{-1}(U)\cap \widetilde{X}}|f_0^l\circ\mu|^2|F|^{-2(n+1)-2\epsilon_0}|F|^{-2\big(\frac{\nu(f_0^l)+l\delta_2}{m_0}-(n+1)\big)} \\
			= &C \int_{\mu^{-1}(U)\cap \widetilde{X}}|f_0^l\circ\mu|^2e^{-2\frac{\nu(f_0^l)+l\delta_2+\epsilon}{m_0}\tilde\varphi\circ\mu}<+\infty.
		\end{split}
	\end{equation*}
Note that $\big(\log|z\circ\mu|\big)|_{\widetilde{X}}$ is a smooth plurisubharmonic function on $\widetilde{X}$ and $\mu^{-1}(U)\cap \widetilde{X}=\{x\in \widetilde{X}\colon \big(\log|z\circ\mu|\big)|_{\widetilde{X}}<\delta\}$. We know that $\mu^{-1}(U)\cap \widetilde{X}$ is a weakly pseudoconvex K\"ahler manifold.

By Theorem \ref{thm:Skoda}, there exist holomorphic functions $(h_1,\ldots,h_r)$ on $\mu^{-1}(U)\cap \widetilde{X}$ such that $f_0^l\circ\mu=\sum_{1\le j\le r}(f_j\circ\mu)^{m_j} h_j$ holds on $\mu^{-1}(U)\cap \widetilde{X}$ and
	\begin{equation}
		\label{eq:est by skoda}
		\begin{split}
			&\int_{\mu^{-1}(U)\cap \widetilde{X}}|h|^2|F|^{-2(m+1)-2\epsilon_0}|F|^{-2\big(\frac{\nu(f_0^l)+l\delta_2}{m_0}-1-(n+1)\big)}\\
			=& \int_{\mu^{-1}(U)\cap \widetilde{X}}|h|^2|F|^{-2m-2\epsilon_0}|F|^{-2\big(\frac{\nu(f_0^l)+l\delta_2}{m_0}-(n+1)\big)}< +\infty.
		\end{split}
	\end{equation}
 It follows from the estimate \eqref{eq:est by skoda} that, using Theorem \ref{thm:Skoda} for each $h_j$ $(1\le j\le r)$, there exist holomorphic functions $(h_{j,1},\ldots,h_{j,r})$ on $\mu^{-1}(U)\cap \widetilde{X}$ such that $h_j=\sum_{1\le k\le r}h_{j,k}(f_k\circ\mu)^{m_k}$ holds on $\mu^{-1}(U)\cap \widetilde{X}$ and
	\begin{equation*}
		\int_{\mu^{-1}(U)\cap \widetilde{X}}\sum_{1\le k\le r}|h_{j,k}|^2|F|^{-2m-2\epsilon_0}|F|^{-2\big(\frac{\nu(f_0^l)+l\delta_2}{m_0}-1-(n+1)\big)}< +\infty.
	\end{equation*}
Since $f_0^l\circ\mu=\sum_{1\le j\le r}(f_j\circ\mu)^{m_j} h_j$ and $h_j=\sum_{1\le k\le r}h_{j,k}(f_k\circ\mu)^{m_k}$ hold on $\mu^{-1}(U)\cap \widetilde{X}$, we have 
\[f_0^l\circ\mu=\sum_{1\le j,k \le r}(f_j\circ\mu)^{m_j}(f_k\circ\mu)^{m_k}h_{j,k}\]
holds on $\mu^{-1}(U)\cap \widetilde{X}$.

Denote $k\coloneqq \big\lfloor\frac{\nu(f_0^l)+l\delta_2}{m_0}\big\rfloor-n-1$. 
Using Theorem \ref{thm:Skoda} $k$ times, we have 
\begin{equation}
    \label{eq:decompostion upper}
f_0^l\circ\mu=\sum_{1\le j_1,j_2,\ldots,j_k \le r}(f_{j_1}\circ\mu)^{m_{j_1}}(f_{j_2}\circ\mu)^{m_{j_2}}\cdots (f_{j_k}\circ\mu)^{m_{j_k}}
h_{J}
\end{equation}
hold on $\mu^{-1}(U)\cap \widetilde{X}$, where $J=(j_1,j_2,\ldots,j_k)$ is a multi-index and $h_{J}$ is an $L^2$ integrable holomorphic function (depending on $l$) on $\mu^{-1}(U)\cap \widetilde{X}$ for any $J$. As $h_{J}$ is $L^2$ integrable on $\mu^{-1}(U)\cap \widetilde{X}$, we know that $|h_J|$ is bounded near $Z_0\cap \widetilde{X}$. Then $\mu_{*
}(h_J)$ is a weakly holomorphic function defined on $U\cap X$. It follows from the decomposition \eqref{eq:decompostion upper} that we have
	\begin{equation}\nonumber
f_0^l=\sum_{1\le j_1,j_2,\ldots,j_k \le r}(f_{j_1})^{m_{j_1}}(f_{j_2})^{m_{j_2}}\cdots (f_{j_k})^{m_{j_k}}
\mu_{*}(h_J)
\end{equation}
	holds on $U\cap X$. Hence, for the germ level, we know
	$$\big(f_0^l\big)_{o}=\sum_{1\le j_1,j_2,\ldots,j_k \le r}\big(f_{j_1}^{m_{j_1}}\big)_{o}\big(f_{j_2}^{m_{j_2}}\big)_{o}\cdots \big(f_{j_m}^{m_{j_m}}\big)_{o}
\big(\mu_{*}(h_J)\big)_{o}$$
holds in $\mathfrak{M}_{X,o}$, where $\mathfrak{M}_{X,o}$ is the ring of germs of meromorphic functions near $o$.

 Lemma \ref{lem:lower bound of weak holomorphic} tells that $\nu(h_J)\ge -C$ for some constant $C>0$ (independent of $f_0$ and $l$) for any $J$. Note that $\nu(f_j^{m_j})\ge m_0$ for $1\le j\le r$ and $k\coloneqq \big\lfloor\frac{\nu(f_0^l)+l\delta_2}{m_0}\big\rfloor-n-1$. We have
	\begin{equation}\nonumber
		\begin{split}
			\nu(f_0^l)&\ge \min_{1\le j_1,j_2,\ldots,j_k \le r}\big\{\nu\big(f_{j_1}^{m_{j_1}}f_{j_2}^{m_{j_2}}\cdots f_{j_m}^{m_{j_m}}\mu_{*}(h_J)\big)\big\}\\
	&\ge m_0k-C\\
			&\ge m_0\Big(\frac{\nu(f_0^l)+l\delta_2}{m_0}-n-2\Big)-C\\
			&=\nu(f_0^l)+l\delta_2-m_0(n+2)-C,
		\end{split}
	\end{equation}
	which implies $l\le\frac{m_0(n+2)}{\delta_2}+C$. This is a contradiction since we can choose $l$ arbitrarily large. Then we must have $\sigma(\log|f_0|,\varphi)\le \nu(f_0)$.
\end{proof}

When we consider valuations on the ring of germs of weakly holomorphic functions, we have the following similar result as Lemma \ref{l:relativetype}.

\begin{Lemma}\label{l:relativetype weakly}
	For any valuation $\nu$ on $\mathcal{O}^w_{X,o}$ where $o\in X_{\sing}$ and holomorphic functions $f_0,f_1,\ldots,f_r$ belonging to $\mathcal{O}_{X,o}$ with $\nu(f_j)>0$ for $1\le j\le r$, we have
	$\sigma(\log|f_0|,\varphi)\le \nu(f_0)$, where $\varphi\coloneqq  \log\big(\sum_{1\le j\le r}|f_j|^{\frac{1}{\nu(f_j)}}\big)$.
\end{Lemma}
\begin{proof}
The proof of Lemma \ref{l:relativetype weakly} is almost the same as the proof of Lemma \ref{l:relativetype}. We just give the sketch, and when there are some differences, we will discuss in detail.

	We  prove Lemma \ref{l:relativetype weakly}  by contradiction. Assume that the conclusion does not hold. Then there exists $\delta_1>0$ such that
	$\log|f_0|\le (\nu(f_0)+\delta_1)\varphi+O(1)$ near $o$. Then we can find some $\delta_2>0$ and a set of positive integers $\{m_0,\ldots,m_r\}$ such that
	\[\log|f_0|\le \frac{\nu(f_0)+\delta_2}{m_0}\tilde\varphi+O(1) \quad  \ \text{near} \ o,\]
	where
	$\tilde\varphi\coloneqq\log\big(\sum_{1\le j\le r}|f_j^{m_j}|\big)$ and $\nu(f_j^{m_j})\ge m_0$ for $1\le j\le r$. For any positive integer $l$, we have
	\[\log|f_0^l|\le \frac{\nu(f_0^l)+l\delta_2}{m_0}\tilde\varphi+O(1)\]
	holds on $\big(U\cap X\big)\backslash X_{\sing}$, where $U\coloneqq \{x\in \mathbb{C}^n: \log|z|<\delta\}$.
	Using the same argument and notations as the proof of Lemma \ref{l:relativetype}, when $l$ is big enough and denote  $k\coloneqq \big\lfloor\frac{\nu(f_0^l)+l\delta_2}{m_0}\big\rfloor-n-1$, we have following decomposition
\begin{equation}
    \label{eq:decompostion upper weakly}
f_0^l\circ\mu=\sum_{1\le j_1,j_2,\ldots,j_k \le r}(f_{j_1}\circ\mu)^{m_{j_1}}(f_{j_2}\circ\mu)^{m_{j_2}}\cdots (f_{j_k}\circ\mu)^{m_{j_k}}
h_{J}
\end{equation}
holds on $\mu^{-1}(U)\cap \widetilde{X}$, where $J=(j_1,j_2,\ldots,j_k)$ is a multi-index and $h_{J}$ is an $L^2$ integrable holomorphic functions (depending on $l$) on $\mu^{-1}(U)\cap \widetilde{X}$ for any $J$. As $h_{J}$ is $L^2$ integrable on $\mu^{-1}(U)\cap \widetilde{X}$, we know that $|h_J|$ is bounded near $Z_0\cap \widetilde{X}$. Then $\mu_{*
}(h_J)$ is a weakly holomorphic function defined on $U\cap X$. It follows from the decomposition \eqref{eq:decompostion upper weakly} that we have
\begin{equation}\nonumber
f_0^l=\sum_{1\le j_1,j_2,\ldots,j_k \le r}(f_{j_1})^{m_{j_1}}(f_{j_2})^{m_{j_2}}\cdots (f_{j_k})^{m_{j_k}}
\mu_{*}(h_J)
\end{equation}
	holds on $U\cap X$. Hence, for the germ level, we know
	$$\big(f_0^l\big)_{o}=\sum_{1\le j_1,j_2,\ldots,j_k \le r}\big(f_{j_1}^{m_{j_1}}\big)_{o}\big(f_{j_2}^{m_{j_2}}\big)_{o}\cdots \big(f_{j_k}^{m_{j_k}}\big)_{o}
\big(\mu_{*}(h_J)\big)_{o}$$
holds in $\mathcal{O}^{w}_{X,o}$.

  Note that $\nu(f_j^{m_j})\ge m_0$ for $1\le j\le r$ and $k\coloneqq \big\lfloor\frac{\nu(f_0^l)+l\delta_2}{m_0}\big\rfloor-n-1$. We have
	\begin{equation}\nonumber
		\begin{split}
			\nu(f_0^l)&\ge \min_{1\le j_1,j_2,\ldots,j_k \le r}\big\{\nu\big(f_{j_1}^{m_{j_1}}f_{j_2}^{m_{j_2}}\cdots f_{j_k}^{m_{j_k}}\mu_{*}(h_J)\big)\big\}\\
	&\ge m_0k\\
			&\ge m_0\Big(\frac{\nu(f_0^l)+l\delta_2}{m_0}-n-2\Big)\\
			&=\nu(f_0^l)+l\delta_2-m_0(n+2),
		\end{split}
	\end{equation}
	which implies $l\le\frac{m_0(n+2)}{\delta_2}$. This is a  contradiction since we can choose $l$ arbitrarily large. Then we must have $\sigma(\log|f_0|,\varphi)\le \nu(f_0)$.
\end{proof}

Now we turn to the valuations on $\mathbb{C}[z_1,\ldots,z_n]/I$ a quotient ring of the polynomial ring, where $I$ is a prime ideal in $\mathbb{C}[z_1,\ldots,z_n]$.

 Let $X\coloneqq V(I)$ be  the affine variety defined by $I$ and $o\in X$, where $o$ is the origin in $\mathbb{C}^n$. Denote the germ of the set $X$ at $x$ by $(X,x)$. 
 We firstly recall the following basic result
\begin{Remark}
\label{rem:subring complex case}
$\mathbb{C}[z_1,...,z_n]/I$ is a subring of $\mathcal{O}_{X,o}:=\mathcal{O}_{\mathbb{C}^n,o}/\big(I\cdot \mathcal{O}_{\mathbb{C}^n,o}\big)$.
\end{Remark}
\begin{proof}
Note that $\mathbb{C}[z_1,...,z_n]$ can be naturally embedded into $ \mathcal{O}_{\mathbb{C}^n,o}$. Then we have an induced morphism
\begin{equation}\nonumber
\begin{split}
\Phi\colon \mathbb{C}[z_1,...,z_n]/I &\longrightarrow \mathcal{O}_{\mathbb{C}^n,o}/\big(I\cdot \mathcal{O}_{\mathbb{C}^n,o}\big) \\
 f+I &\longmapsto f+\big(I\cdot \mathcal{O}_{\mathbb{C}^n,o}\big).
\end{split}
\end{equation}
Now we show that $\Phi$ is injective. Assume that $\Phi(f_1+I)=0+\big(I\cdot \mathcal{O}_{\mathbb{C}^n,o}\big)$, where $f_1\in \mathbb{C}[z_1,...,z_n]$. We would like to show that $f_1\in I$ and hence $f_1+I=0$ in $\mathbb{C}[z_1,...,z_n]/I$. Let $\{g_j\}_{j=1}^m$ be the generators of $I$ in $\mathbb{C}[z_1,...,z_n]$. Then we know that $\{g_j\}_{j=1}^m$ also generate $\big(I\cdot \mathcal{O}_{\mathbb{C}^n,o}\big)$ in $\mathcal{O}_{\mathbb{C}^n,o}$. By definition and $\Phi(f_1+I)=0+\big(I\cdot \mathcal{O}_{\mathbb{C}^n,o}\big)$, there exists $h_j\in \mathcal{O}_{\mathbb{C}^n,o}$, where $j=1,\cdots,m,$ such that
$$f_1=\sum_{j=1}^m h_kg_k$$
in $\mathcal{O}_{\mathbb{C}^n,o}$. Hence there exists an open neighborhood $U$ of $o$ such that $h_k$ is defined on $U$ for each $k=1,2,\ldots,m$. Recall that $X=V(I)$ and $\{g_j\}_{j=1}^m$ are generators of $I$. We know, for any $x\in X$, $g_k(x)=0$ and hence $f_1(x)=0$ for any $x\in X\cap U$. It follows from $I$ is a prime ideal and Hilbert's Nullstellensatz  theorem that $X$ is irreducible, i.e., $X$ cannot be written as the union of two analytic varieties $X_1,X_2\subset \mathbb{C}^n$ with $X_1,X_2\neq V$. Hence $X_{\text{reg}}$ is connected (see chapter $0$ of \cite{griffith}). As $f_1(x)=0$ for any $x\in X\cap U$ and $X_{\text{reg}}$ is connected, we know that $f_1|_{X}\equiv 0$.
Hence $f_1\in \mathbb{I}(X)$, where $\mathbb{I}(X)\coloneqq\{f\in \mathbb{C}[z_1,...,z_n]\colon f(x)=0 \text{ for all }x \in X\}$. It follows from Hilbert's Nullstellensatz  theorem, $X=V(I)$ and $I$ is a prime ideal that $\mathbb{I}(X)=\sqrt{I}=I$. Thus $f\in I$.
\end{proof}

 Let $\mathcal{O}^{w}_{X,x}$ be the ring of germs of weakly holomorphic functions defined  near $x$. When $X$ is  algebraic and $(X,o)$ is irreducible as a germ of analytic set, we show that the universal denominator for $\mathcal{O}^{w}_{X,x}$ can be chosen be a polynomial for any $x\in V$, where $V$ is a neighborhood of $o$.
\begin{Theorem}[=Theorem \ref{th:universal denominators alg appendix} in the appendix]
\label{th:universal denominators alg}
Assume that $X$ is algebraic and $(X,o)$ is irreducible as a germ of analytic set. Let $V$ be a neighborhood of $o$ in $X$. There exists a polynomial $\delta\in \mathbb{C}[z_1,\ldots,z_d]$, where $d=\dim X$, such that   $\delta_x\mathcal{O}^{w}_{X,x}\subset \mathcal{O}_{X,x}$ for any $x\in V$. Such $\delta$ is called a universal denominator for weakly holomorphic functions on $V$.  
\end{Theorem}
\begin{proof}
   See the proof of  Theorem \ref{th:universal denominators alg appendix} in the appendix.
\end{proof}

\begin{Lemma}
\label{l:relativetype2 alg}  Let $\nu$ be a valuation on $\mathbb{C}[z_1,\ldots,z_n]/I$ such that $\nu(z_i)>0$ for any $i=1,\ldots,n$.  For any elements $f_0,f_1,\ldots,f_r\in\mathbb{C}[z_1,\ldots,z_n]/I$  satisfying $f_j(o)=0$ for $0\le j\le r$, we have
	$\sigma(\log|f_0|,\varphi)\le \nu(f_0)$, where $\varphi\coloneqq \log\big(\sum_{1\le j\le r}|f_j|^{\frac{1}{\nu(f_j)}}\big)$ is plurisubharmonic function defined on $X$. 
\end{Lemma}

\begin{proof}

	We also prove Lemma \ref{l:relativetype2 alg}  by contradiction. Assume that the conclusion does not hold. Then there exists $\delta_1>0$ such that
	$\log|f_0|\le (\nu(f_0)+\delta_1)\varphi+O(1)$ near $o$. Then we can find some $\delta_2>0$ and a set of positive integers $\{m_0,\ldots,m_r\}$ such that
	\[\log|f_0|\le \frac{\nu(f_0)+\delta_2}{m_0}\tilde\varphi+O(1) \quad  \ \text{near} \ o,\]
	where
	$\tilde\varphi\coloneqq\log\big(\sum_{1\le j\le r}|f_j^{m_j}|\big)$ and $\nu(f_j^{m_j})\ge m_0$ for $1\le j\le r$. For any positive integer $l$, we have
	\[\log|f_0^l|\le \frac{\nu(f_0^l)+l\delta_2}{m_0}\tilde\varphi+O(1)\]
	holds on $\big(U\cap X\big)\backslash X_{\sing}$, where $U\coloneqq \{x\in \mathbb{C}^n: \log|z|<\delta\}$.
	Using the same argument and notations as the proof of Lemma \ref{l:relativetype}, when $l$ is big enough and denote  $k\coloneqq \big\lfloor\frac{\nu(f_0^l)+l\delta_2}{m_0}\big\rfloor-n-1$, we have following decomposition
\begin{equation}
    \label{eq:decompostion upper alg}
f_0^l\circ\mu=\sum_{1\le j_1,j_2,\ldots,j_k \le r}(f_{j_1}\circ\mu)^{m_{j_1}}(f_{j_2}\circ\mu)^{m_{j_2}}\cdots (f_{j_k}\circ\mu)^{m_{j_k}}
h_{J}
\end{equation}
hold on some neighborhood $U$ of $p\in \mu^{-1}(o)\Subset \widetilde{X}$, where $J=(j_1,j_2,\ldots,j_k)$ is a multi-index and $h_{J}$ are $L^2$ integrable holomorphic functions (depending on $l$) on $U$ for any $J$.
 As $h_{J}$ are $L^2$ integrable on $\mu^{-1}(U)\cap \widetilde{X}$, we know that $|h_J|$ is bounded near $Z_0\cap \widetilde{X}$. Then $\mu_{*
}(h_J)$ is a weakly holomorphic function defined on $U\cap X$. It follows from decomposition \eqref{eq:decompostion upper alg} that we have
	\begin{equation}
    \label{eq:decompostion down alg}
f_0^l=\sum_{1\le j_1,j_2,\ldots,j_k \le r}(f_{j_1})^{m_{j_1}}(f_{j_2})^{m_{j_2}}\cdots (f_{j_k})^{m_{j_k}}
\mu_{*}(h_J)
\end{equation}
	holds on $U\cap X$, where $\mu_{*
}(h_J)$ is a weakly holomorphic function defined on $U\cap X$. 

It follows from Theorem \ref{th:universal denominators} and Theorem \ref{th:universal denominators alg} that, for any $J$, there exists a holomorphic function $g_J$ on $U\cap X$ such that $\mu_{*}(h_J)=\frac{g_J}{\delta}$, where $\delta$ is a polynomial.
Thus, we have
	\begin{equation}
    \label{eq:decompostion down alg 2}
\delta f_0^l=\sum_{1\le j_1,j_2,\ldots,j_k \le r}(f_{j_1})^{m_{j_1}}(f_{j_2})^{m_{j_2}}\cdots (f_{j_k})^{m_{j_k}}g_J
\end{equation}
holds on $U\cap X$.

Let $\widetilde{m}=(z_1,\ldots,z_n)$ be the ideal generated by $z_1,\ldots,z_n$ in $\mathbb{C}[z_1,\ldots,z_n]/I$ and $\widehat{m}=(z_1,\ldots,z_n)$ be the ideal generated by $z_1,\ldots,z_n$ in $\mathcal{O}_{X,o}=\mathcal{O}_{\mathbb{C}^n,o}/I$. For any integer $N>0$, let $g_{J,N}$ be the finite terms of Taylor expansion  of $g_J$ such that   $[g_{J,N}]=[g_J]$ in $\mathcal{O}_{X,o}/\widehat{m}^N$. Note that $g_{J,N}$ can be viewed as an element in $\mathbb{C}[z_1,\ldots,z_n]/I$. Denote
$$R_N\coloneqq \delta f_0^l-\sum_{1\le j_1,j_2,\ldots,j_k \le r}(f_{j_1})^{m_{j_1}}(f_{j_2})^{m_{j_2}}\cdots (f_{j_k})^{m_{j_k}}g_{J,N}.$$
Then $R_N$ is a polynomial in $\mathbb{C}[z_1,\ldots,z_n]/I$.
It follows from equality \eqref{eq:decompostion down alg 2} and $[g_{J,N}]=[g_J]$ in $\mathcal{O}_{X,x}/\widehat{m}^N$  that  $R_N\in \widehat{m}^{N}$. As $R_N$ is a polynomial, we know $R_N$ actually belong to $\widetilde{m}^N$.  And we also know that 
$$\delta f_0^l=\sum_{1\le j_1,j_2,\ldots,j_k \le r}(f_{j_1})^{m_{j_1}}(f_{j_2})^{m_{j_2}}\cdots (f_{j_k})^{m_{j_k}}g_{J,N}+R_N$$
holds in $\mathbb{C}[z_1,\ldots,z_n]/I$.
Note that $\nu(f_j^{m_j})\ge m_0$ for $1\le j\le r$, $v(\widetilde{m})>0$ and $k\coloneqq \big\lfloor\frac{\nu(f_0^l)+l\delta_2}{m_0}\big\rfloor-n-1$. When $N$ is big enough, we have
	\begin{equation}\nonumber
		\begin{split}
		\nu(\delta)+\nu(f_0^l)
		&\ge \big\{\min_{1\le j_1,j_2,\ldots,j_k \le r}\nu\big(f_{j_1}^{m_{j_1}}f_{j_2}^{m_{j_2}}\cdots f_{j_k}^{m_{j_k}}g_{J,N}\big), R_N\big\}\\
			&\ge \big\{\min_{1\le j_1,j_2,\ldots,j_k \le r}\nu\big(f_{j_1}^{m_{j_1}}f_{j_2}^{m_{j_2}}\cdots f_{j_k}^{m_{j_k}}\big), R_N\big\}\\
		&\ge \min_{1\le j_1,j_2,\ldots,j_k \le r}\nu\big(f_{j_1}^{m_{j_1}}f_{j_2}^{m_{j_2}}\cdots f_{j_k}^{m_{j_k}}\big)\\
	&\ge m_0k\\
			&\ge m_0\Big(\frac{\nu(f_0^l)+l\delta_2}{m_0}-n-2\Big)\\
			&=\nu(f_0^l)+l\delta_2-m_0(n+2),
		\end{split}
	\end{equation}
	which implies $l\le\frac{m_0(n+2)+\nu(\delta)}{\delta_2}$. This is a contradiction since we can choose $l$ arbitrarily large. Then we must have $\sigma(\log|f_0|,\varphi)\le \nu(f_0)$.
\end{proof}

To prove Corollary \ref{C:interpolation-comp-poly2}, we also recall the following Hilbert's Nullstellensatz theorem for finitely generated $\mathbb{C}$-algebras. 
For example, the following Lemma can be referred to Theorem $5.5$ in \cite{Matsumura}.
\begin{Lemma}[see \cite{Matsumura}]
\label{lem: Hilbert's Nullstellensatz theorem alg}
Let $A$ be a finitely generated $\mathbb{C}$-algebra and $J\subset A$ be a proper ideal. Then one has $\sqrt{J}=\bigcap_{J\subset m}m$, where the intersection is over all maximal ideals $m$ containing $J$.
\end{Lemma}

Let $A=\mathbb{C}[z_1,\ldots,z_n]/I$, where $I$ is an ideal in $\mathbb{C}[z_1,\ldots,z_n]$. Denote $X=V(I)$ be the zero set of $I$. Let $J\subset A$ be a proper ideal. Define $V_X(J)\coloneqq \{x\in X \mid P(x)=0,\ \text{for any }P\in J\}$. For any subset $Y\subset X$, denote $I(Y)\coloneqq \{P\in A \mid P(x)=0\ \text{for any }x\in Y\}$.

Using Lemma \ref{lem: Hilbert's Nullstellensatz theorem alg}, we immediately have
\begin{Lemma}
\label{lem: Hilbert's Nullstellensatz theorem geo}
Let $A=\mathbb{C}[z_1,\ldots,z_n]/I$, where $I$ is an ideal in $\mathbb{C}[z_1,\ldots,z_n]$ and $J\subset A$ be a proper ideal. Then one has $I\big(V_X(J)\big)=\sqrt{J}.$
\end{Lemma}
For the convenience of the readers, we recall the proof of Lemma \ref{lem: Hilbert's Nullstellensatz theorem geo}.
\begin{proof}[Proof of Lemma \ref{lem: Hilbert's Nullstellensatz theorem geo}]
It is obvious that $\sqrt{J}\subset I\big(V_X(J)\big)$. If $m$ is an maximal ideal in $A$, then $m=\sqrt{m}\subset I\big(V(m)\big)=m$, where the last $``="$ holds since $m$ is maximal.

Note that $J\subset \sqrt{J}=\bigcap_{J\subset m}m$. For any maximal ideal $m$ containing $J$, we have $V(J)\supset V(m)$. Hence $I\big(V_X(J)\big)\subset I\big(V(m)\big)=m$. By the arbitrariness of $m$ and Lemma \ref{lem: Hilbert's Nullstellensatz theorem alg}, we have $I\big(V_X(J)\big)=\bigcap_{J\subset m}m=\sqrt{J}$. This completes the proof.
\end{proof}

Denote the set of all germs of real analytic functions near the origin $o'\in\mathbb{R}^n$ by $C^{\text{an}}_{o'}$. There exists an injective ring homomorphism $P\colon C^{\text{an}}_{o'}\rightarrow\mathcal{O}_o$, which satisfies
$$P\Big(\sum_{\alpha\in\mathbb{Z}_{\ge0}^n} a_{\alpha}x^{\alpha}\Big)=\sum_{\alpha\in\mathbb{Z}_{\ge0}^n}a_{\alpha}z^{\alpha},$$
where $(x_1,\ldots,x_n)$ and $(z_1,\ldots,z_n)$ are the standard coordinates in $\mathbb{R}^n$ and $\mathbb{C}^n$ respectively, and $\sum_{\alpha\in\mathbb{Z}_{\ge0}^n}a_{\alpha}x^{\alpha}$ is the power series expansion of arbitrary real analytic function near $o'$. 
Assume that $I$ is an ideal in $C^{\text{an}}_{o'}$. Denote $\big(P(I)\big)$ be the ideal generated by $P(I)$ in $\mathcal{O}_o$.
Note that $P$ is obvious injective. Hence we have
\begin{Remark}
\label{rem:prime ideal from complex to real}
If $\big(P(I)\big)$ is prime ideal in $\mathcal{O}_o$, then $I$ is prime in $C^{\text{an}}_{o'}$.
\end{Remark}
\begin{proof}
Let $ab\in I$. Then $P(ab)=P(a)P(b)\in \big(P(I)\big)$. As $\big(P(I)\big)$ is prime ideal, we may assume that $P(a)\in \big(P(I)\big)$. As $P$ is injective, we know that $a\in I$.
\end{proof}

Denote the elemens in $C^{\text{an}}_{o'}/I$ and $\mathcal{O}_o/\big(P(I)\big)$ by $[f]$ and $\widetilde{g}$ for any $f\in C^{\text{an}}_{o'}$ and $g\in \mathcal{O}_o$ respectively. 
Then we have an induced ring homomorphism  $\widetilde{P}:C^{\text{an}}_{o'}/I\rightarrow \mathcal{O}_o/\big(P(I)\big)$ defined by $\widetilde{P}([f])=\widetilde{P(f)}$. It is easy to check that $\widetilde{P}$ is well defined.
Denote $X$ be the zero variety defined by  $\big(P(I)\big)$.

We have following properties of the homomorphism $\widetilde{P}$.
\begin{Remark}
\label{rem:injective}
$\widetilde{P}$ is injective.
\end{Remark}
\begin{proof}
    Assume $[F_{1}],[F_{2}]\in C^{\text{an}}_{o'}/I$ such that $\widetilde{P}([F_1])=\widetilde{P}([F_2])$ in $\mathcal{O}_o/\big(P(I)\big) $.
    Then there exist $g_k\in \mathcal{O}_o$ ($k=1,\ldots,m$) such that
    $P(F_1)-P(F_2)=\sum_{k=1}^m g_k P(f_k)$. For any $g\in \mathcal{O}_o$, $g$ can be written as a convergent power series $g=\sum_{\alpha}a_{\alpha}z_1^{\alpha_1}z_2^{\alpha_1}\cdots z_n^{\alpha_n}$, where $\alpha$ is a multi-index, $a_{\alpha}\in \mathbb{C}$ and $(z_1,\ldots,z_n)$ is the coordinate near $o$. We define $$g_{1,\text{Re}}\coloneqq \sum_{\alpha}\text{Re}(a_{\alpha})z_1^{\alpha_1}z_2^{\alpha_1}\cdots z_n^{\alpha_n}$$
    and 
    $$g_{1,\text{Im}}\coloneqq \sum_{\alpha}\text{Im}(a_{\alpha})z_1^{\alpha_1}z_2^{\alpha_1}\cdots z_n^{\alpha_n}.$$
    
    Hence, 
    We have
    $$P(F_1)-P(F_2)=\sum_{k=1}^m \big(P(g_{k,\text{Re}})+\i P(g_{k,\text{Im}})\big) P(f_k),$$
    which implies that 
     $$P(F_1)-P(F_2)=\sum_{k=1}^m P(g_{k,\text{Re}}) P(f_k).$$
    As $P$ is injective, one has $F_1-F_2\in I$ and hence $[F_1]=[F_2]$ in $C^{\text{an}}_{o'}/I$. This means $\widetilde{P}$ is injective.
\end{proof}

\begin{Remark}
\label{rem:decom real}
For any $\widetilde{g}\in \mathcal{O}_o/\big(P(I)\big)$, there exists a unique pair of $[g_{\text{Re}}],[g_{\text{Im}}]\in C^{\text{an}}_{o'}/I$  such that
 \[\widetilde{g}=\widetilde{P}([g_{\text{Re}}])+\i \widetilde{P}([g_{\text{Im}}]) \]
\end{Remark}
\begin{proof}
    Let $g_1,g_2$ belong to $\mathcal{O}_o$ such that $g_1-g_2\in \big(P(I)\big)$ and $\widetilde{g}=\widetilde{g_1}=\widetilde{g_2}$ in $\mathcal{O}_o/\big(P(I)\big)$. Let $\{f_k\}_{i=1}^m$ be the generators of $I$ in $C^{\text{an}}_{o'}$. Then  $\big(P(I)\big)$ was generated by $P(f_k)$. We know that there exists $h_k\in \mathcal{O}_o$ ($k=1,\ldots,m$) such that $g_1-g_2=\sum_{k=1}^m h_kP(f_k)$. 
    Then we have
\begin{equation}
\label{eq:real 1}
    g_{1,\text{Re}}+\i g_{1,\text{Im}}-g_{2,\text{Re}}-\i g_{2,\text{Im}}=\sum_{k=1}^m (h_{k,\text{Re}}+\i h_{k,\text{Im}})P(f_k).
\end{equation}
Note that $P(f_k)$ has real coefficients as a power series of $z_1,\ldots,z_n$. Equality \eqref{eq:real 1} tells that
\begin{equation}
\label{eq:real 2}
\begin{split}
    g_{1,\text{Re}}-g_{2,\text{Re}}&=\sum_{k=1}^m h_{k,\text{Re}}P(f_k)\\
   g_{1,\text{Im}}-g_{2,\text{Im}}&=\sum_{k=1}^m h_{k,\text{Im}}P(f_k)
\end{split}
\end{equation}
Thus we have $g_{1,\text{Re}}-g_{2,\text{Re}}\in \big(P(I)\big)$ and $g_{1,\text{Im}}-g_{2,\text{Im}}\in \big(P(I)\big)$ hold. It follows from $P$ is injective and equalities \eqref{eq:real 2} that $[P^{-1}(g_{1,\text{Re}})]=[P^{-1}(g_{2,\text{Re}})]\in C^{\text{an}}_{o'}/I$ and $[P^{-1}(g_{1,\text{Im}})]=[P^{-1}(g_{2,\text{Im}})]\in C^{\text{an}}_{o'}/I$. Denote $[g_{\text{Re}}]=[P^{-1}(g_{1,\text{Re}})]$ and $[g_{\text{Im}}]=[P^{-1}(g_{1,\text{Im}})]$. Hence we  know there exist  pair $[g_{\text{Re}}],[g_{\text{Im}}]\in C^{\text{an}}_{o'}/I$ such that $\widetilde{g}=\widetilde{P}(g_{\text{Re}})+\i\widetilde{P}(g_{\text{Im}})$.

The uniqueness of the pair $[g_{\text{Re}}],[g_{\text{Im}}]$ follows from the fact $\widetilde{P}$ is injective, see Remark \ref{rem:injective}.
\end{proof}

Denote $X$ be the zero variety defined by  $\big(P(I)\big)$. Assume that $\big(P(I)\big)$ is a prime ideal in $\mathcal{O}_o$.
\begin{Lemma}
	\label{c:relativetype-real}
	Let $\nu$ be a valuation on $C^{\text{an}}_{o'}/I$.   For any  $[f_0], [f_1], \ldots, [f_r]$ in $C^{\text{an}}_{o'}/I$. satisfying $\nu([f_j])>0$ for $1\le j\le r$, we have
	$\sigma(\log|\widetilde{f_0}|,\varphi)\le \nu([f_0])$ holds near $o\in X$, where $\varphi\coloneqq \log\big(\sum_{1\le j\le r}|\widetilde{ f_j}|^{\frac{1}{\nu([f_j])}}\big)$ and $\widetilde {f_j}\coloneqq \widetilde{P}([f_j])$ for $0\le j\le r$.
\end{Lemma}
\begin{proof}
	We prove Lemma \ref{c:relativetype-real} by contradiction. Suppose to the contrary that the conclusion does not hold. There exists $\delta_1>0$ such that
	$\log|\widetilde{f}_0|\le (\nu([f_0])+\delta_1)\varphi+O(1)$ near $o$. Then we can find some $\delta_2>0$ and a set of positive integers $\{m_0,\ldots,m_r\}$ such that
	\[\log|\widetilde{f}_0|\le \frac{\nu([f_0])+\delta_2}{m_0}\widetilde\varphi+O(1) \quad  \ \text{near} \ o,\]
	where
	$\widetilde{\varphi}\coloneqq\log\big(\sum_{1\le j\le r}|\widetilde{f}_j^{m_j}|\big)$ and $\nu([f_j]^{m_j})\ge m_0$ for $1\le j\le r$. For any positive integer $l$, we have
	\[\log|\widetilde{f}_0^l|\le \frac{\nu(\widetilde{f}_0^l)+l\delta_2}{m_0}\widetilde\varphi+O(1)\]
	holds on $\big(U\cap X\big)\backslash X_{\sing}$, where $U\coloneqq \{x\in \mathbb{C}^n: \log|z|<\delta\}$.
	
	Following the argument and notation of Lemma \ref{l:relativetype}, we see that for sufficiently large $l$, by defining $k\coloneqq \big\lfloor\frac{\nu(f_0^l)+l\delta_2}{m_0}\big\rfloor-n-1$, we obtain that the following decomposition
		\begin{equation}
    \label{eq:decompostion down real case}
\widetilde{f_0}^l=\sum_{1\le j_1,j_2,\ldots,j_k \le r}(\widetilde{f}_{j_1})^{m_{j_1}}(\widetilde{f}_{j_2})^{m_{j_2}}\cdots (\widetilde{f}_{j_k})^{m_{j_k}}
\frac{g_J}{\delta}
\end{equation}
	holds on $U$, where $U$ is an open neighborhood of $o$ in $X$, and $g_J$ (depending on $l$) and $\delta$ belong to $\mathcal{O}_{X,o}=\mathcal{O}_{o}/\big(P(I)\big)$. Thus we have
		\begin{equation}
    \label{eq:decompostion down real case 2}
\delta \widetilde{f}_0^l=\sum_{1\le j_1,j_2,\ldots,j_k \le r}(\widetilde{f}_{j_1})^{m_{j_1}}(\widetilde{f}_{j_2})^{m_{j_2}}\cdots (\widetilde{f}_{j_k})^{m_{j_k}}
g_J. 
\end{equation}	
Combining with Remark \ref{rem:decom real}, we have
	\begin{equation}\nonumber
    \begin{split}
  &\big(\widetilde{P}([\delta_{\text{Re}}])+ \i\widetilde{P}([\delta_{\text{Im}}])\big)\widetilde{P}([f_0]^l)\\
  =&\sum_{1\le j_1,j_2,\ldots,j_k \le r}\widetilde{P}([f_{j_1}]^{m_{j_1}})\widetilde{P}([f_{j_2}]^{m_{j_2}})\cdots \widetilde{P}([f_{j_k}]^{m_{j_k}})     
\big(\widetilde{P}([g_{J,\text{Re}}])+ \i\widetilde{P}([g_{J,\text{Im}}])\big).
    \end{split}
\end{equation}		
	Hence we have
		\begin{equation}\nonumber
[\delta_{\text{Re}}][f_0]^l=\sum_{1\le j_1,j_2,\ldots,j_k \le r}[f_{j_1}]^{m_{j_1}}[f_{j_2}]^{m_{j_2}}\cdots [f_{j_k}]^{m_{j_k}}
[g_{J,\text{Re}}]
\end{equation}
	holds in $C^{\text{an}}_{o'}/I$. As $\nu([f_j]^{m_j})\ge m_0$ for $1\le j\le r$, we get
	\begin{equation}\nonumber
		\begin{split}
			\nu([f_0]^l)&\ge m_0\Big(\Big\lfloor\frac{\nu(f_0^l)+l\delta_2}{m_0}\Big\rfloor-n-1\Big)-\nu([\delta_{\text{Re}}])
			\\&\ge m_0\Big(\frac{\nu(f_0^l)+l\delta_2}{m_0}-n-2\Big)-\nu([\delta_{\text{Re}}])\\
			&=\nu(f_0^l)+l\delta_2-m_0(n+2)-\nu([\delta_{\text{Re}}]),
		\end{split}
	\end{equation}
	which implies $l\le\frac{m_0(n+2)+\nu([\delta_{\text{Re}}])}{\delta_2}$, contradiction. Then we have $\sigma(\log|\tilde f_0|,\varphi)\le \nu([f_0])$.
\end{proof}

Let $I$ be an ideal in $\mathbb{R}[x_1,\ldots,x_n]$.
Denote  the restriction of $P$ on $\mathbb{R}[x_1,\ldots,x_n]$ also by $P$ and $\widetilde{P}$ on $\mathbb{R}[x_1,\ldots,x_n]/I$ also by $\widetilde{P}$. Let $\big(P(I)\big)$ be the ideal in $\mathbb{C}[x_1,\ldots,x_n]$ generated   by $P(I)$.

Note that $P$ is obvious injective.
Similarly as Remark \ref{rem:prime ideal from complex to real}, we have 
\begin{Remark}
\label{rem:prime ideal from complex to real poly}
If $\big(P(I)\big)$ is prime ideal in $\mathbb{C}[x_1,\ldots,x_n]$, then $I$ is prime in $\mathbb{R}[x_1,\ldots,x_n]$.
\end{Remark}
\begin{proof}
Since the proof of Remark \ref{rem:prime ideal from complex to real poly} is the same as the proof of Remark \ref{rem:prime ideal from complex to real}, we omit the proof.
\end{proof}

We also have following remark, which is an analogy of Remark \ref{rem:injective} and Remark \ref{rem:decom real} in polynomial case.
\begin{Remark}
\label{rem:injective and decom poly case}
The map $\widetilde{P}:\mathbb{R}[x_1,\ldots,x_n]/I \to \mathbb{C}[x_1,\ldots,x_n]/\big(P(I)\big)$ is injective. 
For any $\widetilde{g}\in \mathbb{C}[x_1,\ldots,x_n]/\big(P(I)\big)$, there exists a unique pair of $[g_{\text{Re}}],[g_{\text{Im}}]\in \mathbb{R}[x_1,\ldots,x_n]/I$  such that
 \[\widetilde{g}=\widetilde{P}([g_{\text{Re}}])+\i \widetilde{P}([g_{\text{Im}}]) \]
\end{Remark}
\begin{proof}
The proof of the injectiveness is almost the same as the proof of Remark \ref{rem:injective}. The proof of  existence and uniqueness of the decomposition is almost the same as the proof of Remark \ref{rem:decom real}. So we omit the proof of Remark \ref{rem:injective and decom poly case}.
\end{proof}

Assume that $\big(P(I)\big)$ is a prime ideal in 
$\mathbb{C}[x_1,\ldots,x_n]$ and $(X,o)$ is irreducible as a germ of analytic set, where $o\in X$ and $X$ is the affine variety defined by $\big(P(I)\big)$.

\begin{Lemma}
	\label{c:relativetype-real poly}
	Let $\nu$ be a valuation on $\mathbb{R}[x_1,\ldots,x_n]/I$  satisfying $\nu([x_j])>0$ for any $1\le j\le n$. For any  $[f_0], [f_1], \ldots, [f_r]$ in $\mathbb{R}[x_1,\ldots,x_n]_{o'}/I$. satisfying $\nu([f_j])>0$ for $1\le j\le r$, we have
	$\sigma(\log|\widetilde f_0|,\varphi)\le \nu([f_0])$ holds near $o\in X$, where $\varphi\coloneqq \log\big(\sum_{1\le j\le r}|\widetilde f_j|^{\frac{1}{\nu([f_j])}}\big)$ and $\widetilde f_j\coloneqq \widetilde{P}([f_j])$ for $0\le j\le r$.
\end{Lemma}

\begin{proof}
	We prove this lemma by contradiction. Suppose to the contrary that the conclusion does not hold. There exists $\delta_1>0$ such that
	$\log|\widetilde{f}_0|\le (\nu([f_0])+\delta_1)\varphi+O(1)$ near $o$. Then we can find some $\delta_2>0$ and a set of positive integers $\{m_0,\ldots,m_r\}$ such that
	\[\log|\widetilde{f}_0|\le \frac{\nu([f_0])+\delta_2}{m_0}\widetilde\varphi+O(1) \quad  \ \text{near} \ o,\]
	where
	$\widetilde{\varphi}\coloneqq\log\big(\sum_{1\le j\le r}|\widetilde{f}_j^{m_j}|\big)$ and $\nu([f_j]^{m_j})\ge m_0$ for $1\le j\le r$. For any positive integer $l$, we have
	\[\log|\widetilde{f}_0^l|\le \frac{\nu(\widetilde{f}_0^l)+l\delta_2}{m_0}\widetilde\varphi+O(1)\]
	holds on $\big(U\cap X\big)\backslash X_{\sing}$, where $U\coloneqq \{x\in \mathbb{C}^n: \log|z|<\delta\}$.

	Following the argument and notation of Lemma \ref{l:relativetype}, we see that for sufficiently large $l$, by defining $k\coloneqq \big\lfloor\frac{\nu(f_0^l)+l\delta_2}{m_0}\big\rfloor-n-1$, we obtain that the following decomposition
\begin{equation}
    \label{eq:decompostion down real case  real poly}
\delta \widetilde{f}_0^l=\sum_{1\le j_1,j_2,\ldots,j_k \le r}(\widetilde{f}_{j_1})^{m_{j_1}}(\tilde{f}_{j_2})^{m_{j_2}}\cdots (\widetilde{f}_{j_k})^{m_{j_k}}
g_J. 
\end{equation}	
	holds on $U$, where $U$ is an open neighborhood of $o$ in $X$, $\delta$ is a polynomial with complex coefficients and $g_J$ (depending on $l$) belongs to $\mathcal{O}_{X,o}=\mathcal{O}_{o}/\big(P(I)\big)$.
	
Combining equality \eqref{eq:decompostion down real case  real poly} and Remark \ref{rem:injective and decom poly case}, we have
	\begin{equation}\nonumber
    \begin{split}
  &\big(\widetilde{P}([\delta_{\text{Re}}])+ \i\widetilde{P}([\delta_{\text{Im}}])\big)\widetilde{P}([f_0]^l)\\
  =&\sum_{1\le j_1,j_2,\ldots,j_k \le r}\widetilde{P}([f_{j_1}]^{m_{j_1}})\widetilde{P}([f_{j_2}]^{m_{j_2}})\cdots \widetilde{P}([f_{j_k}]^{m_{j_k}})\big(\widetilde{P}([g_{J,\text{Re}}])+
 \i\widetilde{P}([g_{J,\text{Im}}])\big)
    \end{split}
\end{equation}		
	Hence we have
		\begin{equation}
    \label{eq:decompostion down real case 4}
[\delta_{\text{Re}}][f_0]^l=\sum_{1\le j_1,j_2,\ldots,j_k \le r}[f_{j_1}]^{m_{j_1}}[f_{j_2}]^{m_{j_2}}\cdots [f_{j_k}]^{m_{j_k}}
[g_{J,\text{Re}}]
\end{equation}
holds in $C^{\text{an}}_{o'}/I$.

Let $m_P=(x_1,\ldots,x_n)$ be the ideal generated by $x_1,\ldots,x_n$ in $\mathbb{R}[x_1,\ldots,x_n]/I$ and $\widehat{m}=(x_1,\ldots,x_n)$ be the ideal generated by $x_1,\ldots,x_n$ in $C^{\text{an}}_{o'}/I$. For any integer $N>0$, let $g_{J,N}$ be the finite terms of the Taylor expansion of $g_{J,\text{Re}}$ such that   $[g_{J,N}]=[g_{J,\text{Re}}]$ in $C^{\text{an}}_{o'}/\widehat{m}^N$. Note that $g_{J,N}$ can be viewed as an element in $\mathbb{R}[x_1,\ldots,x_n]/I$. Denote
$$R_N\coloneqq [\delta_{\text{Re}}][f_0]^l-\sum_{1\le j_1,j_2,\ldots,j_k \le r}[f_{j_1}]^{m_{j_1}}[f_{j_2}]^{m_{j_2}}\cdots [f_{j_k}]^{m_{j_k}}
[g_{J,N}].$$
Then $R_N$ is a polynomial in $\mathbb{R}[x_1,\ldots,x_n]/I$.
It follows from equality \eqref{eq:decompostion down real case 4} and $[g_{J,N}]=[g_J]$ in $C^{\text{an}}_{o'}/\widehat{m}^N$  that  $R_N\in \widehat{m}^{N}$. As $R_N$ is a polynomial, we know $R_N$ actually belongs to $(m_P)^N$.  And we also know that 
$$[\delta_{\text{Re}}][f_0]^l=\sum_{1\le j_1,j_2,\ldots,j_k \le r}[f_{j_1}]^{m_{j_1}}[f_{j_2}]^{m_{j_2}}\cdots [f_{j_k}]^{m_{j_k}}
[g_{J,N}]+R_N$$
holds in $\mathbb{R}[x_1,\ldots,x_n]/I$.
Note that $\nu([f_j]^{m_j})\ge m_0$ for $1\le j\le r$, $v(m_P)>0$ and $k\coloneqq \big\lfloor\frac{\nu(f_0^l)+l\delta_2}{m_0}\big\rfloor-n-1$. When $N$ is large enough, we have
	\begin{equation}\nonumber
		\begin{split}
		\nu([\delta_{\text{Re}}])+\nu([f_0]^l)
		&\ge \big\{\min_{1\le j_1,j_2,\ldots,j_k \le r}\nu\big([f_{j_1}]^{m_{j_1}}[f_{j_2}]^{m_{j_2}}\cdots [f_{j_k}]^{m_{j_k}}[g_{J,N}]\big), R_N\big\}\\
			&\ge \big\{\min_{1\le j_1,j_2,\ldots,j_k \le r}\nu\big([f_{j_1}]^{m_{j_1}}[f_{j_2}]^{m_{j_2}}\cdots [f_{j_k}]^{m_{j_k}}\big), R_N\big\}\\
		&\ge \min_{1\le j_1,j_2,\ldots,j_k \le r}\nu\big([f_{j_1}]^{m_{j_1}}[f_{j_2}]^{m_{j_2}}\cdots [f_{j_k}]^{m_{j_k}}\big)\\
	&\ge m_0k\\
			&\ge m_0\Big(\frac{\nu(f_0^l)+l\delta_2}{m_0}-n-2\Big)\\
			&=\nu(f_0^l)+l\delta_2-m_0(n+2),
		\end{split}
	\end{equation}
	which implies $l\le\frac{m_0(n+2)+\nu([\delta_{\text{Re}}])}{\delta_2}$. This is a contradiction since we can choose $l$ arbitrarily large. Then we must have $\sigma(\log|f_0|,\varphi)\le \nu(f_0)$.	
\end{proof}

We also note that we have the following basic result.
\begin{Remark}
\label{rem:subring real case}
$\mathbb{R}[z_1,...,z_n]/I$ is a subring of $C^{\text{an}}_{o'}/\big(I\cdot C^{\text{an}}_{o'}\big)$.
\end{Remark}
\begin{proof}
The proof of Remark \ref{rem:subring real case} is almost the same as the proof of Remark \ref{rem:subring complex case}. So we omit the proof.
\end{proof}

\subsection{Jumping number in the singular case}

In this section, we discuss the properties of the jumping number $c_o^f(\varphi)$, where $o\in X$ can be a singular point of an irreducible analytic set $X$.

\begin{Definition}
    \label{def:jumping number in singular case}
    Let $X\subset \mathbb{C}^n$ be an  analytic set and $K\subset X$ a compact subset. For any plurisubharmonic function $\varphi$ defined on $K$ and any weakly holomorphic function $f$ defined on $K$, denote the \emph{jumping number} $c_K^f(\varphi)$ by
	\begin{equation*}
	   \begin{split}
	     c_{K}^f(\varphi)\coloneqq \sup\Big\{c>0\colon
	     &\int_{U\cap X_{\reg}}|f|^2e^{-2c\varphi}dV_X<+\infty,\\
	     &\text{for some open neighborhood}\  U\  \text{of}\ K \ \text{in} \ X \Big\}.  
	   \end{split}
	\end{equation*}

\end{Definition}
When $K=\{o\}$ is a singular point of $X$, $c_{o}^f(\varphi)$  generalizes the ordinary jumping number to the singular case. We resolve the singularities of $(X,o)$ and consider the jumping number $c_o^f(\varphi)$ after a sequence of blow-ups.

\begin{Remark}
\label{rem:jumping number invariant}
Using Theorem \ref{thm:desing}, we can resolve
	the singularities of $X$, and denote the corresponding proper modification by $\mu\colon\widetilde\Omega\rightarrow\Omega$, where $\Omega\subset \mathbb{C}^{n}$ open such that $o\in X\cap \Omega$. Denote the strict transform of $X$ by $\widetilde X$ and denote $Z_0\coloneqq \widetilde{X}\cap\mu^{-1}(\{o\})$. Note that $f\circ \mu$ and $\varphi\circ \mu$ are well defined holomorphic function and plurisubharmonic function correspondingly on $\pi^{-1}(U)\cap \widetilde{X}$ for any open neighborhood $U$ of $o$ in $X$ respectively. Since $\mu$ is proper, $Z_0$ is compact in $\widetilde{X}$. Note that $\mu$ is biholomorphic outside an analytic subset, we know that $c_o^f(\varphi)=c_{Z_0}^{\det \mu\cdot f\circ\mu}(\varphi\circ\mu),$
	where $\det \mu$ is the determinant of the Jacobian of the modification $\mu$ (which is a holomorphic function on $\widetilde\Omega$) and $c_{Z_0}^{\det \mu \cdot f\circ\mu}(\varphi\circ\mu)$ is the usual jumping number on a compact set in the regular case.
\end{Remark}

Recall that $X$ is an analytic subset of $\Omega\in\mathbb{C}^{n}$ with pure dimension $d$ and denote  $|z|^2\coloneqq \sum_{1\le j\le n}|z_j|^2$. Assume that $(X,o)$ is singular, we have the following approximation result for jumping numbers.

\begin{Lemma}\label{l:0921-1}
	Let $\varphi$ be a plurisubharmonic function near $o$, where $c>0$ is a constant and $f_j\in \mathcal{O}^{w}_{X,o}$ for $1\le j\le m$.  Then
	\[\lim_{N\rightarrow+\infty}c_o^f\big(\max\{\varphi,N\log|z|\}\big)=c_o^{f}(\varphi)\]
	for any holomorphic function $f\in \mathcal{O}^*_{X,o}$.
\end{Lemma}

\begin{proof}
By definition, we have $c_{o}^{f}\big(\max\{\varphi,N\log|z|\}\big)\ge c_{o}^{f}\big(\varphi\big)$.

Using Theorem \ref{thm:desing}, we can resolve the singularity of $X$ and get a proper holomorphic map $\mu\colon\widetilde {\Omega}\rightarrow \Omega$ where $\Omega\subset \mathbb{C}^{n}$ is an open subset such that $o\in X\cap \Omega$. Remark \ref{rem:jumping number invariant} tells that $c_o^{f}(\varphi)=c_{Z_0}^{\det \mu\cdot f\circ\mu}(\varphi\circ\mu)$ and $c_o^{f}\big(\max\{\varphi,N\log|z|\}\big)=c_{Z_0}^{\det \mu \cdot f\circ\mu}\big(\max\{\varphi\circ\mu,N\log|z\circ\mu|\}\big)$. Note that $\big(N\log|z\circ\mu|\big)|_{\widetilde{X}}$ is a plurisubharmonic function with $\{\big(N\log|z\circ\mu|\big)|_{\widetilde{X}}=-\infty\}=Z_0$.

Assume that $c_{Z_0}^{\det \mu \cdot f\circ\mu}(\varphi\circ\mu)<1$.
Note that $c_{Z_0}^{\det \mu \cdot  f\circ\mu}(\varphi\circ\mu)=\inf_{z\in Z_0} c_{z}^{\det \mu \cdot f\circ\mu}(\varphi\circ\mu)$. We know there exists a point $p\in Z_0$, such that $c_{p}^{\det \mu \cdot f\circ\mu}(\varphi\circ\mu)<1$.
As $(\widetilde{X},p)$ is regular, we can find a local coordinate neighborhood $(W_p,w)$ centered at $p$ such that there exists an $A>0$ such that $\big(\log|z\circ\mu|\big)|_{\widetilde{X}}\le A\log|w|+O(1)$ on $W_p$. Then $c_{p}^{\det \mu \cdot f\circ\mu}\big(\max\{\varphi\circ\mu,N\log|z\circ\mu|\}\big)\le c_{p}^{\det \mu \cdot f\circ\mu}\big(\max\{\varphi\circ\mu,NA\log|w|\}\big)$.

The proof in the smooth case (see \cite{BGMY-valuation}) shows that, $c_{p}^{\det \mu \cdot f\circ\mu}(\varphi\circ\mu)<1$ implies there exists $N\gg0$ such that
	$$c_{p}^{\det \mu \cdot f\circ\mu}\big(\max\{\varphi\circ\mu,NA\log|w|\}\big)\le 1.$$
Hence $c_{p}^{\det \mu \cdot f\circ\mu}\big(\max\{\varphi\circ\mu,N\log|z\circ\mu|\}\big)\le 1$, which implies 
$$c_{Z_0}^{\det \mu \cdot f\circ\mu}\big(\max\{\varphi\circ\mu,N\log|z\circ\mu|\}\big)\le 1.$$
Thus we know $\lim\limits_{N\to +\infty}c_{Z_0}^{\det \mu \cdot f\circ\mu}\big(\max\{\varphi\circ\mu,N\log|z\circ\mu|\}\big)\le c_{Z_0}^{\det \mu \cdot f\circ\mu_1}(\varphi\circ\mu),$ which implies $\lim\limits_{N\to +\infty}c_{o}^{f}\big(\max\{\varphi,N\log|z|\}\big)\le c_{o}^{f}\big(\varphi\big).$
\end{proof}

We also need the following definition of relative types on compact subsets of complex subvarieties.

\begin{Definition}
    \label{def: relative type in singular case}
    For plurisubharmonic functions $\psi,\varphi$ defined on a neighborhood of compact subset $K\subset X$, define the relative type on $K$ by
    \begin{equation*}
        \begin{split}
            \sigma_{K}(\psi,\varphi)\coloneqq \sup\{c\ge0 \colon &\psi\le c\varphi+O(1)\ \text{holds} \ on\  U\backslash X_{\sing}\\
            & \text{for some open neighborhood}\  U\  \text{of}\ K \ \text{in} \ X \}.
        \end{split}
    \end{equation*}
    We may omit $K$ if there is no confusion.
\end{Definition}
\begin{Remark}
\label{rem:relative type on single point}
When $K=\{z_0\}$ is a single point, we have the relative type $\sigma_{z_0}(\psi,\varphi)$ at a point $z_0$. Moreover, we have $\sigma_{K}(\psi,\varphi)= \inf_{p\in K}\sigma_{p}(\psi,\varphi)$.
\end{Remark}
\begin{proof}
    By definition, it is easy to verify that
$\sigma_{K}(\psi,\varphi)\le \inf_{p\in K}\sigma_{p}(\psi,\varphi)$.

Denote $c_1=\inf_{p\in K}\sigma_{p}(\psi,\varphi)$. Let $\delta>0$ be any small real number. Then for any $p\in K$, there exists an open neighborhood $U_p$ of $p$ in $X$ such that $\psi\le (c_1-\delta)\varphi+O(1)$ on $U_{p}\setminus X_{\sing}$. By the compactness of $K$, we know that there exist finite points $\{p_i\}_{i=1}^k$ such that $\cup_{i=1}^k U_{p_i}$ covers $K$. Thus $\psi\le (c_1-\delta)\varphi+O(1)$ on  $\cup_{i=1}^k U_{p_i}\setminus X_{\sing}$. By the arbitrariness of $\delta$, we have $\sigma_{K}(\psi,\varphi)\ge \inf_{p\in K}\sigma_{p}(\psi,\varphi)$.

Thus, we know $\sigma_{K}(\psi,\varphi)= \inf_{p\in K}\sigma_{p}(\psi,\varphi)$.
\end{proof}

 Let $f_j,h_k$ be holomorphic functions defined on a open set $U\subset X$ for any $1\le j\le m$ and any $1\le k\le q$. Let $\varphi_N\coloneqq c\log\Big(\sum_{1\le j\le m}|f_{j}|^{\frac{1}{a_j}}+\sum_{1\le k\le q}|h_{k}|^{N}\Big)$, where $c>0$, $N$ is a positive integer and $\{a_j\}_{1\le j\le m}$ are positive numbers. Let $K\subset U\subset X$ be a compact subset of $X$.  We have the following result for relative types $\sigma(\cdot,\varphi_N)$.

\begin{Lemma}
\label{lem:relative types for analytic singularity}
Let $f$ and $g$ be holomorphic functions defined on $K$. Let $z_0\in K$ be any point such that $z_0\in \{\varphi_N=-\infty\}$. Then there exists a constant $C>0$ (relies on $g$, $\{f_j\}$ and $\{h_k\}$ and is independent of $N$, $f$ and $z_0$) such that
$$\sigma_{z_0}(\log|fg|,\varphi_N)\le \sigma_{z_0}(\log|f|,\varphi_N)+C.$$
\end{Lemma}
\begin{proof}
Using Theorem \ref{thm:desing} and blow-ups, we can resolve
	the singularities of $\big(\cup_{i=1}^{m}\{f_j=0\}\big)\cup\big( \cup_{k=1}^{q}\{h_k=0\}\big)$, and denote the corresponding proper modification by $\mu\colon \widetilde{X}\rightarrow{X}$. Denote $Z_0\coloneqq \mu^{-1}(z_0)$. Note that $\sigma_{z_0}(\log|f|,\varphi_N)=\sigma_{Z_0}(\log|f\circ \mu|,\varphi_N\circ\mu)$ for any $f\in\mathcal{O}_{X,o}$. Remark \ref{rem:relative type on single point} tells that $\sigma_{Z_0}(\log|f\circ \mu|,\varphi_N\circ\mu)=\inf_{p\in Z_0}\sigma_{p}(\log|f\circ \mu|,\varphi_N\circ\mu)$.
	
   Denote  $\{f_j\circ\mu=0\}\cap \mu^{-1}(K)=\sum_{\alpha_j}a_{j,\alpha_j}D_{j,\alpha_j}$ and $\{h_k\circ\mu=0\}\cap \mu^{-1}(K)=\sum_{\beta_{k}}b_{k,\beta_k}D_{k,\beta_k}$ for any $1\le j\le m$, $1\le k\le q$ where $a_{j,\alpha_j},b_{k,\beta_k}$ are positive integers. Note that the divisor $D_{j,\alpha_j}$ and $D_{k,\beta_k}$ can be the same prime divisors.
	By the construction of $\mu$, we know that all $D_{j,\alpha_j}$ and $D_{k,\beta_k}$ are simple normal crossing divisors on $\widetilde{X}$. Denote $D=\big(\cup_{i=1}^{m}\{f_j\circ \mu=0\}\big)\cup\big( \cup_{k=1}^{q}\{h_k\circ \mu=0\}\big)$ be the support of all  divisors appearing in $\{f_j\circ \mu=0\}$ and $\{h_k\circ \mu=0\}$.
	
	 Let $p\in Z_0$. Let $U$ be a small neighborhood of $p$ such that $U\cap D=\cup_{1\le l \le n_p}\{w_l=0\}=\cup_{1\le l \le n_p}D_l$, where $w=(w_1,\ldots,w_n)$ is the coordinate on $U$ centered at $p$ and $n_p\le n$ is a positive integer. It follows from the construction of $\mu$ that  one has
	 \begin{equation*}
	     \begin{split}
	       \varphi_N\circ \mu
	       &=c\log\Big(\sum_{1\le j\le m}|f_{j}\circ\mu|^{\frac{1}{a_j}}+\sum_{1\le k\le q}|h_{k}\circ\mu|^{N}\Big)  \\
	       &=c\log\Big(\sum_{1\le j\le m}|w_{j,\alpha_j}|^{\frac{a_{j,\alpha_j}}{a_j}}+\sum_{1\le k\le q}|w_{k,\beta_k}|^{Nb_{k,\beta_k}}\Big)\\
	       &=c\log\big(\prod_{1\le l\le n_p}|w_l|^{\tau_{l,N}}\big)+O(1),
	     \end{split}
	 \end{equation*}
	 where $\tau_{l,N}=\min_{(j,\alpha_j),(k,\beta_k)}\{\frac{a_{j,\alpha_j}}{a_j},Nb_{k,\beta_k}\}$ and the minimum is taken among all indices $(j,\alpha_j)$ and $(k,\beta_k)$ satisfying $D_l=D_{j,\alpha_j}$ and $D_l=D_{b_k,\beta_k}$ respectively. 
	 For any $F\in  \mathcal{O}_{\widetilde X,p}$, we have
	 \begin{equation}\nonumber
	     \sigma_p(\log|F|,\varphi_N)=\min_{1\le l \le n_p}\big\{\frac{\ord_{D_l}(F)}{c\tau_{l,N}}\big\}.
	 \end{equation}
	Hence 
	\begin{equation}
	\label{eq:estimate for relative type 1}
	\begin{split}
	\sigma_p(\log|fg\circ\mu|,\varphi_N)
	&=\min_{1\le l \le n_p}\big\{\frac{\ord_{D_l}(f g\circ\mu)}{c\tau_{l,N}}\big\}   \\ 
	&=\min_{1\le l \le n_p}\big\{\frac{\ord_{D_l}(f\circ\mu)+\ord_{D_l}(g\circ\mu)}{c\tau_{l,N}}\big\} \\
	&\le \min_{1\le l \le n_p}\big\{\frac{\ord_{D_l}(f\circ\mu)}{c\tau_{l,N}}\big\}+\max_{1\le l \le n_p}\big\{\frac{\ord_{D_l}(g\circ\mu)}{c\tau_{l,N}}\big\}\\
	&=\sigma_p(\log|f\circ\mu|,\varphi_N)+\max_{1\le l \le n_p}\big\{\frac{\ord_{D_l}(g\circ\mu)}{c\tau_{l,N}}\big\}
	\end{split}
	\end{equation}
Note that $\tau_{l,N}$ is increasing with respect to $N$ and thus $\max_{1\le l \le n_p}\big\{\frac{\ord_{D_l}(g\circ \mu)}{c\tau_{l,N}}\big\}$ is decreasing with respect to $N$. 
	Combining with the estimate \eqref{eq:estimate for relative type 1}, we have
	\begin{equation}
	\label{eq:estimate for relative type 2}
	\begin{split}
	\sigma_p(\log|fg\circ\mu|,\varphi_N)
	&=\sigma_p(\log|f\circ\mu|,\varphi_N)+\max_{1\le l \le n_p}\big\{\frac{\ord_{D_l}(g\circ\mu)}{c\tau_{l,N}}\big\}\\
	&\le\sigma_p(\log|f\circ\mu|,\varphi_N)+\max_{1\le l \le n_p}\big\{\frac{\ord_{D_l}(g\circ\mu)}{c\tau_{l,1}}\big\}
	\end{split}
	\end{equation}	
Set $C\coloneqq \max_{D_{\gamma}\subset D}\big\{\frac{\ord_{D_{\gamma}}(g\circ \mu)}{c\tau_{\gamma,1}}\big\}$, where $D_{\gamma}$ is any prime divisor contained in $D$. Note that $D$ contains finite prime divisors (depending on  $\{f_j\}$ and $\{h_k\}$) and hence $C$ is a finite number only depending on $g$, $\{f_j\}$ and $\{h_k\}$ and independent of $N$, $f$, $p$ and $z_0\in K$. It follows from 
	inequality \eqref{eq:estimate for relative type 2} and the definition of $C$ that we have
	$$\inf_{p\in Z_0} \sigma_p(\log|(f\circ\mu)\cdot (g\circ\mu)|,\varphi_N\circ\mu)\le
	\inf_{p\in Z_0} \sigma_p(\log|f\circ\mu|,\varphi_N\circ\mu)+C.$$
Hence we know there exists a constant $C>0$ (independent of $N$ and $K$) such that
$$\sigma_{z_0}(\log|fg|,\varphi_N)\le \sigma_{z_0}(\log|f|,\varphi_N)+C$$	 
holds for any positive $N$ and $z_0\in K$.
\end{proof}

When $o$ is a smooth point of $X$, we have the following relation between the jumping number and the relative type of H\"older continuous plurisubharmonic functions.

\begin{Lemma}[see \cite{BFJ08}]
	\label{l:holder smooth}
	Let $\varphi$ be a plurisubharmonic function near $o\in \mathbb{C}^n$ such that $e^{\varphi}$ is $\alpha$-H\"older for some $\alpha>0$. Then for any $(f,o)\in\mathcal{O}_o^*$, we have 
	\[c_o^{f}(\varphi)\le\sigma(\log|f|,\varphi)+\frac{n}{\alpha}.\]
\end{Lemma}

 Recall that $X$ is an analytic subset of $\Omega\in\mathbb{C}^{n}$ with pure dimension $d$ and denote  $|z|^2\coloneqq \sum_{1\le j\le n}|z_j|^2$. Assume that $(X,o)$ is singular, we have the following relation  between the jumping number and the relative type when $\varphi$ has analytic singularity.

\begin{Lemma}
	\label{l:holder singular}
	Let $\varphi_0\coloneqq c\log\Big(\sum_{1\le j\le m}|f_{j}|^{\frac{1}{a_j}}\Big)$ and $\varphi_N\coloneqq c\log\Big(\sum_{1\le j\le m}|f_{j}|^{\frac{1}{a_j}}+|z|^{2N}\Big)$ near $o$, where $c>0$ is a constant, $\{a_j\}_{1\le j\le m}$ are positive numbers and $f_j\in \mathcal{O}^{w}_{X,o}$ for $1\le j\le m$. Then there exists a constant $C$ (independent of $N$) such that for any $(f,o)\in\mathcal{O}^{w}_{X,o}\backslash{\{0\}}$ and $N\ge 0$, we have
	\[c_o^{f}(\varphi_N)\le\sigma_{o}(\log|f|,\varphi_N)+C.\]
\end{Lemma}
\begin{proof}
Using Theorem \ref{thm:desing}, we can resolve
	the singularities of $X$, and denote the corresponding proper modification by $\mu\colon \widetilde\Omega\rightarrow\Omega$, where $\Omega\subset \mathbb{C}^{n}$ such that $o\in X\cap \Omega$. Denote the strict transform of $X$ by $\widetilde X$ and denote $Z_0\coloneqq \widetilde{X}\cap\mu^{-1}(\{o\})$. Note that $f\circ \mu$ and $\varphi_N\circ \mu$ are well defined holomorphic function and plurisubharmonic function correspondingly on $\pi^{-1}(U)\cap \widetilde{X}$ for any open neighborhood $U$ of $o$ in $X$ respectively. Since $\mu$ is proper, $Z_0$ is compact in $\widetilde{X}$. Note that $\mu$ is holomorphic and biholomorphic outside an analytic subset. We know that $c_o^f(\varphi_N)=c_{Z_0}^{\det \mu \cdot f\circ\mu}(\varphi\circ\mu)$ and $\sigma_{o}(\log|f|,\varphi_N)=\sigma_{Z_0}(\log|f\circ \mu_1|,\varphi\circ\mu_1)$.
	
	By adjusting $c$, we may assume all $\frac{1}{a_j}\ge 2$ and then $e^{\varphi_N}$ is $1$-H\"older continuous. For any $p\in Z_0$, Lemma \ref{l:holder smooth} tells that 
	$$c_{p}^{\det \mu \cdot f\circ\mu}(\varphi_N\circ\mu)\le \sigma_{p}(\log|\det \mu \cdot f\circ \mu_1|,\varphi_N\circ\mu)+C_1,$$
	where $C_1$ is constant only depending on $c$ and independent of $p$ and $N$. Lemma \ref{lem:relative types for analytic singularity} tells that, for any $p\in Z_0$, there exists a constant $C_2>0$ only depending on $\{f_j\circ \mu\}$, $\{z\circ \mu\}$ and $\mu$ and independent of $p$ and $N$ such that, for any $N>0$, $$\sigma_{p}(\log|\det \mu \cdot f\circ \mu|,\varphi_N\circ\mu)\le \sigma_{p}(\log|f\circ \mu|,\varphi_N\circ\mu)+C_2.$$
	Thus, for any $p\in Z_0$ and $N>0$, we have
	$$c_{p}^{\det \mu \cdot f\circ\mu}(\varphi_N\circ\mu)\le \sigma_{p}(\log|f\circ \mu|,\varphi_N\circ\mu)+C,$$
	where $C\coloneqq C_1+C_2>0$ is a constant independent of $p$ and $N$.
	
	Recall $c_{Z_0}^{\det \mu\cdot f\circ\mu}(\varphi_N\circ\mu)=\inf_{p\in Z_0}c_{p}^{\det \mu\cdot f\circ\mu}(\varphi_N\circ\mu)$ and $\sigma_{Z_0}(\log|f\circ \mu|,\varphi_N\circ\mu)=\inf_{p\in Z_0}\sigma_{p}(\log|f\circ \mu|,\varphi_N\circ\mu)$. Note that  $c_o^f(\varphi_N)=c_{Z_0}^{\det \mu \cdot f\circ\mu}(\varphi\circ\mu)$ and $\sigma_{o}(\log|f|,\varphi_N)=\sigma_{Z_0}(\log|f\circ \mu_1|,\varphi\circ\mu_1)$. We know there exists a constant $C$ (independent of $N$) such that for any $(f,o)\in\mathcal{O}^{w}_{X,o}\backslash{\{0\}}$ and $N\ge 0$,
	\begin{equation*}
	    \begin{split}
	   c_o^{f}(\varphi_N)&=c_{Z_0}^{\det \mu\cdot f\circ\mu}(\varphi_N\circ\mu)\\
	  & \le \sigma_{Z_0}(\log|f\circ \mu|,\varphi_N\circ\mu)+C\\   
	   &=\sigma_{o}(\log|f|,\varphi_N)+C
	    \end{split}
	\end{equation*}
The proof is complete.
\end{proof}

\begin{Remark}\label{r:tame}Let $\varphi$ be a plurisubharmonic function near $o$ satisfying that 
 $$\sup_{(f,o)\in\mathcal{O}^{w}_{X,o}\backslash{\{0\}}}\big(c_o^{f}(\varphi)-\sigma(\log|f|,\varphi)\big)<+\infty.$$ Then for any $(f,o)\in\mathcal{O}^{w}_{X,o}\backslash{\{0\}}$ with $c_o^f(\varphi)<+\infty$, considering the Tian function $\Tn(t)\coloneqq \{c\in\mathbb{R}:|f|^{2t}e^{-2c\varphi}$ is integrable near $o\}$, we have 
$$\lim_{t\rightarrow+\infty}\frac{\Tn(t)}{t}=\Tn'_-(t)=\Tn'_+(t)=\sigma(\log|f|,\varphi),$$
where $\Tn'_-(t)$ and $\Tn'_+(t)$ are the left and right derivatives of function $\Tn(t)$, respectively.
\end{Remark}
\begin{proof}
By the definitions of $\Tn(t)$ and $\sigma(\log|f|,\varphi)$, we have 
$\Tn'_+(t)\ge \sigma(\log|f|,\varphi)$. It follows from the concavity of $\Tn(t)$ that $\frac{\Tn(t)}{t}$, $\Tn'_-(t)$ and $\Tn'_+(t)$ are decreasing and $$\lim_{t\rightarrow+\infty}\frac{\Tn(t)}{t}=\lim_{t\rightarrow+\infty}\Tn'_-(t)=\lim_{t\rightarrow+\infty}\Tn'_+(t).$$ As $\sup_{(f,o)\in\mathcal{O}_{o}^*}(c_o^{f}(\varphi)-\sigma(\log|f|,\varphi))<+\infty$, there exists $C>0$ such that $\Tn(m)=c_o^{f^m}(\varphi)\le m\sigma(\log|f|,\varphi)+C$ for any $m\in\mathbb{Z}_{>0}$, which implies $\lim_{m\rightarrow+\infty}\frac{\Tn(m)}{m}\le \sigma(\log|f|,\varphi)$. As $\frac{\Tn(t)}{t}\ge \sigma(\log|f|,\varphi)$ for any $t\ge0$ and $\lim_{t\rightarrow+\infty}\frac{\Tn(t)}{t}$ exists, we have $\lim_{t\rightarrow+\infty}\frac{\Tn(t)}{t}=\sigma(\log|f|,\varphi)$.
\end{proof}

We recall the following basic property of holomorphic functions.
\begin{Lemma}
\label{pro: order finite when compact}
Let $X$ be a connected complex manifold. Let $f\not\equiv 0$ be a holomorphic function defined on an open neighborhood of a compact set $K\subset X$. Then $\sup\limits_{p\in K}\ord_p{f}<+\infty$.
\end{Lemma}
\begin{proof}
If not, then $\sup\limits_{p\in K}\ord_p{f}=+\infty$. There exists a sequence of points $\{p_m\}_{m=1}^{+\infty}$ contained in $K$ such that $\lim\limits_{m\to+\infty}\ord_{p_m}{f}=+\infty$. As $K$ is compact, there exists a point $p\in K$ such that some subsequence of $\{p_m\}$ (also denoted by $\{p_m\}$) converges to $p$.

Let $(U;z_1,z_2,\ldots,z_n)$ be a small local coordinated neighborhood near $p$. As $\lim\limits_{m\to+\infty}\ord_{p_m}{f}=+\infty$,  for any multiple index $\alpha=(\alpha_1,\ldots,\alpha_n)$, we have $\frac{\partial^{\alpha}f}{\partial z^{\alpha}}(p_m)=0$ when $m$ is large enough. By the continuity of $\frac{\partial^{\alpha}f}{\partial z^{\alpha}}$ and that $p_m$ converges to $p$, we know $\frac{\partial^{\alpha}f}{\partial z^{\alpha}}(p)=0$ for any multiple index $\alpha=(\alpha_1,\ldots,\alpha_n)$. Thus, by the rigidity of holomorphic function, we know that $f$ is identical zero which is a contradiction.
\end{proof}

We recall the following result due to Skoda.
\begin{Lemma}[\cite{Skda72}; see also \cite{demailly2010}]\label{l:Lelong}
	For any plurisubharmonic function $\varphi$ near $o$, if the Lelong number $\nu_o(\varphi)<1$, then $e^{-2\varphi}$ is locally integrable near $o$.
\end{Lemma}

Using Lemma \ref{l:Lelong}, we present the following lemma.

\begin{Lemma}
\label{lem:singular skoda 72}
Let $(X,o)$ be an irreducible germ of an analytic set. 
	Let $C>0$, and $\varphi$ a plurisubharmonic function defined near $(X,o)$ such that $c_o(\varphi)<C$. Then for any holomorphic function $f\in \mathcal{O}^{w}_{X,o}$, there exists a constant $C_f\ge 0$ (depending on $f$ but independent of $\varphi$) such that
 $$\sigma_o(\log|f|,\varphi)\le CC_f.$$
\end{Lemma}

\begin{proof}
Using Theorem \ref{thm:desing}, we can resolve
	the singularities of $X$, and denote the corresponding proper modification by $\mu\colon \widetilde\Omega\rightarrow\Omega$, where $\Omega\subset \mathbb{C}^{n+l}$ such that $o\in X\cap \Omega$. Denote the strict transform of $X$ by $\widetilde X$ and denote $Z_0\coloneqq \widetilde\Omega\cap\mu^{-1}(\{o\})$. Note that $c_o(\varphi)=c^{\det \mu}_{Z_0\cap \widetilde{X}}(\varphi\circ\mu)$. Hence $c_o(\varphi)<C$  implies that
	$$c^{\det \mu}_{Z_0\cap \widetilde{X}}(\varphi\circ\mu)<C.$$ 
	Recall that $c^{\det \mu}_{Z_0\cap \widetilde{X}}(\varphi\circ\mu)=\inf_{p\in Z_0\cap \widetilde{X}}c^{\det \mu}_{p}(\varphi\circ\mu)$.
	There exists a point $p\in Z_0\cap \widetilde{X}$ such that $c^{\det \mu}_{p}(\varphi\circ\mu)<C$.  As $c^{\det \mu}_{p}(\varphi\circ\mu)\ge c_{p}(\varphi\circ\mu)$, we know that $c_{p}(\varphi\circ\mu)<C$.
	It follows from Lemma \ref{l:Lelong} that the Lelong number of $\varphi \circ \mu$ at $p$ $$\nu_{p}(\varphi\circ\mu)>\frac{1}{C},$$ 
	which implies $\sigma_{p}(\log|f\circ \mu|,\varphi\circ\mu)\le C \ord_{p}(f)$.
	
	Recall that $\sigma_o(\log|f|,\varphi)=\sigma_{Z_0\cap \widetilde{X}}(\log|f\circ \mu|,\varphi\circ\mu)$ and $\sigma_{Z_0\cap \widetilde{X}}(\log|f\circ \mu|,\varphi\circ\mu)=\inf_{p\in Z_0\cap \widetilde{X}}\sigma(\log|f\circ \mu|,\varphi\circ\mu)$. Hence we know 
	$$\sigma_o(\log|f|,\varphi)\le \sigma_{p}(\log|f\circ \mu|,\varphi\circ\mu)\le C \ord_{p}(f\circ \mu)\le C \sup_{p\in Z_0\cap \widetilde{X}} \ord_{p}(f\circ \mu).$$
	Note that $Z_0\cap \widetilde{X}$ is compact in $\widetilde{X}$, Lemma \ref{pro: order finite when compact} tells that $C_f\coloneqq\sup\limits_{p\in Z_0\cap \widetilde{X}} \ord_{p}(f\circ \mu)<+\infty$. Hence we know $ \sigma_o(\log|f|,\varphi)\le CC_f$.
\end{proof}

\subsection{Convergence results for valuations and for relative types}

\

Let $\{\nu_{j}\}_{j\in\mathbb{Z}_{>0}}$ be a sequence of valuations on $\mathcal{O}_o$. We need the following  convergence result for valuations, which will be used to prove the main theorem.

\begin{Proposition}
\label{r:proof of 1.2} Assume that  $\sup_j\nu_j(g)<+\infty$ for any holomorphic function $g$ near $o$. There exists a subsequence of $\{\nu_j\}$ denoted by $\{\nu_{j_l}\}$ such that $\{\nu_{j_l}\}$ converges to a valuation $\nu_0$ on $\mathcal{O}_o$.

\end{Proposition}

\begin{proof}
	Denote $I_k^j\coloneqq \{(f,o)\in\mathcal{O}_o\colon \nu_j(f)\ge k\}$ for every $j,k\in\mathbb{Z}_{>0}$. 
	
	\
	
	\emph{Step 1.} Find a subsequence $\{I_k^{j_l}\}$ of $I_k^j$ such that $\cup_{l_1\ge1}\big(\cap_{l\ge l_1}I_k^{j_l}\big)$ is maximal for any $k$. It means that for any subsequence $\{I_k^{j'_l}\}$ of $\{I_k^{j_l}\}$, $\cup_{l_1\ge1}\big(\cap_{l\ge l_1}I_k^{j'_l}\big)=\cup_{l_1\ge1}\big(\cap_{l\ge l_1}I_k^{j_l}\big)$.
	
	Note that
	$$\cup_{j_1\ge1}\big(\cap_{j\ge j_1}I_k^{j}\big)=\{(f,o)\in\mathcal{O}_o\colon\exists\, j_1\text{ s.t. $\nu_j(f)\ge k$ holds for any $j\ge j_1$}\}.$$
	
	As $\mathcal{O}_o$ is  a Noetherian ring, there exists a subsequence $\{I_1^{j^{(1)}(l)}\}$ of $I_1^j$ such that $\cup_{l_1\ge1}\big(\cap_{l\ge l_1}I_1^{j^{(1)}(l)}\big)$ is maximal. Similarly, there exists a subsequence $\{I_1^{j^{(2)}(l)}\}$ of $I_1^{j^{(1)}(l)}$ such that $\cup_{l_1\ge1}\big(\cap_{l\ge l_1}I_2^{j^{(2)}(l)}\big)$ is maximal. By induction, for any $r\in\mathbb{Z}_{\ge1}$, there exists a sequence $\{I_1^{j^{(r)}(l)}\}$  such that $\cup_{l_1\ge1}\big(\cap_{l\ge l_1}I_r^{j^{(r)}(l)}\big)$ is maximal and $\{I_1^{j^{(r+1)}(l)}\}$ is a subsequence of $\{I_1^{j^{(r)}(l)}\}$. Using the diagonal method, we obtain a subsequence $\{I_k^{j_l}\}$ of $I_k^j$ such that $\cup_{l_1\ge1}\big(\cap_{l\ge l_1}I_k^{j_l}\big)$ is maximal for any $k$.

	\

	\emph{Step 2.}
	We prove that $\lim_{l\rightarrow+\infty}\nu_{j_l}(f)$ exists for any $(f,o)\in\mathcal{O}_o$.
	
	We prove it by contradiction: if not, there exist two subsequences of valuations $\{\nu_{j'_{l}}\}$ and $\{\nu_{j''_l}\}$ of $\{\nu_{j_l}\}$ such that $$\lim_{l\rightarrow+\infty}\nu_{j'_l}(f)=a_1<a_2=\lim_{l\rightarrow+\infty}\nu_{j''_l}(f).$$
	Note that $\nu_j(f^m)=m\nu_j(f)$. Then, without loss of generality, assume that there exists a positive integer $k_0\in (a_1,a_2)$. Then we have
	$$(f,o)\in \cup_{l_1\ge1}\big(\cap_{l\ge l_1}I_{k_0}^{j'_l}\big)$$
	and $$(f,o)\not\in \cup_{l_1\ge1}\big(\cap_{l\ge l_1}I_{k_0}^{j_l}\big),$$
	which contradicts to the maximality of $\cup_{l_1\ge1}\big(\cap_{l\ge l_1}I_{k_0}^{j_l}\big)$.
	Thus, $\lim_{l\rightarrow+\infty}\nu_{j_l}(f)$ exists for any $(f,o)\in\mathcal{O}_o$.
	
	Denote $\nu_0(f)\coloneqq \lim_{l\rightarrow+\infty}\nu_{j_l}(f)$ for any $(f,o)\in\mathcal{O}_o$. It follows from every $\nu_j$ is a valuation on $\mathcal{O}_o$ that $\nu_0$ is a valuation on $\mathcal{O}_o$. Proposition \ref{r:proof of 1.2} has been proved.
\end{proof}

Let us consider valuations on the ring of germs of weakly holomorphic functions. We recall the following basic property of the ring.

\begin{Lemma}[see \cite{Whitneybook}]
\label{noetherian of weakly holomorphic function}
The ring of germs of weakly holomorphic functions $\mathcal{O}^{w}_o$ is Noetherian.
\end{Lemma}
 As $\mathcal{O}^{w}_o$ is Noetherian, using the same proof as Proposition \ref{r:proof of 1.2}, we also have the following convergence result. 
\begin{Proposition}
\label{r:proof of 1.2 weakly} Let $\{\nu_{j}\}_{j\in\mathbb{Z}_{>0}}$ be a sequence of valuations on $\mathcal{O}^{w}_o$. Assume that  $\sup_j\nu_j(g)<+\infty$ for any weakly holomorphic function $g$ near $o$. Then there exists a subsequence $\{\nu_{j_l}\}$ of $\{\nu_j\}$ which converges to a valuation $\nu_0$ on $\mathcal{O}^{w}_o$.
\end{Proposition}

\section{Proof of Remark \ref{rem:max_existence}}

\begin{proof}
	[Proof of Remark \ref{rem:max_existence}]
	Theorem \ref{thm:SOC} shows that there exists $\varepsilon>0$,
	such that
	$$|f_{0}|^{2}e^{-2\varphi_{0}}|z|^{2N_{0}}e^{-2(1+\varepsilon)\varphi}$$
	is integrable near $z_0$.
	We have
	\begin{equation}
		\nonumber\begin{split}
			&\int_U|f_{0}|^{2}e^{-2\varphi_{0}}e^{-2\varphi}-|f_{0}|^{2}e^{-2\varphi_{0}}e^{-2\max\big\{\varphi,\frac{N_{0}}{\varepsilon}\log|z|\big\}}\\
			\le&\int_{U\cap\big\{\frac{N_{0}}{\varepsilon}\log|z|\ge\varphi\big\}}|f_{0}|^{2}e^{-2\varphi_{0}}e^{-2\varphi}\\
			\le&\int_{U\cap\big\{\frac{N_{0}}{\varepsilon}\log|z|\ge\varphi\big\}}|f_{0}|^{2}e^{-2\varphi_{0}}|z|^{2N_0}e^{-2(1+\varepsilon)\varphi}\\
			<&+\infty,
		\end{split}
	\end{equation}
	where $U\subset X$ is a neighborhood of	$z_0$.
	Then 	it suffices to consider that there exists $N$ large enough such that
	$$N\log|z|\leq \varphi$$
	near $z_0$. By a similar discussion, for any local Zhou weight $\Phi_{z_0,\max},$ $\Phi_{z_0,\max}\ge N_1\log|z|+O(1)$ near $o$ for some $N_1\gg0$.
	
	Let $V\subset X$ be a small neighborhood of $z_0$ such that $e^{-\varphi}$ is integrable near any $z\in V\setminus\{z_0\}$ and $\varphi<0$ on $V$.
	Let $(u_{\alpha})_{\alpha}$ be the negative plurisubharmonic functions on $V$
	such that $u_{\alpha}\geq \varphi+O(1)$ near $z_0$ and $|f_0|^2e^{-2\varphi_0}e^{-2u_{\alpha}}$ is not integrable near $z_0$.
	
	\emph{Zorn's Lemma} shows that
	there exists $\Gamma$ which is a maximal set such that for any $\alpha,\alpha'\in\Gamma$,
	$u_{\alpha}\leq u_{\alpha'}+O(1)$ or $u_{\alpha'}\leq u_{\alpha}+O(1)$ holds near $o$,
	where $(u_{\alpha})$ are negative plurisubharmonic functions on $V$.
	
	Let $u(z)\coloneqq \sup_{\alpha\in\Gamma}u_{\alpha}(z)$ on $V$,
	and let $u^{*}(z)=\lim_{\varepsilon\to0}\sup_{U_{z,\varepsilon}}u$, where $U_{z,\varepsilon}\coloneqq X\cap\mathbb{B}^{n+l}(z,\varepsilon)$.
	Lemma \ref{lem:Choquet} shows that there exists subsequence $(v_{j})$ of $(u_{\alpha})$ such that
	$(\max_{j}v_{j})^{*}=u^{*}$.
	Moreover one can choose $v_{j}(\coloneqq \sup_{j'\leq j}v_{j})$ increasing with respect to $j$
	such that $|f_0|^2e^{-2\varphi_0}e^{-2v_j}$ is not integrable near $z_0$.
	
	Lemma \ref{lem:Choquet} shows that
	$(v_{j})$ is convergent to $v^{*}$ with respect to $j$ almost everywhere with respect to Lebesgue measure,
	and $v^{*}$ is a plurisubharmonic function on $V$.
	Theorem \ref{thm:SOC} 
	shows that $|f_0|^2e^{-2\varphi_0}e^{-2v^*}$ is not integrable near $z_0$.
	
	In the following part,
	we prove that $v^{*}$ is a local Zhou weight  related to $|f_0|^2e^{-2\varphi_0}$ near $z_0$ by contradiction.
	If not, then there exists a plurisubharmonic function $\tilde{v}$ near $z_0$ such that $\tilde{v}\geq v^{*}$,
	$|f_0|^2e^{-2\varphi_0}e^{-2\tilde v}$ is not integrable near $z_0$, and
	\[\limsup_{z\to z_0}\big(\tilde{v}(z)-v^{*}(z)\big)=+\infty.\]
	
	As $\varphi<0$ on $V$, by the definition of $v^*$, we have $v^{*}\geq\varphi$, which shows $v^{*}\geq N\log|z|$.
	Then for  $U_{z_0,\varepsilon}\coloneqq \mathbb{B}(z_0,\varepsilon)\cap X$ ($\varepsilon$ is small enough),
	there exists $M\ll 0$ such that $\tilde{v}+M< N\log|z|\le v^{*}$ near the boundary of $U_{z_0,\varepsilon}$,
	which implies that $\max\{\tilde{v}+M,v^{*}\}=v^{*}$ near the boundary of $U_{z_0,\varepsilon}$.
	Let
	\begin{equation*}
		\tilde{\varphi}\coloneqq \left\{
		\begin{array}{ll}
			\max\{\tilde{v}+M,v^{*}\} & \ \text{on} \ U_{z_0,\varepsilon}, \\
			v^{*} & \ \text{on} \ V\setminus U_{z_0,\varepsilon}.
		\end{array}
		\right.
	\end{equation*}
	$\tilde v\ge v^*$ implies that $\tilde\varphi=\tilde v+O(1)$ near $z_0$. Then $\tilde{\varphi}$ is a plurisubharmonic function on $V$ such that
	\[\limsup_{z\to z_0}(\tilde{\varphi}(z)-u^{*}(z))=\limsup_{z\to z_0}(\tilde{\varphi}(z)-v^{*}(z))\geq \limsup_{z\to z_0}(\tilde{v}(z)+M-v^{*}(z))=+\infty,\]
	and $|f_0|^2e^{-2\varphi_0}e^{-2\tilde\varphi}$ is not integrable near $z_0$, which contradicts the definition of $u^{*}$.
	
	This proves Remark \ref{rem:max_existence}.
\end{proof}

\section{Proof of Theorem \ref{thm:expression of relative types}}

\begin{proof}[Proof of Theorem \ref{thm:expression of relative types}]
Remark \ref{rem:max_existence} shows that $\Phi_{z_0,\max}\ge N\log|z|+O(1)$ near $z_0$ for some $N>0$.
Following from Propositions \ref{p:sidediv} and \ref{p:relative=derivative}, we know that for any plurisubharmonic function $\psi$ satisfying that  $\psi\leq s_{0}\log|z|+O(1)$ near $z_0\in X$ for some $s_0>0$,
 the limit $	\lim_{t_{1}\to+\infty}\frac{1}{2t_{1}}\frac{\int_{\{\mathrm{Tn}(0)\varphi<-t_{1}\}\cap U}|f_{0}|^{2}e^{-2\varphi_{0}}(-2\psi)}{\int_{\{\mathrm{Tn}(0)\varphi<-t_{1}\}\cap U}|f_{0}|^{2}e^{-2\varphi_{0}}}$ exists and 
$$\lim_{t_{1}\to+\infty}\frac{1}{2t_{1}}\frac{\int_{\{\mathrm{Tn}(0)\varphi<-t_{1}\}\cap U}|f_{0}|^{2}e^{-2\varphi_{0}}(-2\psi)}{\int_{\{\mathrm{Tn}(0)\varphi<-t_{1}\}\cap U}|f_{0}|^{2}e^{-2\varphi_{0}}}=\sigma(\psi,\Phi_{z_0,\max}).$$ 
Thus, Theorem \ref{thm:expression of relative types} holds.
\end{proof}

\begin{proof}[Proof of Corollary \ref{coro:valuation}]
If $f(z_0)=0$, it is clear that $\log|f|\le C_1\log|z|+O(1)$ near $z_0$ for some $C_1>0$. Thus, $\nu(f,\Phi_{z_0,\max})>0$ since $\Phi_{z_0,\max}\ge C_2\log|z|+O(1)$ near $z_0$. 

If $f(z_0)\not=0$, then there exists a neighborhood $U$ such that $\inf_U|f|>0$. As $|f_0|^2e^{-2\varphi_0-2\Phi_{z_0,\max}}$ is not integrable near $z_0$ and $|f_0|^2e^{-2\varphi_0}$ is  integrable near $z_0$, we know that $\liminf_{X_{\reg}\ni z\rightarrow z_0}\Phi_{z_0,\max}=-\infty$. Then $\nu(f,\Phi_{z_0,\max})=0$. The statement $(3)$ holds.

Note that $\log|fg|=\log|f|+\log|g|$. It follows from Theorem \ref{thm:expression of relative types} and statement $(3)$ that $\nu(fg,\Phi_{z_0,\max})=\nu(f,\Phi_{z_0,\max})+\nu(g,\Phi_{z_0,\max})$. The statement $(1)$ holds.

As $\log|f+g|\le\log(|f|+|g|)=\max\{\log|f|,\log|g|\}+O(1)$, we have
$$\nu(f+g,\Phi_{z_0,\max})\ge\sigma(\max\{\log|f|,\log|g|\},\Phi_{z_0,\max}).$$
For any positive number $c<\min\{\nu(f,\Phi_{z_0,\max}),\nu(g,\Phi_{z_0,\max})\}$, by definition, we have $\log|f|\le c\Phi_{z_0,\max}+O(1)$ and $\log|g|\le c\Phi_{z_0,\max}+O(1)$ near $z_0$, which shows $\max\{\log|f|,\log|g|\}\log c\Phi_{z_0,\max}+O(1)$ near $z_0$. Hence, $\nu(f+g,\Phi_{z_0,\max})\ge\sigma(\max\{\log|f|,\log|g|\},\Phi_{z_0,\max})\ge \min\{\nu(f,\Phi_{z_0,\max}),\nu(g,\Phi_{z_0,\max})\}$. The statement $(2)$ holds.

Thus, Corollary \ref{coro:valuation} holds.
\end{proof}

\begin{Proposition}
	\label{thm:main_value1}
	Let $f=(f_{1},\cdots,f_{m'})$ be a vector, where $f_{1},\cdots,f_{m'}$ are holomorphic functions near $z_0$.
	Denote by $|f|\coloneqq (|f_{1}|^2+\cdots+|f_{m'}|^{2})^{1/2}$.
	Let $\Phi_{z_0,\max}$ be a local Zhou weight related to $|f_{0}|^{2}e^{-2\varphi_{0}}$ near $z_0$. Then the following two statements hold:
	
	$(1)$ for any $\alpha>0$,
	$\big(1+\alpha\sigma(\log|f|,\Phi_{z_0,\max})\big)\Phi_{z_0,\max}$ is a local Zhou weight related to $|f|^{2\alpha}|f_{0}|^{2}e^{-2\varphi_{0}}$ near $z_0$;
	
	$(2)$ $\big(1+\sigma(\varphi_{0},\Phi_{z_0,\max})\big)\Phi_{z_0,\max}$ is a local Zhou weight related to $|f_{0}|^{2}$.
\end{Proposition}	
\begin{proof}
	Firstly, we prove the statement $(1)$ in Proposition \ref{thm:main_value1}.
	
	Denote by
	\[\mathrm{Tn}(t)\coloneqq \sup\big\{c:|f_0|^2|f|^{2t}e^{-2\varphi_0-2c\Phi_{z_0,\max}} \ \text{is integrable near} \ o\big\}.\]
	By Proposition \ref{p:relative=derivative}, we have $\mathrm{Tn}(\alpha)=1+\sigma(\log|f|,\Phi_{z_0,\max})\alpha$, which implies that $$|f_0|^2|f|^{2\alpha}e^{-2\varphi_0-2(1+\alpha\sigma(\log|f|,\Phi_{z_0,\max}))\Phi_{z_0,\max}}$$ is not integrable near $z_0$ by Theorem \ref{thm:SOC}. As $\log|f|\le\sigma(\log|f|,\Phi_{z_0,\max})\Phi_{z_0,\max}+O(1)$ and there exists $N\gg0$ such that $|f_0|^2|z|^{2N}e^{-2\varphi_0-2\Phi_{z_0,\max}}$ is integrable near $z_0$, we know that $$|f_0|^2|f|^{2\alpha}|z|^{2N}e^{-2\varphi_0-2\big(1+\alpha\sigma(\log|f|,\Phi_{z_0,\max})\big)\Phi_{z_0,\max}}$$ is integrable near $z_0$.
	
	Let $\tilde{\varphi}$ be a plurisubharmonic function near $z_0$ satisfying that $$\tilde\varphi\ge\big(1+\alpha\sigma(\log|f|,\Phi_{z_0,\max})\big)\Phi_{z_0,\max}+O(1)$$ and $|f|^{2\alpha}|f_0|^2e^{-2\varphi_0-2\tilde\varphi}$ is not integrable near $z_0$.
	Note that
	$$\log|f|\le\sigma(\log|f|,\Phi_{z_0,\max})\Phi_{z_0,\max}+O(1),$$
	then
	\[\tilde{\varphi}\ge\frac{1+\alpha\sigma(\log|f|,\Phi_{z_0,\max})}{\sigma(\log|f|,\Phi_{z_0,\max})}\log|f|+O(1)\]
	and
	\begin{equation*}
		\begin{split}
			|f|^{2\alpha}|f_0|^2e^{-2\varphi_0-2\tilde\varphi} & \le C|f_0|^2
			e^{-2\varphi_0}e^{-\frac{2}{1+\alpha\sigma(\log|f|,\Phi_{z_0,\max})}\tilde{\varphi}}.
		\end{split}
	\end{equation*}
	As $|f|^{2\alpha}|f_0|^2e^{-2\varphi_0-2\tilde\varphi}$ is not integrable near $z_0$, we know that $$|f_0|^2e^{-2\varphi_0-\frac{2}{1+\alpha\sigma(\log|f|,\Phi_{z_0,\max})}\tilde\varphi}$$
	is not integrable near $z_0$. Note that  $\Phi_{z_0,\max}$ is a local Zhou weight related to $|f_{0}|^{2}e^{-2\varphi_{0}}$. Then we obtain
	$$\tilde\varphi=(1+\alpha\sigma(\log|f|,\Phi_{z_0,\max}))\Phi_{z_0,\max}+O(1),$$
	which shows that $\big(1+\alpha\sigma(\log|f|,\Phi_{z_0,\max})\big)\Phi_{z_0,\max}$ is a local Zhou weight related to $|f|^{2\alpha}|f_{0}|^{2}e^{-2\varphi_{0}}$.
	
	Next, we give the proof of statement $(2)$, which is similar to the proof of statement $(1)$.
	
	Denote by
	\[\mathrm{Tn}(t)\coloneqq \sup\big\{c\colon |f_0|^2e^{-2\varphi_0}e^{2t\varphi_0}e^{-2c\Phi_{z_0,\max}} \ \text{is integrable near} \ o\big\}.\]
	By Proposition \ref{p:relative=derivative}, we have $\mathrm{Tn}(1)=1+\sigma(\varphi_0,\Phi_{z_0,\max})$, which implies that $$|f_0|^2e^{-2\big(1+\sigma(\varphi_0,\Phi_{z_0,\max})\big)\Phi_{z_0,\max}}$$ is not integrable near $z_0$. As $\varphi_0\le\sigma(\varphi_0,\Phi_{z_0,\max})\Phi_{z_0,\max}+O(1)$ and there exists $N\gg0$ such that $|f_0|^2|z|^{2N}e^{-2\varphi_0-2\Phi_{z_0,\max}}$ is integrable near $z_0$, we know that $|f_0|^2|z|^{2N}e^{-2\big(1+\sigma(\varphi_0,\Phi_{z_0,\max})\big)\Phi_{z_0,\max}}$ is integrable near $z_0$.
	
	Let $\tilde{\varphi}$ be a plurisubharmonic function near $z_0$ satisfying that $$\tilde\varphi\ge\big(1+\sigma(\varphi_0,\Phi_{z_0,\max})\big)\Phi_{z_0,\max}+O(1)$$ and $|f_0|^2e^{-2\tilde\varphi}$ is not integrable near $z_0$.
	It follows from $\varphi_0\le\sigma(\varphi_0,\Phi_{z_0,\max})\Phi_{z_0,\max}+O(1)$ that
	\begin{equation*}
		\begin{split}
			|f_0|^2e^{-2\tilde\varphi} & =e^{2\varphi_0}|f_0|^2
			e^{-2\varphi_0}e^{-2\tilde{\varphi}} \\
			& \le Ce^{2\sigma(\varphi_0,\Phi_{z_0,\max})\Phi_{z_0,\max}}|f_0|^2
			e^{-2\varphi_0}e^{-2\tilde{\varphi}} \\
			& \le C_1|f_0|^2
			e^{-2\varphi_0}e^{2\frac{\sigma(\varphi_0,\Phi_{z_0,\max})}{1+\sigma(\varphi_0,\Phi_{z_0,\max})}\tilde{\varphi}}
			e^{-2\tilde{\varphi}}\\
			&= C_1|f_0|^2
			e^{-2\varphi_0}e^{-\frac{2}{1+\sigma(\varphi_0,\Phi_{z_0,\max})}\tilde{\varphi}}.
		\end{split}
	\end{equation*}
	As $|f_0|^2e^{-2\tilde\varphi}$ is not integrable near $z_0$, we know that $|f_0|^2e^{-2\varphi_0-\frac{2}{1+\sigma(\varphi_0,\Phi_{z_0,\max})}\tilde\varphi}$
	is not integrable near $z_0$. Note that  $\Phi_{z_0,\max}$ is a local Zhou weight related to $|f_{0}|^{2}e^{-2\varphi_{0}}$. Then we obtain
	$$\tilde\varphi=\big(1+ \sigma(\varphi_0,\Phi_{z_0,\max})\big)\Phi_{z_0,\max}+O(1),$$
	which shows that $\big(1+\sigma(\varphi_0,\Phi_{z_0,\max})\big)\Phi_{z_0,\max}$ is a local Zhou weight related to $|f_{0}|^{2}$.
\end{proof}

\section{Proof of Theorem \ref{thm:valu-jump} }

It is clear that 
\[\sigma \Big(\log|G|,\big(1+\sigma(\varphi_{0},\Phi_{z_0,\max})\big)\Phi_{z_0,\max}\Big)=\frac{\nu(G,\Phi_{z_0,\max})}{1+\sigma(\varphi_{0},\Phi_{z_0,\max})}\]
and 
\[c^G_{z_0}\Big(\big(1+\sigma(\varphi_{0},\Phi_{z_0,\max})\big)\Phi_{z_0,\max}\Big)=\frac{c^G_{z_0}(\Phi_{z_0,\max})}{1+\sigma(\varphi_{0},\Phi_{z_0,\max})}.\]
Proposition \ref{thm:main_value1} shows that it suffices to consider the case $\varphi_0\equiv0$. Let us consider a mixed Tian function
$$\mathrm{Tn}(s,t)\coloneqq \sup\big\{c:|G|^{2s}|f_0|^{2t}e^{-2c\Phi_{z_0,\max}}\text{  is integrable near $z_0$}\big\}.$$

By definitions of $\mathrm{Tn}(s,t)$ and Zhou numbers, we have 
$$\mathrm{Tn}(1,0)=c_{z_0}^G(\Phi_{z_0,\max})\ge c_{z_0}(\Phi_{z_0,\max})+\nu(G,\Phi_{z_0,\max})$$
and $\mathrm{Tn}(1,1)\ge \mathrm{Tn}(1,0)+\sigma(\log|f_0|,\Phi_{z_0,\max})$.
Since $\Phi_{z_0,\max}$ is a local Zhou weight related to $|f_0|^2$ near $z_0$, we have $\mathrm{Tn}(0,1)=1$ and $\mathrm{Tn}(1,1)=\mathrm{Tn}(0,1)+\nu(G,\Phi_{z_0,\max})$ (by Proposition \ref{thm:main_value1}). Then we have  
\begin{equation}
	\nonumber
	\begin{split}
c_{z_0}^G(\Phi_{z_0,\max})=\mathrm{Tn}(1,0)&\le \mathrm{Tn}(1,1)-\sigma(\log|f_0|,\Phi_{z_0,\max})\\
&=\mathrm{Tn}(0,1)+\nu(G,\Phi_{z_0,\max})-\sigma(\log|f_0|,\Phi_{z_0,\max})\\
&=\nu(G,\Phi_{z_0,\max})-\sigma(\log|f_0|,\Phi_{z_0,\max})+1.
	\end{split}
\end{equation}

Thus, Theorem \ref{thm:valu-jump} holds.

\section{Proofs of Theorem \ref{thm:interpolation}, Theorem \ref{thm:interpolation weakly} and their Corollaries}

We firstly proof Theorem \ref{thm:interpolation}.
\begin{proof}
	We prove this theorem in two steps.
	
	\emph{Step 1. $(1)\Rightarrow(2)$.}
	
	Note that $F\coloneqq \prod_{0\le j\le m}f_j$.
If there exists a valuation $\nu$  such that $\nu(f_j)=a_j$ for any $j$, then by Lemma \ref{l:relativetype} and $a_0=0$, 
	\[\sigma(\log|F|,\varphi)\le \nu(F)=\sum_{0\le j\le m}\nu(f_j)=\sum_{1\le j\le m}a_j.\]
	On the other hand, by the definition of relative type, clearly we also have
	\[\sigma(\log|F|,\varphi)\ge\sum_{0\le j\le m}\sigma(\log|f_j|,\varphi)\ge \sum_{1\le j\le m}a_j.\]
	Thus, $\sigma(\log|F|,\varphi)=\sum_{1\le j\le m}a_j$ holds.

	\emph{Step 2. $(2)\Rightarrow(1)$.}

	Recall that $X$ is an analytic subset of $\Omega\in\mathbb{C}^{n}$ with pure dimension $n$ and denote  $|z|^2\coloneqq \sum_{1\le j\le n}|z_j|^2$
	
		As $\sigma(\log|F|,\varphi)<+\infty,$ we have $f_j(o)=0$ for any $1\le j\le m.$
	Denote $\varphi_N\coloneqq \max\{\varphi,N\log|z|\}$ for any $N>0$. Then we have $\sigma(\log|f_j|,\varphi_N)\ge\sigma(\log|f_j|,\varphi)\ge a_j$ for all $0\le j\le m$. It follows from Lemma \ref{l:holder singular} that there exists $C>0$ such that $c_o^{f}(\varphi)\le \sigma(\log|f|,\varphi)+C$ for any $(f,o)\in\mathcal{O}_{X,o}$. Thus,  for any $k>0$, there exists $N_k$ such that
	\begin{equation}
		\nonumber
		\begin{split}
			\sigma(\log|F|,\varphi)&\ge\frac{c_o^{F^k}(\varphi)-C}{k}\\
			&\ge\frac{c_o^{F^k}(\varphi_{N_k})-1-C}{k}\\
			&\ge\frac{k\sigma(\log|F|,\varphi_{N_k})-1-C}{k}\\
			&\ge \lim_{N\rightarrow+\infty}\sigma(\log|F|,\varphi_{N})-\frac{1+C}{k},
		\end{split}
	\end{equation}
	where the second ``$\ge$" follows from Lemma \ref{l:0921-1}. Then we have 
	\[ \lim_{N\rightarrow+\infty}\sigma(\log|F|,\varphi_{N})=\sigma(\log|F|,\varphi).\]
	
	Consider the Tian function
	$$\Tn (t)\coloneqq \sup\big\{c\ge 0 \colon |F|^{2t}e^{-2c\varphi_N}\text{ is integrable near }o\big\}.$$
	It follows from Lemma \ref{l:holder singular} and Remark \ref{r:tame} that 
	\begin{equation}
	    \label{eq:0918a}
	    \lim_{t\rightarrow+\infty}\Tn '_-(t)=\lim_{t\rightarrow+\infty}\Tn '_+(t)=\sigma(\log|F|,\varphi_N).
	\end{equation}
	Then for every positive integer $s$, $|F|^{2s}e^{-2\Tn (s)\varphi_N}$ is not integrable near $o$ (because of the strong openness property of multiplier ideal sheaves \cite{GZopen-c}), then by \Cref{rem:max_existence}, there exists a local Zhou weight $\Phi_{N,s}$ near $o$ related to $|F|^{2s}$ such that
	$$\Phi_{N,s}\ge \Tn (s)\varphi_N.$$
	Consider the Tian function
	$$\widetilde{\Tn }(t)\coloneqq \sup\Big\{c\ge0\colon |F|^{2t}e^{-2c\frac{\Phi_{N,s}}{\Tn (s)}}\text{ is integrable near }o\Big\}.$$
	Thus, we have $\widetilde{\Tn }(s)=\Tn (s)$ and $\widetilde{\Tn }(t)\ge \Tn (t)$ for any $t\ge0$. It follows from Proposition \ref{p:relative=derivative} that $\widetilde{\Tn }'(s)=\sigma\Big(\log|F|,\frac{\Phi_{N,s}}{\Tn (s)}\Big)$. Then we have
	\begin{equation}
	    \label{eq:0918b}
	    \Tn '_+(s)\le\sigma\Big(\log|F|,\frac{\Phi_{N,s}}{\Tn (s)}\Big)\le\Tn '_-(s).
	\end{equation}

	Denote the corresponding Zhou valuation of $\Phi_{N,s}$ by $\nu_{N,s}$. Then we have
	$$\Tn (s)\nu_{N,s}(f_j)=\Tn (s)\sigma(\log|f_j|,\Phi_{N,s})\ge\sigma(\log|f_j|,\varphi_N)\ge a_j$$
	for all $1\le j\le m$ and
	$$\Tn (s)\nu_{N,s}(F)=\Tn (s)\sigma(\log|F|,\Phi_{N,s})=\sigma\Big(\log|F|,\frac{\Phi_{N,s}}{\Tn (s)}\Big).$$
	Note that 
	\[\varphi_{N}=\log\Big(\sum_{1\le j\le m}|f_j|^{\frac{1}{a_j}}+\sum_{1\le k\le n}|z_k|^N\Big)+O(1)\]
	near $o$. By Lemma \ref{l:holder singular}, there exists a constant $C$ (independent of $N$) such that $\sup_{(f,o)\in\mathcal{O}_o^*}(c_o^{f}(\varphi_{N})-\sigma(\log|f|,\varphi_{N}))< C$. 
	As $\widetilde{\Tn }(t)$ is a concave function, we have
	\begin{equation}
		\label{eq:0712a}
		c_o\Big(\frac{\Phi_{N,s}}{\Tn (s)}\Big)=\widetilde{\Tn }(0)\le \Tn (s)-s\Tn '_+(s).
	\end{equation}
	It follows from  the concavity of $\Tn (t)$ and equality \eqref{eq:0918a} that $\Tn (s)-s\Tn '_+(s)\le c_o^{F^s}(\varphi_N)-\sigma(\log|F^s|,\varphi_N)< C$. Then the inequality \eqref{eq:0712a} implies $c_o\big(\frac{\Phi_{N,s}}{\Tn (s)}\big)< C$ for any $N$ and $s$. It follows from Lemma \ref{lem:singular skoda 72}
 that for any $(f,o)$,
	\begin{equation}
		\label{eq:0712b}
		\Tn (s)\nu_{N,s}(f)=\sigma\Big(\log|f|,\frac{\Phi_{N,s}}{\Tn (s)}\Big) 
	\end{equation}
is uniformly bounded with respect to $N$ and $s$.
	
	By equality \eqref{eq:0918a}, inequalities \eqref{eq:0918b}, \eqref{eq:0712b}, and  Proposition \ref{r:proof of 1.2}, there exists a subsequence of the valuations $\{\Tn (s)\nu_{N,s}\}_{s}$ (also denoted by $\{\Tn (s)\nu_{N,s}\}_{s}$), which converges to a valuation $\nu_N$. Then we have
	$$\nu_N(f_j)\ge a_j, \quad j=0,\ldots,m,$$
	and
	$$\nu_N(F)=\sigma(\log|F|,\varphi_N).$$
	By the same reason, there exists a subsequence of the valuations  $\{\nu_N\}_N$ (also denoted by $\{\nu_{N}\}_{N}$), which converges to a valuation $\nu$. Then we have
	$\nu(f_j)\ge a_j$ for any $0\le j\le m$ and
	\[\sum_{0\le j\le m}\nu(f_j)= \nu(F)=\lim_{N\rightarrow+\infty}\sigma(\log|F|,\varphi_N)=\sigma(\log|F|,\varphi)=\sum_{1\le j\le m}a_j.\]
Combining the equality above and $a_0=0$, we have
	$$\nu(f_j)= a_j, \quad j=0,\ldots,m.$$
\end{proof}

We now prove Theorem  \ref{thm:interpolation weakly}.
\begin{proof}
The idea for the proof the Theorem  \ref{thm:interpolation weakly} is the same as the proof of Theorem \ref{thm:interpolation}.
We just give the different Lemmas or Propositions used in the two proof and omit the step by step proof of Theorem  \ref{thm:interpolation weakly}.

	We prove this theorem in two steps.
	
	\emph{Step 1. $(1)\Rightarrow(2)$.}
	
	Replace Lemma \ref{l:relativetype}  by Lemma \ref{l:relativetype weakly}. Using the same argument as  $(1)\Rightarrow(2)$ in the proof of Theorem \ref{thm:interpolation}, we know $(1)\Rightarrow(2)$ in Theorem  \ref{thm:interpolation weakly} holds.

	\emph{Step 2. $(2)\Rightarrow(1)$.}
	
Replace  \Cref{rem:max_existence} by  \Cref{rem:max_existence} and \Cref{rem:zhou weight for weak holomorphic function}, Proposition \ref{r:proof of 1.2} by Proposition \ref{r:proof of 1.2 weakly}. Using the same argument as  $(2)\Rightarrow(1)$ in the proof of Theorem \ref{thm:interpolation}, we know $(2)\Rightarrow(1)$ of Theorem  \ref{thm:interpolation weakly} holds.
\end{proof}

Now we prove Corollary \ref{C:interpolation-comp-poly}.
\begin{proof}[Proof of Corollary \ref{C:interpolation-comp-poly}]
    	If $\sigma(\log|F|,\varphi)=\sum_{1\le j\le m}a_j$, Theorem \ref{thm:interpolation} shows that  there exists a valuation $\nu$ on $\mathcal{O}_{X,o}=\mathcal{O}_{\mathbb{C}^n,o}/\big(I\cdot\mathcal{O}_{\mathbb{C}^n,o}\big)$ such that $\nu(f_j)=a_j$ for any $0\le j\le m$. Remark \ref{rem:subring complex case} tells $\mathbb{C}[z_1,\ldots,z_n]/I$ is a subring of $\mathcal{O}_{X,o}$. Thus $\nu$ is also a valuation on $\mathbb{C}[z_1,\ldots,z_n]/I$.

		Conversely, if there exists a valuation $\nu$ on $\mathbb{C}[z_1,\ldots,z_n]/I$ satisfying that $\nu(f_j)=a_j$ for any $0\le j\le m$ and $\nu(z_l)>0$ for any $1\le l\le n$, by Lemma \ref{l:relativetype2 alg} and the definition of $\sigma(\log|F|,\varphi)$, we have  $\sigma(\log|F|,\varphi)=\sum_{1\le j\le m}a_j$.
\end{proof}

\begin{proof}[Proof of Corollary \ref{C:interpolation-comp-poly2}]
	The sufficient part follows from Corollary \ref{C:interpolation-comp-poly}, then it suffices to prove the necessary part. If there exists a valuation $\nu$ on $\mathbb{C}[z_1,\ldots,z_n]/I$ satisfying that $\nu(f_0)=0$ and $\nu(f_j)=a_j>0$ for any $1\le j\le m$. Let $J$ be the ideal in $\mathbb{C}[z_1,\ldots,z_n]/I$ generated by $\{f_j\}_{1\le j\le m}$. 
	As $\cap_{1\le j \le m}\{f_j=0\}=\{o\}$, it follows from Hilbert's Nullstellensatz (see Lemma \ref{lem: Hilbert's Nullstellensatz theorem geo}) that  $z_k^N\in J$ for large $N$ and any $1\le k\le n$, which implies $\nu(z_k)>0$ for any $k$. Thus, we have $\sigma(\log|F|,\varphi)=\sum_{1\le j\le m}a_j$ by Corollary \ref{C:interpolation-comp-poly}.
\end{proof}

Now we prove Corollary \ref{C:interpolation-real}.
\begin{proof}[Proof of Corollary \ref{C:interpolation-real}]
    If $\sigma(\log|F|,\varphi)=\sum_{1\le j\le m}a_j$, Theorem \ref{thm:interpolation} shows that  there exists a valuation $\widetilde\nu$ on $\mathcal{O}_{X,o}=\mathcal{O}_{\mathbb{C}^n,o}/\big(\widetilde{P}(I)\big)$ such that $\widetilde\nu(\widetilde{P}(f_j))=a_j$ for any $0\le j\le m$. Remark \ref{rem:injective} tells $\widetilde{P}:C^{\text{an}}_{o'}/I\rightarrow\mathcal{O}_{\mathbb{C}^n,o}/\big(\widetilde{P}(I)\big)$ is an injective ring homomorphism. Then there exists a valuation $\nu$ on $C^{\text{an}}_{o'}/I$ such that $\nu(f_j)=a_j$ for any $j$.

		Conversely, if there exists a valuation $\nu$ on $C^{\text{an}}_{o'}/I$ such that $\nu(f_j)=a_j$ for any $j$, by Lemma \ref{c:relativetype-real} and the definition of $\sigma(\log|F|,\varphi)$, we have  $\sigma(\log|F|,\varphi)=\sum_{1\le j\le m}a_j$.
\end{proof}

Now we prove Corollary \ref{C:interpolation-real-poly}.
\begin{proof}[Proof of Corollary \ref{C:interpolation-real-poly}]
	If $\sigma(\log|F|,\varphi)=\sum_{1\le j\le m}a_j$, Corollary \ref{C:interpolation-comp-poly} shows that  there exists a valuation $\widetilde\nu$ on $\mathbb{C}[z_1,\ldots,z_n]/\big(P(I)\big)$ such that $\widetilde\nu(\widetilde{P}(f_j))=a_j$ for any $0\le j\le m$. Remark 
	\ref{rem:injective and decom poly case} tells $\widetilde{P}:\mathbb{R}[x_1,\ldots,x_n]/I \to \mathbb{C}[x_1,\ldots,x_n]/\big(P(I)\big)$ is an injective ring homomorphism. Then there exists a valuation $\nu$ on $\mathbb{R}[x_1,\ldots,x_n]/I$ such that $\nu(f_j)=a_j$ for any $j$.

	Conversely, if there exists a valuation $\nu$ on $\mathbb{R}[x_1,\ldots,x_n]/I$ satisfying that $\nu(f_j)=a_j$ for any $1\le j\le m$ and $\nu(x_l)>0$ for any $1\le l\le n$, by Lemma \ref{c:relativetype-real poly} and the definition of $\sigma(\log|F|,\varphi)$, we have  $\sigma(\log|F|,\varphi)=\sum_{1\le j\le m}a_j$.
\end{proof}

\vspace{.1in} {\em Acknowledgements}.
The first-named author completed part of this work during a visit to the School of Mathematical Sciences at Peking University and would like to thank the School for its hospitality and support. The second author was supported by National Key R\&D Program of China 2021YFA1003100 and NSFC-12425101. The third author was supported by NSFC-12401099 and the Talent Fund of Beijing Jiaotong University 2024-004. The fourth author was supported by NSFC-12501106.

\appendix \section{The universal denominators for algebraic varieties}
 It is known  (see \cite{demailly-book}; see also Theorem \ref{thm:universal denominators for complex space}) that for any complex analytic subvariety $X$ and $x\in X$, there exists locally a holomorphic function $\delta_{X,x}$ such that $\delta_{X,x}\mathcal{O}^{w}_{X,x}\subset \mathcal{O}_{X,x}$, where  $\mathcal{O}^{w}_{X,x}$ is the ring of germs of weakly holomorphic functions. Such $\delta_{X,x}$ is called universal denominator of weakly holomorphic functions.

Now we turn to the algebraic case.
Let $I$ be a prime ideal in $\mathbb{C}[z_1,\ldots,z_n]$. Let $X\coloneqq V(I)$ be  the affine variety defined by $I$ and $o\in X$ where $o$ is the origin in $\mathbb{C}^n$. Denote the germ of the set $X$ at $o$ by $(X,o)$ and assume that $(X,o)$ is irreducible as a germ of analytic set. In this appendix, we show that when $(X,o)$ is irreducible  as a germ of analytic set, one can choose the universal denominator to be a polynomial, see Theorem \ref{th:universal denominators alg appendix}.

We firstly recall some basic results. 

\begin{Lemma}[Noether's normalization lemma, see \cite{AtiyahMacdonald}]
\label{lem:Noether's normalization lemma}
Let $k$ be a field and let $A\neq 0$ be a finitely generated $k$-algebra which is generated by $x_1,\ldots,x_n$. Then there exist elements $y_1, \dots, y_r \in A$ which are algebraically independent over $k$ and such that $A$ is integral over $k[y_1, \dots, y_r]$. Moreover, one can choose $y_1,\ldots,y_r$ to be linear combinations of $x_1,\ldots,x_n$.
\end{Lemma}
\begin{proof}
For the convenience of the readers, we recall the proof of Lemma \ref{lem:Noether's normalization lemma} which can be referred to Section 5 of  \cite{AtiyahMacdonald}.

Since $k=\mathbb{C}$ in our case, we may assume that $k$ is infinite. 
As $x_1,x_2,\ldots,x_n$ generate $A$ as a $k$-algebra. We can renumber the $\{x_i\}_{i=1}^n$ such that $\{x_1,\ldots,x_r\}$ is a maximal algebraically independent system of $\{x_i\}_{i=1}^n$ over $k$. If $r=n$, there is nothing to prove. We may assume $r<n$. Now we prove Lemma \ref{lem:Noether's normalization lemma} by induction on $n$. Assume that Lemma \ref{lem:Noether's normalization lemma} holds when the finitely generated $k$-algebra has $n-1$ generators.

As $r<n$, then $x_n$ is algebraic over $k[x_1,\ldots,x_{n-1}]$, i.e., there exists a polynomial $f\in k[T_1,\ldots,T_n]$ such that $f(x_1,x_2,\ldots,x_n)=0$ in $A$. Decompose $f$ as a sum of homogeneous polynomials and let $F$ be the homogeneous part of the highest degree in the decomposition of $f$. Denote $\deg F=m$. As $k$ is infinite, we can find $(\lambda_1,\lambda_2,\ldots,\lambda_{n-1},1)\in k^n$ such that $F(\lambda_1,\lambda_2,\ldots,\lambda_{n-1},1)\neq 0$ in $k$. Let $w_{i}=x_i-\lambda_{i}x_n$ for $i=1,\ldots,n-1$. Then denote
$$\tilde{f}(T)=f(w_{1}+\lambda_1 T,w_{2}+\lambda_2 T,\ldots,w_{n-1}+\lambda_{n-1} T,T)\in k[w_1,\ldots,w_{n-1}][T].$$
We know that $\deg \tilde{f}=m$ and $\tilde{f}(x_n)=0$ in $A$. The coefficient of $T^m$ in $\tilde{f}$ is just $F(\lambda_1,\lambda_2,\ldots,\lambda_{n-1},1)\neq 0$. Then $\frac{1}{F(\lambda_1,\lambda_2,\ldots,\lambda_{n-1},1)}\tilde{f}$ is a unitary polynomial that annihilates $x_n$. Hence $x_n$ is integral over $A'=k[w_1,\ldots,w_{n-1}]$ which implies $A$ is integral over $A'=k[w_1,\ldots,w_{n-1}]$. Note that $w_1,\ldots,w_{n-1}$ can be written as linear
combinations of $x_1,\ldots,x_n$. Now applying the induction hypothesis to $A'=k[w_1,\ldots,w_{n-1}]$, we know that there exist $y_1, \dots, y_r \in A'$ which are algebraically independent over $k$ and such that $A'$ is integral over $k[y_1, \dots, y_r]$. Moreover, $y_1,\ldots,y_r$ can be chosen to be linear combinations of $w_1,\ldots,w_{n-1}$. By the transitivity of integral dependence (see \cite[Corollary 5.4]{AtiyahMacdonald}), we get that $A$ is integral over $k[y_1, \dots, y_r]$ and $y_1,\ldots,y_r$ can be chosen to be linear combinations of $x_1,\ldots,x_{n}$.
\end{proof}

\begin{Lemma}[Gauss lemma]
\label{lem:Gauss lemma}
Let $R$ be a unique factorization domain and $F$ its field of fractions. Any $f\neq 0$ in $R[x]$ is irreducible in $R[x]$ if and only if it is both irreducible in $F[x]$ and  primitive in $R[x]$.
\end{Lemma}

Using Noether's normalization lemma \ref{lem:Noether's normalization lemma}, after linear change of coordinate, we may assume that there exists an integer $d>0$ such that  $A\coloneqq \mathbb{C}[z_1,\ldots,z_n]/I$ is a finite integral extension of $\mathbb{C}[z_1,\ldots,z_d]$. 

Recall that $X=V(I)$ and  $o\in X$. 
Denote by $\tilde{f}$ the class of any $f\in \mathbb{C}[z_1,\ldots,z_n]$ in $A=\mathbb{C}[z_1,\ldots,z_n]/I$.
For any $d+1\le k \le n$, there exists a unitary polynomial $Q_k(z',T)\in \mathbb{C}[z_1,\ldots,z_d][T]$ such that $Q_k(z',\tilde{z}_k)=0$ in $A$ which means $Q_k(z',z_k)\in I$ as an element of $\mathbb{C}[z_1,\ldots,z_d][z_k]\subset \mathbb{C}[z_1,\ldots,z_n]$. 
 In a conclusion,
we have
\begin{Proposition}
\label{pro: existence of coordinate}
    There exist an integer $d$, a coordinates $(z_1, \ldots, z_n)$ with the following properties:  
$\mathcal{I}_d \coloneqq I\cap \mathbb{C}[z_1,\ldots,z_d]= \{0\}$
and for every integer $k = d + 1, \ldots, n$ there is a Weierstrass polynomial $Q_k \in \mathcal{I}_k\coloneqq  I\cap \mathbb{C}[z_1,\ldots,z_d][z_k] $ of the form  
$$Q_k(z', z_k) = z_k^{s_k} + \sum_{1 \leq j \leq s_k} a_{k,j}(z') z_k^{s_k - j}, \quad a_{j,k}(z') \in \mathbb{C}[z_1,\ldots,z_d].$$
\end{Proposition}

Let $\pi:\mathbb{C}^n\to \mathbb{C}^d$ be the mapping which sends $(z_1,\ldots,z_n)$ to $(z_1,\ldots,z_d)$.
\begin{Lemma}
\label{lem:projection is proper and surjective}
Let $\pi_X\colon X\to \mathbb{C}^d$ be the restriction of $\pi$ to $X$. Then $\pi_X$ is proper and surjective.
\end{Lemma}
\begin{proof}
Note that $\pi_X$ is continuous. To prove that $\pi_X$ is proper, it suffices to prove that, for any $K\subset \mathbb{C}^d$ a bounded subset, $\pi_X^{-1}(K)$ is bounded in $\mathbb{C}^n$. Note that $X\subset \cap_{k=d+1}^n \{Q_k=0\}$ and $Q_k\in \mathbb{C}[z_1,\ldots,z_d][z_k]$. If $z'\in \mathbb{C}^d$ is bounded, we know that the roots of each $Q_k$ are bounded. Hence $\pi_X^{-1}(K)$ is bounded in $\mathbb{C}^n$.

To prove that $\pi_X$ is surjective, without loss of generality, it suffices to prove that $o'\in \mathbb{C}^d$ belongs to $\pi_X(X)$. Denote by $m'=(z_1,\ldots,z_d)$ the maximal ideal in $\mathbb{C}_d\coloneqq\mathbb{C}[z_1,\ldots,z_d]$. 
It follows from $A$ is a finite integral extension of $\mathbb{C}_d$ and \cite[Corollary 5.2]{AtiyahMacdonald} that $\mathbb{C}_d$ is a subring of $A$ and $A$ is a finitely-generated $\mathbb{C}_d$-module.

We claim that $A/m'A\neq 0$, where $m'A$ is the ideal generated by $m'$ in $A$. If $A/m'A=0$, then $m'A=A$. It follows from $A$ is a finitely-generated $\mathbb{C}_d$-module and Nakayama's Lemma (cf. \cite[Corollary 2.5]{AtiyahMacdonald}) that there exits an $r\in \mathbb{C}_d$ such that $r\equiv 1$ (mod $m'$) in $\mathbb{C}_d$ and $rA=0$. View $r$ as a function on $\mathbb{C}^d$, then $r\equiv 1$ (mod $m'$) in $\mathbb{C}_d$ implies that $r(o')=1$ and hence $r\neq 0$ in $\mathbb{C}_d$. $rA=0$ implies $r\cdot 1=0$ in $A$ and thus $r\in I$. Then $r\in I\cap \mathbb{C}_d=\{0\}$. This contradicts to $r\neq 0$ in $\mathbb{C}_d$. Hence we must have $A/m'A\neq 0$.

Note that $A/m'A\cong \mathbb{C}[z_1,\ldots,z_n]/(I+m'\mathbb{C}[z_1,\ldots,z_n])$. 
It follows from $A/m'A\neq 0$ that $\mathbb{C}[z_1,\ldots,z_n]/(I+m'\mathbb{C}[z_1,\ldots,z_n])\neq 0$ and hence $I+m'\mathbb{C}[z_1,\ldots,z_n]\neq (1)$. Thus $V(I)\cap V(m'\mathbb{C}[z_1,\ldots,z_n])=V(I+m'\mathbb{C}[z_1,\ldots,z_n])$ is not empty. This implies that $o'\in \pi_X(X)$.
\end{proof}

Denote $\mathcal{O}_n$ and $\mathcal{O}_d$ be the germ of holmorphic functions at the origin of $\mathbb{C}^n$ and $\mathbb{C}^d$ respectively.
Using Lemma \ref{lem:projection is proper and surjective}, we have
\begin{Lemma}
\label{lem: same dimension}
$(I\cdot\mathcal{O}_n)\cap \mathcal{O}_d=0$.
\end{Lemma}
\begin{proof}
We prove Lemma \ref{lem: same dimension} by contradiction.
Assume there exists a nonzero holomorphic function $f\in (I\cdot\mathcal{O}_n)\cap \mathcal{O}_d$.
 It follows from $f\in (I\cdot\mathcal{O}_n)\cap \mathcal{O}_d$ that $\pi_X(X)\subset\{f=0\}$ is a  subset in $\mathbb{C}^d$. However Lemma \ref{lem:projection is proper and surjective} tells that $\pi_X$ is surjective, which is a contradiction.
\end{proof}

Denote by $\tilde{f}$ the class of any $f\in \mathbb{C}[z_1,\ldots,z_n]$ in $A=\mathbb{C}[z_1,\ldots,z_n]/I$. Denote by $\mathfrak{R}_A$ and $\mathfrak{R}_d$ the quotient fields of $A$ and $\mathbb{C}[z_1,\ldots,z_d]$ respectively. Then $\mathfrak{R}_A=\mathfrak{R}_d[\tilde{z}_{d+1},\ldots,\tilde{z}_{n}]$ is a finite algebraic extension of $\mathfrak{R}_d$.
Primitive element theorem tells that there exist $c_i\in \mathbb{C}$ ($i=d+1,\ldots,n$) such that $\tilde{u}=\sum_{i=d+1}^n c_i\tilde{z}_i$ and $\mathfrak{R}_A=\mathfrak{R}_d[\tilde{u}]$. Hence $\tilde{u}$ belongs to $A$. 
\begin{Lemma}
\label{lem:coefficient of minimal poly}
Let $R$ be a unique factorization domain and $F$ its field of fractions. Assume that $R\subset A$ and $A$ is a ring. For any $f\in A$ integral over $R[T]$, denote the minimal polynomial of $\tilde{f}$ over $R$ by $W_f(z',T)$. Then $W_f(z',T)$ is also the minimal polynomial of $\tilde{f}$ over $F$.
\end{Lemma}
\begin{proof}
    As  $R$ is UFD and $f\in A$ is integral over $R[T]$, we know that $R[T]$ is UFD and there exists a unitary irreducible polynomial $W_f(T)\in R[T]$  such that $W_f(f)=0$. Denote by $M_f(T)\in F[T]$ the minimal polynomial of $F$ over $F$. It follows from Gauss Lemma \ref{lem:Gauss lemma} that $W_f(z',T)$ is also irreducible in $F[T]$. Then by the definition of minimal polynomial and that $W_f(T)$ is irreducible in $F[T]$, we have $M_f(T)\mid W_f(T)$ and hence $M_f(T)=W_f(T)\in R[T]$.
\end{proof}

\begin{Remark}
\label{rem:rem for coefficient of minimal poly}
In this appendix, we will use Lemma \ref{lem:coefficient of minimal poly} in the following cases
\begin{enumerate}
    \item $R=\mathbb{C}[z_1,\ldots,z_d]$, $A=\mathbb{C}[z_1,\ldots,z_n]/I $ and $F=\mathfrak{R}_d$; 
    \item $R=\mathcal{O}_d$, $A=\mathcal{O}_n/(I\cdot\mathcal{O}_n)$ and $F=\mathfrak{M}_d$, where $\mathfrak{M}_d$ is the quotient fields of   $R=\mathcal{O}_d$.
\end{enumerate}
\end{Remark}

Denote the minimal polynomial of $\tilde{u}$ over $\mathfrak{R}_d$ by $W_u(z',T)$.
It follows from Lemma \ref{lem:coefficient of minimal poly} and Remark \ref{rem:rem for coefficient of minimal poly} that the discriminant $\delta(z_1,\ldots,z_d)$ of $W_u(z',T)$ belongs to $\mathbb{C}[z_1,\ldots,z_d]$. Hence we conclude that 
\begin{Lemma}
\label{lem:discriminant poly}
There exists a element $\tilde{u}\in A$ such that $\mathfrak{R}_A=\mathfrak{R}_d[\tilde{u}]$. Denote the minimal polynomial of $\tilde{u}$ over $\mathfrak{R}_d$ by $W_u(z',T)$. Then $W_u(z',T)\in \mathbb{C}[z_1,\ldots,z_d][T]$ and the discriminant $\delta(z_1,\ldots,z_d)$ of $W_u(z',T)$ belongs to $\mathbb{C}[z_1,\ldots,z_d]$. 
\end{Lemma}

Let $q\coloneqq [\mathfrak{R}_A:\mathfrak{R}_d]$ be the degree of the filed extension. As $\mathfrak{R}_A=\mathfrak{R}_d[\tilde{u}]$, we know that $W_{u}(z',T)$ is of degree $q$.

Denote by $\mathfrak{M}_A$ the quotient fields of $A=\mathcal{O}_n/(I\cdot\mathcal{O}_n)$ and by $\mathfrak{M}_d$ the  the quotient fields of   $R=\mathcal{O}_d$. It follows  from Proposition \ref{pro: existence of coordinate}, Lemma \ref{lem: same dimension}, Proposition 4.13 (and its proof) in \cite[p. 93]{demailly-book} and the arguments in \cite[pp. 93--94]{demailly-book} that we know
\begin{enumerate}
    \item $\mathcal{O}_n/I$ is a finite integral extension of $\mathcal{O}_d$;
    \item There exists a linear form $u_1=\sum_{d+1\le k \le n} c'_k z_k$ where $c'_k\in \mathbb{C}$ such that $\mathfrak{M}_A=\mathfrak{M}_d[u_1]$.
\end{enumerate}

 Recall that $\mathfrak{R}_A$ and $\mathfrak{R}_d$ are the rational function fields of $X$ and $\mathbb{C}_d$ respectively, which can be naturally viewed as a subset of $\mathfrak{M}_A$ and $\mathfrak{M}_d$ respectively. Note that $\mathfrak{R}_A=\mathfrak{R}_d[\tilde{u}]$.

\begin{Remark}
\label{rem:key remark}
We have $\mathfrak{M}_A=\mathfrak{M}_d[\tilde{u}]$. 
\end{Remark}
\begin{proof}
Since each $\tilde{z}_k$ belongs to $\mathfrak{R}_A$ and $\mathfrak{R}_A=\mathfrak{R}_d[\tilde{u}]$, we know that, for $d+1\le k\le n$, $\tilde{z}_k$ can be represented by the linear combination of $\tilde{u}^{j}$ ($0\le j \le q-1$) with coefficients in $\mathfrak{R}_d$. As $u_1=\sum_{d+1\le k \le n} c'_kz_k$ where $c'_k\in \mathbb{C}$, we know that $u_1$ can be also represented by the linear combination of $\tilde{u}^{j}$ ($0\le j \le q-1$) with coefficients in $\mathfrak{R}_d\subset \mathfrak{M}_d$. Thus we have $\mathfrak{M}_A=\mathfrak{M}_d[\tilde{u}]$.
\end{proof}

Denote by $H_u(z',T)\in \mathcal{O}_d[T] $ the minimal polynomial of $\tilde{u}\in \mathcal{O}_n/I$ over $\mathcal{O}_d$. It follows from Lemma \ref{lem:coefficient of minimal poly} and Remark \ref{rem:rem for coefficient of minimal poly} that we know $H_u(z',T)$ is also the minimal polynomial of $\tilde{u}$ over $\mathfrak{M}_d$. Recall that $W_u(z',T)\in \mathbb{C}[z_1,\ldots,z_d][T]\subset  \mathcal{O}_d[T]$ is the minimal polynomial of $\tilde{u}$ over $\mathbb{C}[z_1,\ldots,z_d]$ and $H_u(z',T)\in  \mathcal{O}_d[T]$ is the minimal polynomial of $\tilde{u}$ over $\mathcal{O}_d$. Denote the discriminants of $W_u(z',T)$ and $H_u(z',T)$ by $\delta_{\mathfrak{R}}(z')\in \mathbb{C}[z_1,\ldots,z_d]\subset \mathcal{O}_d$ and $\delta_{\mathfrak{M}}(z')\in \mathcal{O}_d$ respectively. We have the following relation between $\delta_{\mathfrak{R}}(z')$ and $\delta_{\mathfrak{M}}(z')$.
\begin{Lemma}
\label{lem:relation between discriminants}
$\delta_{\mathfrak{M}}(z')\mid\delta_{\mathfrak{R}}(z')$ in $\mathcal{O}_d$.
\end{Lemma}
\begin{proof}
    As $H_u(z',T)\in \mathcal{O}_d[T] $ is the minimal polynomial of $\tilde{u}\in \mathcal{O}_n/I$ over $\mathcal{O}_d$ and $W_u(z',T)\in \mathcal{O}_d[T]$ annihilates $\tilde{u}$, we have $H_u(z',T) \mid W_u(z',T)$ in $\mathcal{O}_d[T]$, i.e., there exists a polynomial $P_u(z',T)\in \mathcal{O}_d[T]$ such that $W_u(z',T)=H_u(z',T)P_u(z',T)$. Denote by $R_{H,P}$ the resultant of the polynomials $H$ and $P$, and denote by $\delta_p(z')$ the discriminant of $P_u(z',T)$. 
    
   Recall that the resultant of two polynomials is the  determinant of the Sylvester matrix of the two polynomials. Then it follows from $H_u(z',T),P_u(z',T)\in \mathcal{O}_d[T]$ that $R_{H,P}$ and $\delta_{P}(z')$ belong to $\mathcal{O}_d$. By the definition of discriminant, we know that $\delta_{\mathfrak{R}}(z')=\delta_{\mathfrak{M}}(z')\delta_{P}(z')R^{2}_{H,P}$. Thus, $\delta_{\mathfrak{M}}(z')\mid\delta_{\mathfrak{R}}(z')$ in $\mathcal{O}_d$.
\end{proof}

Now we recall the following existence of ``universal denominators" of weakly holomorphic functions on an irreducible complex variety $X$. 
\begin{Theorem}[see \cite{demailly-book}]
    \label{thm:universal denominators for complex space}
  For every point $x \in X$, there is a neighborhood $V$ of $x$ in $X$, such that $\delta_{\mathfrak{M},y} \mathcal{O}^{w}_{X,y} \subset \mathcal{O}_{X,y}$ for all $y \in V$, where  $\delta_{\mathfrak{M}}(z')$ is the discriminant of the minimal polynomial  $H_u(z',T) $  of $\tilde{u}$ over $\mathcal{O}_d$ and $\delta_{\mathfrak{M}}(z')$ is called a universal denominator on $V$.
\end{Theorem}

Using Lemma \ref{lem:relation between discriminants} and Theorem \ref{thm:universal denominators for complex space}, we can show that, the universal denominator for $\mathcal{O}^{w}_{X}$ can be chosen to be a polynomial.
\begin{Theorem}
\label{th:universal denominators alg appendix}
Assume that $X$ is irreducible and algebraic. Let $V$ be a neighborhood of $o$ in $X$ and $y\in V$. Then for any $f\in \mathcal{O}^{w}_{X,y}$, one has $\delta_{\mathfrak{R},y}(z')f\in \mathcal{O}_{X,y}$, where $\delta_{\mathfrak{R}}(z')\in \mathbb{C}[z_1,\ldots,z_d]$ is the discriminant of the minimal polynomial  $W_u(z',T) $  of $\tilde{u}$ over $\mathbb{C}[z_1,\ldots,z_d]$. 
\end{Theorem}
\begin{proof}
It follows from Theorem \ref{thm:universal denominators for complex space} that, for any $f\in\mathcal{O}^{w}_{X,y}$, one has $\delta_{\mathfrak{M},y} f \subset \mathcal{O}_{X,y}$. Lemma  \ref{lem:relation between discriminants} tells $\delta_{\mathfrak{R}}(z')=\delta_{\mathfrak{M}}(z')\delta_{P}(z')R_{H,P}$, where $\delta_{P}(z'),R_{H,P}\in \mathcal{O}_d$. Thus we know $\delta_{\mathfrak{R},y}(z')f=\delta_{P}(z')R_{H,P} \delta_{\mathfrak{M}}(z')f\in \mathcal{O}_{X,y}$.

Theorem \ref{th:universal denominators alg appendix} has been proved.
\end{proof}


\

\end{document}